\documentclass[14pt]{article}

\usepackage[english]{babel}
\usepackage[utf8]{inputenc}
\usepackage{amsmath,amsthm,amsfonts,amssymb,bm,tikz,subfig,float,sidecap,enumitem,varioref,standalone,yfonts}
\usepackage{MnSymbol}
\usepackage[left=1in, right=1in, top=1in]{geometry}
\usepackage{setspace}
\usepackage{multicol}
\usepackage{tikz}
\usetikzlibrary{decorations.markings}
\usepackage[nice]{nicefrac}
\usepackage{blkarray}
\usepackage{xcolor}
\usepackage{animate}
\usepackage{graphicx}
\usepackage{gensymb}
\usepackage[justification=centering]{caption}
\PassOptionsToPackage{demo}{graphicx}
\usepackage[colorinlistoftodos]{todonotes}
\usepackage[enableskew]{youngtab}

\newtheorem{theorem}{Theorem}
\newtheorem{lemma}{Lemma}
\newtheorem{ex}{Example}

\newtheorem{remark}{Remark}
\newtheorem{definition}{Definition}

\usepackage{graphicx}
\usepackage{picinpar,graphicx}
\usepackage{graphics}
\usepackage{multicol}
\usepackage{setspace}
\usepackage{color}
\usepackage{url} 
\usepackage{cleveref}
\usepackage{verbatim}
\usepackage{mathrsfs}
\usepackage{wrapfig}
\usepackage{afterpage}

\usepackage{lineno}
\usepackage{lipsum}

\newcommand{\C}{\mathbb{C}}
\newcommand{\N}{\mathbb{N}}

\newcommand{\gl}{\mathfrak{gl}}
\newcommand{\so}{\mathfrak{so}}
\newcommand{\g}{\mathfrak{g}}
\newcommand{\p}{\mathfrak{p}}
\newcommand{\oF}{\overline{F}}
\newcommand{\ind}{\mathrm{ind\;}}
\newcommand{\mf}{\mathfrak}
\usepackage[sort&compress]{natbib}

\title{\LARGE{{Regular functionals on seaweed Lie algebras}}}

\author{Vincent E. Coll, Jr.$^\dagger$ and Aria L. Dougherty$^{\dagger\dagger}$\footnote{The main results of this article were developed in the second author's 2019 Lehigh Univeristy Ph.D. thesis: ``Regular Functionals on Seaweed Lie Algebras."}}
\begin{document}


\maketitle

\noindent
\begin{center}
\textit{$^\dagger$Department of Mathematics, Lehigh University, Bethlehem, PA, USA: vec208@lehigh.edu }\\
\textit{$^{\dagger\dagger}$Department of Mathematics, Lehigh University, Bethlehem, PA, USA:  dr.ariadougherty@gmail.com}\\
\end{center}

\begin{abstract}
\noindent
The index of a Lie algebra $\mathfrak{g}$ is defined by ind $\mathfrak{g}=$ $\min_{f\in \mathfrak{g}^*}\dim(\ker (B_f))$, where $f$ is an element of the linear dual $\mathfrak{g}^*$ and $B_f(x,y)=f([x,y])$ is the associated skew-symmetric Kirillov form. We develop a broad general framework for the explicit construction of regular (index realizing) functionals for seaweed subalgebras of $\mathfrak{gl}(n)$ and the classical Lie algebras: $A_n=\mf{sl}(n+1),$ $B_n=\so(2n+1)$, and $C_n=\mf{sp}(2n)$.  Until now, this problem has remained open in $\mathfrak{gl}(n)$ -- and in all the classical types.
\end{abstract}




\tableofcontents
\section{Introduction}

The \textit{index} of a Lie algebra $\g$ is an important algebraic invariant which was first formally introduced by Dixmier (\textbf{[9]}, 1974). It is defined by 

\[
\ind \g=\min_{f\in \mathfrak{g}^*} \dim  (\ker (B_f)),
\]
where $f$ is an element of the linear dual $\mathfrak{g}^*$  and $B_f$ is the associated skew-symmetric \textit{Kirillov form} defined by 
\[
B_f(x,y)=f([x,y]), \textit{ for all }~ x,y\in\g. 
\]

Here, we focus on a class of matrix algebras called \textit{seaweed algebras}, or simply ``seaweeds".  These algebras, along with their evocative name,  were first introduced by Dergachev and A. Kirillov in (\textbf{[8]}, 2000), where they defined such algebras as subalgebras of $\mathfrak{gl}(n)$ preserving certain flags of subspaces 
developed from two compositions of $n$.  The passage to the classical seaweeds is accomplished by requiring that elements of the seaweed subalgebra of $\mathfrak{gl}(n)$
satisfy additional algebraic conditions. For example, the Type-$A$ case ($A_{n}=\mathfrak{sl}(n+1)$) is defined by a vanishing trace condition. There is a basis-free definition but we do not require it here.

On a given seaweed $\g$, index-realizing functionals are called \textit{regular} and exist in profusion, being dense in both the Zariski and Euclidean topolgies of $\g^*$ (see \textbf{[15]}).   Even so, methods for explicitly constructing regular functionals are few.  

One such method is due indirectly to Kostant.  In 1960, and after the fashion of the Gram-Schmidt orthogonalization process, Kostant developed an algorithm, called a \textit{cascade}\footnote{See (\textbf{[12]}, 1976) for an early description of the cascade by Anthony Joseph.  This paper cites Bertram Kostant and Jacques Tits as discovering this process independently;  Tits is cited as (\textbf{[19]}, 1960), and Kostant is cited as private communication with no specified year.  For a more recent paper on the cascade by Kostant, see (\textbf{[14]}, 2012).}, that produces a set of strongly orthogonal roots from a root system which defines a Lie algebra. In 2004,  Tauvel and Yu \textbf{[18]} noted that, in many cases, and as a by-product of this process, a regular functional could be constructed using representative elements in $\g$ of the root spaces for the highest roots generated by the cascade.  
For \textit{Frobenius} (index zero) seaweeds, the cascade will always produce a regular, or Frobenius, functional (see \textbf{[13]}, p. 19), but in the non-Frobenius case the cascade will often fail.


Beyond the Kostant cascade, and prior to the work here, the authors were not aware of an algorithmic procedure that would produce a regular functional in the generic case for seaweeds of classical type.  Indeed, Joseph has noted that the resolution of this gap is a significant open problem in even the Type-$A$ case (see \textbf{[13]}, p. 774).

Why the cascade sometimes fails to produce a regular functional, or rather for what seaweeds it fails, is the starting point for our investigation. This research is detailed in \textbf{[3]} and \textbf{[10]}, where it becomes apparent that the success or failure of the cascade depends on the homotopy type of the seaweed, a component structure of a certain planar graph called a \textit{meander}. Meanders were introduced in \textbf{[8]}
by Dergachev and A. Kirillov, where they showed that the index of a seaweed subalgebra of $\mathfrak{gl}(n)$ -- and by an easy extension,
$\mathfrak{sl}(n)$ -- could be computed by an elementary combinatorial formula based on the number and type of the connected components of the meander associated with the seaweed.  Even so, significant computational complexity persists. This complexity can be can be mollified by ``winding-down" the meander through a sequence of deterministic graph-theoretic moves (``winding-down moves") which yields the meander's essential configuration, which we call the meander's \textit{homotopy type} (see \textbf{[6]} and \textbf{[4]}). This winding down procedure may be regarded as a graph theoretic rendering of Panychev's well-known reduction algorithm (see \textbf{[16]}). (Reversing the winding-down moves yields ``winding-up" moves from which any meander (and so any seaweed), of any size or configuration, can be constructed.)\footnote{In \textbf{[5]}, Coll et al. extended this formulaic construction to the Type-$C$ and Type-$B$ cases (cf. \textbf{[17]}).  More recently, Cameron (in his 2019 Ph.D. thesis at Lehigh University \textbf{[2]}) has extended these results to the Type-$D$ case, thus completing the combinatorial classification of seaweeds in the classical types (cf.\textbf{[11]}).}

The strategy for producing a regular functional on a seaweed $\mathfrak{g}$ in the classical types considered here is to first develop an explicit regular functional $F_n$ on $\mathfrak{gl}(n)$ -- see Section 3.  One then uses $\mathfrak{g}$'s meander to show how the meander's components identify a ``configuration" of admissible positions in $\mathfrak{g}$.  Each of these configurations contains certain square matrix blocks $C_{c\times c}$, called \textit{core blocks of the configuration}. The union of all core blocks over all configurations constitutes $\textgoth{C}$ - the \textit{core of the seaweed} $\mathfrak{g}$.  To build a regular functional on $\g$, one inserts a copy of $F_c$ into $C_{c\times c}$, for all $C$ in $\textgoth{C}$, and zeros elsewhere in the admissible locations of $\g$ -- with some exceptions based on aspects of the seaweed's ``shape". The regularity of the adjusted functional is established by Theorem \ref{Functional Construction}.  This not only resolves the open problem for seaweed subalgebras of $\gl(n)$, but, in conjunction with Theorem \ref{My Theorem 1}, more broadly delivers an algorithmic procedure for the construction of regular functionals on seaweeds in the classical types $A$, $B$, and $C$ -- see section 4.  To ease exposition, we consider types $A$ and $C$ first, and then conclude with the type $B$ case.



\section{Preliminaries}
\label{Preliminaries}

\begin{subsection}{Seaweed subalgebras of $\gl(n)$}
\label{Seaweed Subalgebras}
\noindent
\textit{Notation:}  All Lie algebras $\g$ are finite-dimensional over the complex numbers, and by Ado's theorem (see \textbf{[1]}) are therefore assumed to be subalgebras of $\mathfrak{gl}(n)$. We assume that $\g$ comes equipped with a triangular decomposition $\g=\mathfrak{u}_+\oplus\mathfrak{h} \oplus\mathfrak{u}_-,$ where $\mathfrak{h}$ is a Cartan subalgebra of $\g$ and $\mathfrak{u}_+$ and $\mathfrak{u}_-$ consist of the upper and lower triangular matrices, respectively. Any seaweed is conjugate, over its algebraic group, to a 
seaweed in this \textit{standard} form.  A basis-free definition (due to Panyushev \textbf{[16]}) reckons seaweed subalgebras of a simple Lie algebra $\mathfrak{g}$ as the intersection of two parabolic algebras whose sum is $\mathfrak{g}$.  For this reason, Joseph has elsewhere \textbf{[13]} called these algebras \textit{biparabolic}.  We do not require the latter definition in our work here.

\bigskip 

The basic objects of our study are the evocatively named ``seaweed" Lie algebras first introduced by Dergachev and A.  Kirillov in \textbf{[8]} and defined as follows.

\begin{definition}
\label{seaweeddef}
Let $(a_1,\cdots,a_m)$ and $(b_1,\cdots,b_t)$ be two compositions of $n$, and let $\{0\}=V_0\subset V_1\subset\cdots\subset V_m=\C^n$, and $\C^n=W_0\supset W_1\supset\cdots\supset W_t=\{0\},$ where $V_i=\text{span} \{e_1,\cdots,e_{a_1+\cdots+a_i}\}$ and $W_j=\text{span}\{e_{b_1+\cdots+b_j+1},\cdots,e_n\}$. The {standard} {seaweed} $\g$ of \textit{{type}} $\frac{a_1|\cdots|a_m}{b_1|\cdots|b_t}$ is the subalgebra of $\gl(n)$ which preserves the spaces $V_i$ and $W_j$.
\end{definition}


To each seaweed of type $\frac{a_1|\cdots|a_m}{b_1|\cdots|b_t}$ we  associate a planar graph called a \textit{{meander}}, constructed as follows.  First, place $n$ vertices $v_1$ through $v_n$ in a horizontal line.  Next, create two partitions of the vertices by forming \textit{{top}} and {\textit{bottom blocks}} of vertices of size $a_1$, $a_2$, $\cdots$, $a_m$, and $b_1$, $b_2$, $\cdots$, $b_t$, respectively.  Place edges in each top (likewise bottom) block in the same way. Add an edge from the first vertex of the block to the last vertex of the same block.  Repeat this edge addition on the second vertex and the second to last vertex within the same block and so on within each block of both partitions. Top edges are drawn concave down and bottom edges are drawn concave up. We say that the meander is of \textit{type} $\frac{a_1|\cdots|a_m}{b_1|\cdots|b_t}$ (see Example \ref{MeanderEx}). To any seaweed $\g$, denote the meander associated with $\g$ by $M(\g)$.

\begin{ex} 
\label{MeanderEx}
Consider $\g$ of type $\frac{4|1}{2|1|2}$. The meander $M(\g)$ is illustrated in Figure \ref{MeanderinSeaweed} \textup(left\textup).
\end{ex}

A meander can be visualized inside its associated seaweed $\g$ if one views the diagonal entries $e_{i,i}$ of $\g$ as the $n$ vertices $v_i$ of the meander and reckons the top edges $(v_i,v_j)$ with $i<j$ of the meander as the unions of line segments connecting the matrix locations $(i,i)\rightarrow(j,i)\rightarrow(j,j)$ and the bottom edges $(v_i,v_j)$ with $i<j$ of the meander as the unions of line segments connecting the matrix locations $(i,i)\rightarrow(i,j)\rightarrow(j,j)$. See Figure \ref{MeanderinSeaweed} \textup(right\textup), where the asterisks represent possible nonzero entries from $\C$; blank locations are forced zeroes.

\vspace{-1em}

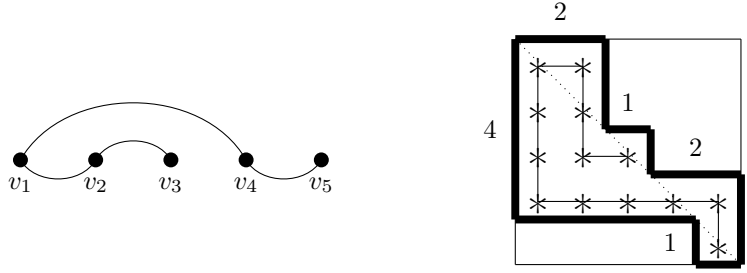
\begin{figure}[H]
$$\begin{tikzpicture}[scale=1]
\def\Node{\node [circle, fill, inner sep=2pt]}
\Node [label=below:$v_1$] at (0,0){};
\Node [label=below:$v_2$] at (1,0){};
\Node [label=below:$v_3$] at (2,0){};
\Node [label=below:$v_4$] at (3,0){};
\Node [label=below:$v_5$] at (4,0){};
\draw (0,0) to[bend left=60](3,0);
\draw (1,0) to[bend left=60](2,0);
\draw (0,0) to[bend right=60](1,0);
\draw (3,0) to[bend right=60](4,0);
\node at (0,-1.5){};
\end{tikzpicture}
\hspace{5em}
\begin{tikzpicture}[scale=0.6]
\def\Node{\node [circle, fill, inner sep=3pt]}
	\draw (0,0) -- (0,5);
	\draw (0,5) -- (5,5);
	\draw (5,5) -- (5,0);
	\draw (5,0) -- (0,0);
	\draw [line width=3](0,5) -- (2,5);
	\draw [line width=3](2,5) -- (2,3);
	\draw [line width=3](2,3) -- (3,3);
	\draw [line width=3](3,3) -- (3,2);
	\draw [line width=3](3,2) -- (5,2);
	\draw [line width=3](5,2) -- (5,0);
	\draw [line width=3](0,5) -- (0,1);
	\draw [line width=3](0,1) -- (4,1);
	\draw [line width=3](4,1) -- (4,0);
	\draw [line width=3](4,0) -- (5,0);
	\draw [dotted] (0,5) -- (5,0);
	\node at (0.5,4.2) {{\LARGE *}};
	\node at (1.5,4.2) {{\LARGE *}};
	\node at (0.5,3.2) {{\LARGE *}};
	\node at (1.5,3.2) {{\LARGE *}};
	\node at (0.5,2.2) {{\LARGE *}};
	\node at (1.5,2.2) {{\LARGE *}};
	\node at (2.5,2.2) {{\LARGE *}};
	\node at (0.5,1.2) {{\LARGE *}};
	\node at (1.5,1.2) {{\LARGE *}};
	\node at (2.5,1.2) {{\LARGE *}};
	\node at (3.5,1.2) {{\LARGE *}};
	\node at (4.5,1.2) {{\LARGE *}};
	\node at (4.5,0.2) {{\LARGE *}};
	\draw (2.5,2.4)--(1.5,2.4)--(1.5,4.4)--(0.5,4.4)--(0.5,1.4)--(4.5,1.4)--(4.5,0.4);
	\node [label=left:{4}] at (0,3){};
	\node [label=left:{1}] at (4,0.5){};
	\node [label=above:{2}] at (1,5) {};
	\node [label=above:{1}] at (2.5,3) {};
	\node [label=above:{2}] at (4,2) {};
\end{tikzpicture}$$

\vspace{-1em}

\caption{Meander of type $\frac{4|1}{2|1|2}$ \textup(left\textup) visualized in its seaweed \textup(right\textup)}
\label{MeanderinSeaweed}
\end{figure}
\end{subsection}

\begin{subsection}{The index and homotopy type of a seaweed algebra}
\label{Index}
The \textit{index} of a Lie algebra $\g$ is defined by 

$$\ind\g=\underset{f\in\g^*}{\min}\dim\ker B_f.$$

Using the meander associated with a Lie algebra, Dergachev and A. Kirillov provide a combinatorial formula for the index of $\g$ in terms of the number and type of the meander's connected components.

\begin{theorem}[Dergachev and A.  Kirillov \textbf{[8]}]
\label{Comb Formula}
If $\g$ is a seaweed subalgebra of $\mathfrak{gl}(n)$, and $M(\g)$ is its associated meander, then
$$\text{ind }\g= 2C+P,$$
where $C$ is the number of cycles and $P$ is the number of paths in $M(\g)$.
\end{theorem}

We have the following immediate Corollary.

\begin{theorem}
\label{Index Cor}
The Lie algebra $\gl(n)$ has index $n$.
\end{theorem}

Any meander can be contracted, or ``wound down," to the empty meander through a sequence of graph-theoretic moves,  each of which is uniquely determined by the structure of the meander at the time of the move application.

Make note of Panychev.... 
\begin{lemma}[Coll, Hyatt, and Magnant \textbf{[6]}]
\label{winding down}
Let $\g$ be a seaweed of type $\frac{a_1|\cdots|a_m}{b_1|\cdots|b_t}$ with associated meander $M(\g)$.  Create a meander $M'$ by one of the following moves.
\begin{enumerate}
	\item {Block Elimination  \textup(Bl\textup):} If $a_1=2b_1$, then $M(\g)\mapsto M'$ of type $\frac{b_1|a_2|\cdots|a_m}{b_2|b_3|\cdots|b_t}$.
    \item {Rotation Contraction \textup(R\textup):} If $b_1<a_1<2b_1$, then $M(\g)\mapsto M'$ of type $\frac{b_1|a_2|\cdots|a_m}{(2b_1-a_1)|b_2|\cdots|b_t}$.
    \item {Pure Contraction \textup(P\textup):} If $a_1>2b_1$, then $M(\g)\mapsto M'$ of type $\frac{(a_1-2b_1)|b_1|a_2|\cdots|a_m}{b_2|b_3|\cdots|b_t}$.
    \item {Flip \textup(F\textup):} If $a_1<b_1$, then $M(\g)\mapsto M'$ of type $\frac{b_1|b_2|\cdots|b_t}{a_1|\cdots|a_m}$.
    \item {Component Deletion \textup(C\textup(c\textup)\textup):} If $a_1=b_1=c$, then $M(\g)\mapsto M'$ of type $\frac{a_2|\cdots|a_m}{b_2|\cdots|b_t}$.
\end{enumerate}
  These moves are called \textit{{winding-down moves}}.  For all moves, except the Component Deletion move, $\g$ and $\g'$ \textup(the seaweed with meander $M(\g')=M'$\textup) have the same index. 
\end{lemma}

Given a meander $M(\g)$, there exists a unique sequence of moves (elements of the set $\{Bl,R,P,F,C(c)\})$
which reduces $M(\g)$ to the empty meander. This sequence is called the {\textit{signature}} of $M(\g)$.
If $C(c_1),\cdots,C(c_h)$ are the component deletion moves which appear (in order) in the signature of $M(\g)$, then $M(\g)$'s {\textit{homotopy type}}, denoted $H(c_1,\cdots,c_h)$, is the meander of type $\frac{c_1|\cdots|c_h}{c_1|\cdots|c_h}$. The individual meanders of type $\frac{c_i}{c_i}$ for each $i$ are the \textit{{components}} of the homotopy type, with the numbers $c_i$ referred to as the {\textit{sizes}} of the respective components.

\begin{ex}
\label{unwinding ex}
Let $M(\g)$ be the meander of type $\frac{17|3}{10|4|6}$. By repeated applications of Lemma \ref{winding down}, $M(\g)$ has signature $RPC(4)FBC(3)$. The unwinding of $M(\g)$ is demonstrated in Figure \ref{unwinding figure}, and the  homotopy type of $M(\g)$ is $H(4,3)$.  See Figure \ref{homotopy figure}, where the meanders ``essential configuration" is illustrated graphically.
\begin{figure}[H]
$$\begin{tikzpicture}[scale=.25]
\def\Node{\node [circle, fill, inner sep=1.5pt]}
\node at (10.5,-3){$\frac{17|3}{10|4|6}$};
\Node (1) at (1,0){};
\Node (2) at (2,0){};
\Node (3) at (3,0){};
\Node (4) at (4,0){};
\Node (5) at (5,0){};
\Node (6) at (6,0){};
\Node (7) at (7,0){};
\Node (8) at (8,0){};
\Node (9) at (9,0){};
\Node (10) at (10,0){};
\Node (11) at (11,0){};
\Node (12) at (12,0){};
\Node (13) at (13,0){};
\Node (14) at (14,0){};
\Node (15) at (15,0){};
\Node (16) at (16,0){};
\Node (17) at (17,0){};
\Node (18) at (18,0){};
\Node (19) at (19,0){};
\Node (20) at (20,0){};
\draw (1) to[bend left=60] (17);
\draw (2) to[bend left=60] (16);
\draw (3) to[bend left=60] (15);
\draw (4) to[bend left=60] (14);
\draw (5) to[bend left=60] (13);
\draw (6) to[bend left=60] (12);
\draw (7) to[bend left=60] (11);
\draw (8) to[bend left=60] (10);
\draw (18) to[bend left=60] (20);
\draw (1) to[bend right=60] (10);
\draw (2) to[bend right=60] (9);
\draw (3) to[bend right=60] (8);
\draw (4) to[bend right=60] (7);
\draw (5) to[bend right=60] (6);
\draw (11) to[bend right=60] (14);
\draw (12) to[bend right=60] (13);
\draw (15) to[bend right=60] (20);
\draw (16) to[bend right=60] (19);
\draw (17) to[bend right=60] (18);
\end{tikzpicture}
\hspace{1em}
\begin{tikzpicture}[scale=.25]
\def\Node{\node [circle, fill, inner sep=1.5pt]}
\node at (-2,0){$\overset{R}{\mapsto}$};
\node at (6,-3){$\frac{10|3}{3|4|6}$};
\Node (1) at (1,0){};
\Node (2) at (2,0){};
\Node (3) at (3,0){};
\Node (4) at (4,0){};
\Node (5) at (5,0){};
\Node (6) at (6,0){};
\Node (7) at (7,0){};
\Node (8) at (8,0){};
\Node (9) at (9,0){};
\Node (10) at (10,0){};
\Node (11) at (11,0){};
\Node (12) at (12,0){};
\Node (13) at (13,0){};
\draw (1) to[bend left=60] (10);
\draw (2) to[bend left=60] (9);
\draw (3) to[bend left=60](8);
\draw (4) to[bend left=60](7);
\draw (5) to[bend left=60](6);
\draw (11) to[bend left=60](13);
\draw (1) to[bend right=60](3);
\draw (4) to[bend right=60](7);
\draw (5) to[bend right=60](6);
\draw (8) to[bend right=60](13);
\draw (9) to[bend right=60](12);
\draw (10) to[bend right=60](11);
\end{tikzpicture}
\hspace{1em}
\begin{tikzpicture}[scale=.25]
\def\Node{\node [circle, fill, inner sep=1.5pt]}
\node at (-1,0){$\overset{P}{\mapsto}$};
\node at (6,-3){$\frac{4|3|3}{4|6}$};
\Node (1) at (1,0){};
\Node (2) at (2,0){};
\Node (3) at (3,0){};
\Node (4) at (4,0){};
\Node (5) at (5,0){};
\Node (6) at (6,0){};
\Node (7) at (7,0){};
\Node (8) at (8,0){};
\Node (9) at (9,0){};
\Node (10) at (10,0){};
\draw (1) to[bend left=60] (4);
\draw (2) to[bend left=60](3);
\draw (5) to[bend left=60](7);
\draw (8) to[bend left=60](10);
\draw (1) to[bend right=60](4);
\draw (2) to[bend right=60](3);
\draw (5) to[bend right=60](10);
\draw (6) to[bend right=60](9);
\draw (7) to[bend right=60](8);
\end{tikzpicture}$$
$$\begin{tikzpicture}[scale=.25]
\def\Node{\node [circle, fill, inner sep=1.5pt]}
\node at (-1.5,0){$\overset{C(4)}{\mapsto}$};
\node at (3.5,-3){$\frac{3|3}{6}$};
\Node (1) at (1,0){};
\Node (2) at (2,0){};
\Node (3) at (3,0){};
\Node (4) at (4,0){};
\Node (5) at (5,0){};
\Node (6) at (6,0){};
\draw (1) to[bend left=60](3);
\draw (4) to[bend left=60](6);
\draw (1) to[bend right=60](6);
\draw (2) to[bend right=60](5);
\draw (3) to[bend right=60](4);
\end{tikzpicture}
\hspace{.2em}
\begin{tikzpicture}[scale=.25]
\def\Node{\node [circle, fill, inner sep=1.5pt]}
\node at (-1,0){$\overset{F}{\mapsto}$};
\node at (3.5,-3){$\frac{6}{3|3}$};
\Node (1) at (1,0){};
\Node (2) at (2,0){};
\Node (3) at (3,0){};
\Node (4) at (4,0){};
\Node (5) at (5,0){};
\Node (6) at (6,0){};
\draw (1) to[bend right=60](3);
\draw (4) to[bend right=60](6);
\draw (1) to[bend left=60](6);
\draw (2) to[bend left=60](5);
\draw (3) to[bend left=60](4);
\end{tikzpicture}
\hspace{.2em}
\begin{tikzpicture}[scale=.25]
\def\Node{\node [circle, fill, inner sep=1.5pt]}
\node at (-1,0){$\overset{B}{\mapsto}$};
\node at (2,-3){$\frac{3}{3}$};
\Node (1) at (1,0){};
\Node (2) at (2,0){};
\Node (3) at (3,0){};
\draw (1) to[bend right=60](3);
\draw (1) to[bend left=60](3);
\end{tikzpicture}
\hspace{.3em}
\begin{tikzpicture}[scale=.3]
\node at (-1,-1){$\overset{C(3)}{\mapsto}$};
\node at (1,-1){$\emptyset$};
\node at (2,-3.75){\color{white}$\frac{0}{0}$};
\end{tikzpicture}$$
\caption {Winding down the meander of type $\frac{17|3}{10|4|6}$}
\label{unwinding figure}
\end{figure}
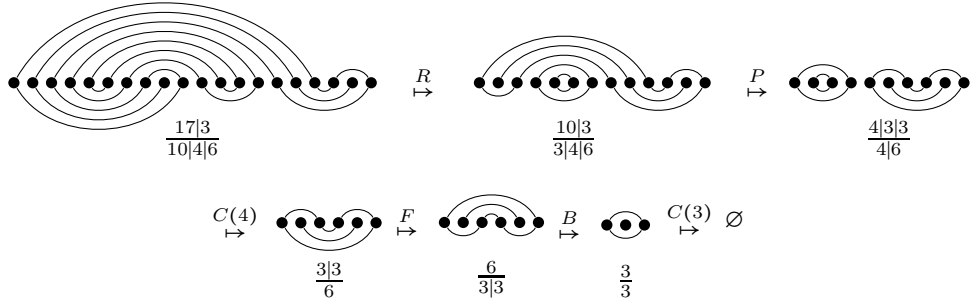
\noindent 
\vspace{-3em}
\begin{figure}[H]
$$\begin{tikzpicture}[scale=2]
\def\Node{\node [circle, fill, inner sep=1.5pt]}
	\draw (0,0) node[draw,circle,fill=white,minimum size=40pt,inner sep=0pt] (2+) {};
	\draw (0,0) node[draw,circle,fill=white,minimum size=20pt,inner sep=0pt] (2+) {};
    
	\draw (1,0) node[draw,circle,fill=white,minimum size=30,inner sep=0pt] (2+) {};
	\draw (1,0) node[draw,circle,fill=black,minimum size=4pt,inner sep=0pt] (2+) {};
\end{tikzpicture}$$
\caption{The homotopy type of $\frac{17|3}{10|4|6}$ is $H(4,3)$.}
\label{homotopy figure}
\end{figure}
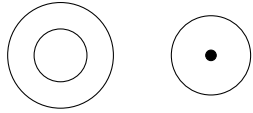
\end{ex}

\begin{theorem}
\label{Index and Homotopy}
If $\g$ is a seaweed with homotopy type $H(c_1,\cdots,c_h)$, then $\ind(\g)=\sum_{i=1}^h c_i.$
\end{theorem}

In what follows we will find it useful to 
define another meander associated with a seaweed.

\begin{definition}
Given a seaweed $\g$ with signature $S$, the {component meander} $CM(\g)$ associated with $\g$ is the meander with the same signature as $\g$ except that the component deletions are all of size one.
\end{definition}

\begin{ex}
\label{compmeander}
Consider $\g$ of type $\frac{10|2|4}{16}$. The signature of $\g$ is $FRPC(2)C(4)$, and $\g$ has homotopy type $H(2,4)$. The component meander of $\g$ has signature $FRPC(1)C(1)$. By reversing the winding-down moves (see Figure \ref{cmwu}, where the path which results from the component of size four is red and the path which results from the component of size two is blue), we construct the component meander of $\mathfrak{g}$. 

\begin{figure}[H]
$$\begin{tikzpicture}[scale=0.5]
    \def\Node{\node [circle, fill, inner sep=1.5pt]}
    \node at (0,0){$\emptyset$};
    \node at (1,0.2){$\overset{C(1)}{\mapsto}$};
    \Node [red] at (2,0){};
    \node at (3,0.2){$\overset{C(1)}{\mapsto}$};
    \Node [blue] at (4,0){};
    \Node [red] at (5,0){};
    \node at (6,0.2){$\overset{P}{\mapsto}$};
    \Node (x) [red] at (7,0){};
    \Node [blue] at (8,0){};
    \Node (y) [red] at (9,0){};
    \draw [line width=0.45mm, red] (x) to[bend left=60] (y);
    
    \node at (10,0.2){$\overset{R}{\mapsto}$};
    \Node (1) [red] at (11,0){};
    \Node (A) [blue] at (12,0){};
    \Node (3) [red] at (13,0){};
    \Node (B) [blue] at (14,0){};
    \Node (2) [red] at (15,0){};
    \draw [line width=0.45mm, red] (1) to[bend left=60] (2);
    \draw [line width=0.45mm, red] (1) to[bend right=60] (3);
    \draw [line width=0.45mm, blue] (A) to[bend left=60] (B);
    \node at (16,0.2){$\overset{F}{\mapsto}$};  
    \Node (1) [red] at (17,0){};
    \Node (A) [blue] at (18,0){};
    \Node (3) [red] at (19,0){};
    \Node (B) [blue] at (20,0){};
    \Node (2) [red] at (21,0){};  
    \draw [line width=0.45mm, red] (1) to[bend right=60] (2);
    \draw [line width=0.45mm, red] (1) to[bend left=60] (3);
    \draw [line width=0.45mm, blue] (A) to[bend right=60] (B);
\end{tikzpicture}$$
\caption{Winding-up of the component meander for $\g$ of type $\frac{10|2|4}{16}$}
\label{cmwu}
\end{figure}
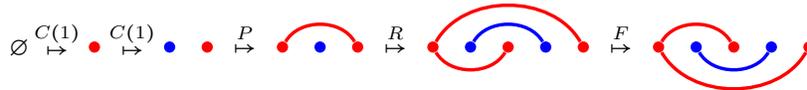

See Figure \ref{A} \textup(left\textup) which illustrates $M(\g)$ with the component of size two in blue and the component of size four in red. Figure \ref{A} \textup(right\textup) illustrates the $CM(\g)$ with the resulting paths from the components of size four and two in red and blue, respectively.
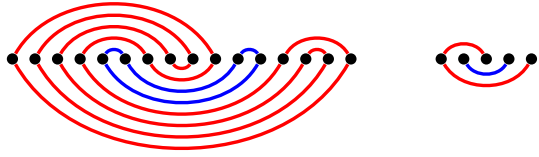
\begin{figure}[H]
$$\begin{tikzpicture}[scale=0.3]
    \def\Node{\node [circle, fill, inner sep=1.5pt]}
    \Node (1) at (1,0){};
    \Node (2) at (2,0){};
    \Node (3) at (3,0){};
    \Node (4) at (4,0){};
    \Node (5) at (5,0){};
    \Node (6) at (6,0){};
    \Node (7) at (7,0){};
    \Node (8) at (8,0){};
    \Node (9) at (9,0){};
    \Node (10) at (10,0){};
    \Node (11) at (11,0){};
    \Node (12) at (12,0){};
    \Node (13) at (13,0){};
    \Node (14) at (14,0){};
    \Node (15) at (15,0){};
    \Node (16) at (16,0){};
    
    \draw [red, line width=0.45mm] (1) to[bend left=60] (10);
    \draw [red, line width=0.45mm] (2) to[bend left=60] (9);
    \draw [red, line width=0.45mm] (3) to[bend left=60] (8);
    \draw [red, line width=0.45mm] (4) to[bend left=60] (7);
    \draw [blue, line width=0.45mm] (5) to[bend left=60] (6);
    \draw [blue, line width=0.45mm] (11) to[bend left=60] (12);
    \draw [red, line width=0.45mm] (13) to[bend left=60] (16);
    \draw [red, line width=0.45mm] (14) to[bend left=60] (15);
    \draw [red, line width=0.45mm] (1) to[bend right=60] (16);
    \draw [red, line width=0.45mm] (2) to[bend right=60] (15);
    \draw [red, line width=0.45mm] (3) to[bend right=60] (14);
    \draw [red, line width=0.45mm] (4) to[bend right=60] (13);
    \draw [blue, line width=0.45mm] (5) to[bend right=60] (12);
    \draw [blue, line width=0.45mm] (6) to[bend right=60] (11);
    \draw [red, line width=0.45mm] (7) to[bend right=60] (10);
    \draw [red, line width=0.45mm] (8) to[bend right=60] (9);

    \Node (A) at (20,0){};
    \Node (B) at (21,0){};
    \Node (C) at (22,0){};
    \Node (D) at (23,0){};
    \Node (E) at (24,0){};
    
    \draw [red, line width=0.45mm] (A) to[bend left=60] (C);
    \draw [blue, line width=0.45mm] (B) to[bend right=60] (D);
    \draw [red, line width=0.45mm] (A) to[bend right=60](E);
\end{tikzpicture}$$
\vspace{-2em}
\caption{$M(\g)$ and $CM(\g)$, where $\g$ of type $\frac{10|2|4}{16}$}
\label{A}
\end{figure}
\end{ex}

The vertices of $CM(\g)$ are $v_{A_1},\cdots,v_{A_t}$, where $A_i$ is the set of indices for the adjacent vertices that were merged into one vertex in $CM(\g)$ from $M(\g)$. The size of the subscript for $v_{A_i}$ is equal to $c_j$ for its corresponding component in the homotopy type of $\g$.

\begin{ex}
    Consider $\g$ of Example \ref{compmeander}. The vertex labels for $CM(\g)$ are $v_{\{1,2,3,4\}}$, $v_{\{5,6\}}$, $v_{\{7,8,9,10\}}$, $v_{\{11,12\}}$, and $v_{\{13,14,15,16\}}.$
\end{ex}
\end{subsection}

\begin{subsection}{Distinguished subsets of a seaweed algebra}

In this section, we highlight several important subsets of a seaweed. These subsets are defined by configurations of positions cut out of the seaweed by the components of the meander associated with the seaweed. 

\begin{definition}
\label{res}
Let $\g$ be a seaweed such that $M(\g)$ has homotopy type $H(c_1,\cdots,c_h)$. Let $\mathscr{I}_{c_i}$ represent the index set of the component $c_i$ in $M(\g)$. In other words, visualize the meander in the matrix form of $\g$ \textup(see Example \ref{restriction}\textup); $\mathscr{I}_{c_i}$ consists of each index $(j,k)$ ``covered" by an edge in the component $c_i$ of $M(\g)$. Denote by $\g|_{c_i}$ the {\textit{configuration of positions}} in the component $c_i$; that is the set of all matrices generated by $e_{j,k}$ such that $(j,k)\in\mathscr{I}_{c_i}$.
\end{definition}

\begin{ex}
\label{restriction}
Let $\g$ be the seaweed from our running Example \ref{compmeander}, and let $M(\g)$ be its associated meander. See Figure \ref{MeanderinSeaweed2} \textup(left\textup). By Lemma \ref{winding down}, the homotopy type of $\g$ is $H(2,4)$. As before, we can visualize $M(\g)$ inside of $\g$ \textup(see Figure \ref{MeanderinSeaweed2} \textup(right\textup)\textup).  The restriction of $\g$ to its individual components is the span of the matrices $e_{i,j}$, where $(i,j)$ is an index covered by the specified component in the visualization of the meander within the matrix form of the seaweed.  A seaweed might have multiple components of the same size. Further, the restriction of a seaweed to one of its components often has no additional algebraic structure; it may simply be a subspace of $\g$.
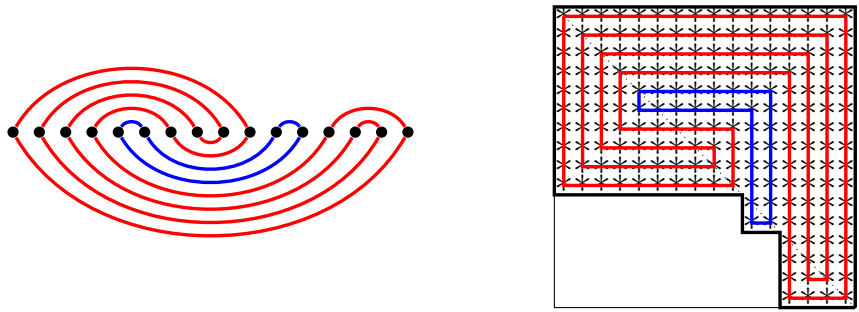
\begin{figure}[H]
$$\begin{tikzpicture}[scale=0.35]
    \def\Node{\node [circle, fill, inner sep=1.5pt]}
    \Node (1) at (0,8){};
    \Node (2) at (1,8){};
    \Node (3) at (2,8){};
    \Node (4) at (3,8){};
    \Node (5) at (4,8){};
    \Node (6) at (5,8){};
    \Node (7) at (6,8){};
    \Node (8) at (7,8){};
    \Node (9) at (8,8){};
    \Node (10) at (9,8){};
    \Node (11) at (10,8){};
    \Node (12) at (11,8){};
    \Node (13) at (12,8){};
    \Node (14) at (13,8){};
    \Node (15) at (14,8){};
    \Node (16) at (15,8){};
    
    \draw [red, line width=0.45mm] (1) to[bend left=60] (10);
    \draw [red, line width=0.45mm] (2) to[bend left=60] (9);
    \draw [red, line width=0.45mm] (3) to[bend left=60] (8);
    \draw [red, line width=0.45mm] (4) to[bend left=60] (7);
    \draw [blue, line width=0.45mm] (5) to[bend left=60] (6);
    \draw [blue, line width=0.45mm] (11) to[bend left=60] (12);
    \draw [red, line width=0.45mm] (13) to[bend left=60] (16);
    \draw [red, line width=0.45mm] (14) to[bend left=60] (15);
    \draw [red, line width=0.45mm] (1) to[bend right=60] (16);
    \draw [red, line width=0.45mm] (2) to[bend right=60] (15);
    \draw [red, line width=0.45mm] (3) to[bend right=60] (14);
    \draw [red, line width=0.45mm] (4) to[bend right=60] (13);
    \draw [blue, line width=0.45mm] (5) to[bend right=60] (12);
    \draw [blue, line width=0.45mm] (6) to[bend right=60] (11);
    \draw [red, line width=0.45mm] (7) to[bend right=60] (10);
    \draw [red, line width=0.45mm] (8) to[bend right=60] (9);
    
    \node at (0,1){};
\end{tikzpicture}
\hspace{5em}
\begin{tikzpicture}[scale=0.25]
\draw [line width=0.45mm] (16,0)--(16,16)--(0,16)--(0,6)--(10,6)--(10,4)--(12,4)--(12,0)--(16,0);
\draw (0,0)--(0,16)--(16,16)--(16,0)--(0,0);
\node at (0.5,15.3){\Large *};
\node at (0.5,14.3){\Large *};
\node at (0.5,13.3){\Large *};
\node at (0.5,12.3){\Large *};
\node at (0.5,11.3){\Large *};
\node at (0.5,10.3){\Large *};
\node at (0.5,9.3){\Large *};
\node at (0.5,8.3){\Large *};
\node at (0.5,7.3){\Large *};
\node at (0.5,6.3){\Large *};
\node at (1.5,15.3){\Large *};
\node at (1.5,14.3){\Large *};
\node at (1.5,13.3){\Large *};
\node at (1.5,12.3){\Large *};
\node at (1.5,11.3){\Large *};
\node at (1.5,10.3){\Large *};
\node at (1.5,9.3){\Large *};
\node at (1.5,8.3){\Large *};
\node at (1.5,7.3){\Large *};
\node at (1.5,6.3){\Large *};
\node at (2.5,15.3){\Large *};
\node at (2.5,14.3){\Large *};
\node at (2.5,13.3){\Large *};
\node at (2.5,12.3){\Large *};
\node at (2.5,11.3){\Large *};
\node at (2.5,10.3){\Large *};
\node at (2.5,9.3){\Large *};
\node at (2.5,8.3){\Large *};
\node at (2.5,7.3){\Large *};
\node at (2.5,6.3){\Large *};
\node at (3.5,15.3){\Large *};
\node at (3.5,14.3){\Large *};
\node at (3.5,13.3){\Large *};
\node at (3.5,12.3){\Large *};
\node at (3.5,11.3){\Large *};
\node at (3.5,10.3){\Large *};
\node at (3.5,9.3){\Large *};
\node at (3.5,8.3){\Large *};
\node at (3.5,7.3){\Large *};
\node at (3.5,6.3){\Large *};
\node at (4.5,15.3){\Large *};
\node at (4.5,14.3){\Large *};
\node at (4.5,13.3){\Large *};
\node at (4.5,12.3){\Large *};
\node at (4.5,11.3){\Large *};
\node at (4.5,10.3){\Large *};
\node at (4.5,9.3){\Large *};
\node at (4.5,8.3){\Large *};
\node at (4.5,7.3){\Large *};
\node at (4.5,6.3){\Large *};
\node at (5.5,15.3){\Large *};
\node at (5.5,14.3){\Large *};
\node at (5.5,13.3){\Large *};
\node at (5.5,12.3){\Large *};
\node at (5.5,11.3){\Large *};
\node at (5.5,10.3){\Large *};
\node at (5.5,9.3){\Large *};
\node at (5.5,8.3){\Large *};
\node at (5.5,7.3){\Large *};
\node at (5.5,6.3){\Large *};
\node at (6.5,15.3){\Large *};
\node at (6.5,14.3){\Large *};
\node at (6.5,13.3){\Large *};
\node at (6.5,12.3){\Large *};
\node at (6.5,11.3){\Large *};
\node at (6.5,10.3){\Large *};
\node at (6.5,9.3){\Large *};
\node at (6.5,8.3){\Large *};
\node at (6.5,7.3){\Large *};
\node at (6.5,6.3){\Large *};
\node at (7.5,15.3){\Large *};
\node at (7.5,14.3){\Large *};
\node at (7.5,13.3){\Large *};
\node at (7.5,12.3){\Large *};
\node at (7.5,11.3){\Large *};
\node at (7.5,10.3){\Large *};
\node at (7.5,9.3){\Large *};
\node at (7.5,8.3){\Large *};
\node at (7.5,7.3){\Large *};
\node at (7.5,6.3){\Large *};
\node at (8.5,15.3){\Large *};
\node at (8.5,14.3){\Large *};
\node at (8.5,13.3){\Large *};
\node at (8.5,12.3){\Large *};
\node at (8.5,11.3){\Large *};
\node at (8.5,10.3){\Large *};
\node at (8.5,9.3){\Large *};
\node at (8.5,8.3){\Large *};
\node at (8.5,7.3){\Large *};
\node at (8.5,6.3){\Large *};
\node at (9.5,15.3){\Large *};
\node at (9.5,14.3){\Large *};
\node at (9.5,13.3){\Large *};
\node at (9.5,12.3){\Large *};
\node at (9.5,11.3){\Large *};
\node at (9.5,10.3){\Large *};
\node at (9.5,9.3){\Large *};
\node at (9.5,8.3){\Large *};
\node at (9.5,7.3){\Large *};
\node at (9.5,6.3){\Large *};
\node at (10.5,15.3){\Large *};
\node at (10.5,14.3){\Large *};
\node at (10.5,13.3){\Large *};
\node at (10.5,12.3){\Large *};
\node at (10.5,11.3){\Large *};
\node at (10.5,10.3){\Large *};
\node at (10.5,9.3){\Large *};
\node at (10.5,8.3){\Large *};
\node at (10.5,7.3){\Large *};
\node at (10.5,6.3){\Large *};
\node at (10.5,5.3){\Large *};
\node at (10.5,4.3){\Large *};
\node at (11.5,15.3){\Large *};
\node at (11.5,14.3){\Large *};
\node at (11.5,13.3){\Large *};
\node at (11.5,12.3){\Large *};
\node at (11.5,11.3){\Large *};
\node at (11.5,10.3){\Large *};
\node at (11.5,9.3){\Large *};
\node at (11.5,8.3){\Large *};
\node at (11.5,7.3){\Large *};
\node at (11.5,6.3){\Large *};
\node at (11.5,5.3){\Large *};
\node at (11.5,4.3){\Large *};
\node at (12.5,15.3){\Large *};
\node at (12.5,14.3){\Large *};
\node at (12.5,13.3){\Large *};
\node at (12.5,12.3){\Large *};
\node at (12.5,11.3){\Large *};
\node at (12.5,10.3){\Large *};
\node at (12.5,9.3){\Large *};
\node at (12.5,8.3){\Large *};
\node at (12.5,7.3){\Large *};
\node at (12.5,6.3){\Large *};
\node at (12.5,5.3){\Large *};
\node at (12.5,4.3){\Large *};
\node at (12.5,3.3){\Large *};
\node at (12.5,2.3){\Large *};
\node at (12.5,1.3){\Large *};
\node at (12.5,0.3){\Large *};
\node at (13.5,15.3){\Large *};
\node at (13.5,14.3){\Large *};
\node at (13.5,13.3){\Large *};
\node at (13.5,12.3){\Large *};
\node at (13.5,11.3){\Large *};
\node at (13.5,10.3){\Large *};
\node at (13.5,9.3){\Large *};
\node at (13.5,8.3){\Large *};
\node at (13.5,7.3){\Large *};
\node at (13.5,6.3){\Large *};
\node at (13.5,5.3){\Large *};
\node at (13.5,4.3){\Large *};
\node at (13.5,3.3){\Large *};
\node at (13.5,2.3){\Large *};
\node at (13.5,1.3){\Large *};
\node at (13.5,0.3){\Large *};
\node at (14.5,15.3){\Large *};
\node at (14.5,14.3){\Large *};
\node at (14.5,13.3){\Large *};
\node at (14.5,12.3){\Large *};
\node at (14.5,11.3){\Large *};
\node at (14.5,10.3){\Large *};
\node at (14.5,9.3){\Large *};
\node at (14.5,8.3){\Large *};
\node at (14.5,7.3){\Large *};
\node at (14.5,6.3){\Large *};
\node at (14.5,5.3){\Large *};
\node at (14.5,4.3){\Large *};
\node at (14.5,3.3){\Large *};
\node at (14.5,2.3){\Large *};
\node at (14.5,1.3){\Large *};
\node at (14.5,0.3){\Large *};
\node at (15.5,15.3){\Large *};
\node at (15.5,14.3){\Large *};
\node at (15.5,13.3){\Large *};
\node at (15.5,12.3){\Large *};
\node at (15.5,11.3){\Large *};
\node at (15.5,10.3){\Large *};
\node at (15.5,9.3){\Large *};
\node at (15.5,8.3){\Large *};
\node at (15.5,7.3){\Large *};
\node at (15.5,6.3){\Large *};
\node at (15.5,5.3){\Large *};
\node at (15.5,4.3){\Large *};
\node at (15.5,3.3){\Large *};
\node at (15.5,2.3){\Large *};
\node at (15.5,1.3){\Large *};
\node at (15.5,0.3){\Large *};

\draw [dotted] (0,16)--(16,0);
\draw [red, line width=0.45mm] (0.5,15.5)--(15.5,15.5)--(15.5,0.5)--(12.5,0.5)--(12.5,12.5)--(3.5,12.5)--(3.5,9.5)--(9.5,9.5)--(9.5,6.5)--(0.5,6.5)--(0.5,15.5);
\draw [red, line width=0.45mm] (1.5,14.5)--(14.5,14.5)--(14.5,1.5)--(13.5,1.5)--(13.5,13.5)--(2.5,13.5)--(2.5,8.5)--(8.5,8.5)--(8.5,7.5)--(1.5,7.5)--(1.5,14.5);
\draw [blue, line width=0.45mm] (4.5,11.5)--(4.5,10.5)--(10.5,10.5)--(10.5,4.5)--(11.5,4.5)--(11.5,11.5)--(4.5,11.5);
\end{tikzpicture}$$
\caption{Meander of type $\frac{10|2|4}{16}$ \textup(left\textup), visualized in the seaweed \textup(right\textup)}
\label{MeanderinSeaweed2}
\end{figure}

\end{ex}

The following Theorem \ref{components} is a trivial consequence of Definition \ref{res}.

\begin{theorem}\label{components}
If $\g$ is a seaweed with homotopy type $H(c_1,\cdots,c_h)$, then $\g=\sum_{i=1}^h \g|_{c_i}.$
\end{theorem}

We now highlight two other important subsets of $\g$ called, respectively, the \textit{core} of $\g$ and \textit{peak set} of $\g$.

\begin{definition}
\label{core}
Let $\g$ be a seaweed with homotopy type $H(c_1,\cdots,c_h)$ and component meander $CM(\g)$. Consider one component $c_i$. Define the sets 

\vspace{-1em}

$$
V_{c_i}=\{A_j\;|\;v_{A_j}\text{is a vertex in $CM(\g)$ on the path of $c_i$}\},$$

\vspace{-1em}

$$\textgoth{C}_{c_i}=\{A_I\times A_I\;|\;A_I\in V_{c_i}\}.$$ 

The set $\textgoth{C}_{c_i}$ is the {core} of $\g|_{c_i}$ -- the set of $c_i\times c_i$ blocks on the diagonal of $\g$ contained in $\g|_{c_i}$. Fix a vertex $v_{A_I}$ on the path of $c_i$ in $CM(\g)$. Partition $V_{c_i}$ into two sets: 

\begin{itemize}
    \item []$\mathscr{A}_{c_i}=\{A_j\;|\;\text{ the path from $v_{A_I}$ to $v_{A_j}$ has odd length}\},$ and
    \item[] $\mathscr{B}_{c_i}=\{A_j\;|\;\text{ the path from $v_{A_I}$ to $v_{A_j}$ has even or zero length}\}.$
\end{itemize}

\noindent
Note that the choice of partitioning by distance from $V_{A_I}$ is arbitrary.  Now orient $CM(\g)$ counter-clockwise \textup(i.e., top edges are oriented from right to left, bottom edges are oriented left to right\textup). Let $E_{CM(\g)}$ be the set of edges in the oriented $CM(\g)$. Define the {peak set} of $\g|_{c_i}$ as 

$$\textgoth{P}_{c_i}=\{A_I\times A_J\;|\;A_I,A_J\in V_{c_i}\text{ with }(A_I,A_J)\in E_{CM(\g)}\}.$$

\noindent
We define the {core of $\g$} and the {peak set of $\g$} as the union of the core and peak sets, respectively, of the components in the homotopy type. In other words,

$$\textgoth{C}_\g=\bigcup_{i=1}^h\textgoth{C}_{c_i}\hspace{2em}\text{ and }\hspace{2em}\textgoth{P}_\g=\bigcup_{i=1}^h\textgoth{P}_{c_i}.$$
\end{definition}

\begin{ex}
\label{Components}
Consider once again $\g$ from Example \ref{compmeander}. Table \ref{comptab} lists $V_{c_i}$, $\textgoth{C}_{c_i}$, $\mathscr{A}_{c_i}$, $\mathscr{B}_{c_i}$, and $\textgoth{P}_{c_i}$ for $c_1=2$ and $c_2=4$. In the second and third columns of Table \ref{comptab}, $\g|_{c_i}$ is shaded to better highlight which configuration of $\g|_{c_i}$ is being identified. Further, $\textgoth{C}_{c_i}$ and $\textgoth{P}_{c_i}$ are represented as matrices with an asterisk to represent every index $(i,j)$ which could appear in $\bigcup_{C\in\textgoth{C}_{c_i}}C$ or $\bigcup_{P\in\textgoth{P}_{c_i}}P$. Each individual $c_i\times c_i$ block is a set of indices in the corresponding core or peak set.
\vspace{-0.5em}
\begin{table}[H]
\begin{center}
\scalebox{0.8}{
\begin{tabular}{|| c || c | c ||}
\hline\hline 
& $c_1=2$ & $c_2=4$ \\
\hline 
$V_{c_i}$ & $\{\{5,6\},\;\{11,12\}\}$ & $\{\{1,2,3,4\},\;\{7,8,9,10\},\;\{13,14,15,16\}\}$\\
\hline 
$\textgoth{C}_{c_i}$ && \\
& 
$\begin{tikzpicture}[scale=0.25]
    \def\Node{\node [circle, fill, inner sep=1.5pt]}
    \draw [line width=0.45mm] (0,0)--(0,16)--(16,16)--(16,0)--(0,0);
    \draw (0,16)--(0,6)--(10,6)--(10,4)--(12,4)--(12,0);
    \filldraw[line width=0.25mm, draw=black, fill=gray!10] (4,12)--(12,12)--(12,4)--(10,4)--(10,10)--(4,10)--(4,12);
    \draw [dotted] (0,16)--(16,0);
    
    \draw (4,12)--(6,12)--(6,10)--(4,10)--(4,12);
    \draw (10,6)--(12,6)--(12,4)--(10,4)--(10,6);
    \node at (4.5,11.2){\large *};
    \node at (4.5,10.2){\large *};
    \node at (5.5,11.2){\large *};
    \node at (5.5,10.2){\large *};
    
    \node at (10.5,5.2){\large *};
    \node at (10.5,4.2){\large *};
    \node at (11.5,5.2){\large *};
    \node at (11.5,4.2){\large *};
\end{tikzpicture}$ &
$\begin{tikzpicture}[scale=0.25]
    \def\Node{\node [circle, fill, inner sep=1.5pt]}
    \draw [line width=0.45mm] (0,0)--(0,16)--(16,16)--(16,0)--(0,0);
    \draw (0,16)--(0,6)--(10,6)--(10,4)--(12,4)--(12,0);
    \filldraw[line width=0.25mm, draw=black, fill=gray!10] (10,6)--(0,6)--(0,16)--(16,16)--(16,0)--(12,0)--(12,12)--(4,12)--(4,10)--(10,10)--(10,6);
    \draw [dotted] (0,16)--(16,0);
    
    \draw (0,16)--(4,16)--(4,12)--(0,12);
    \draw (6,10)--(10,10)--(10,6)--(6,6)--(6,10);
    \draw (12,4)--(16,4);
    \node at (0.5,15.2){\large *};
    \node at (0.5,14.2){\large *};
    \node at (0.5,13.2){\large *};
    \node at (0.5,12.2){\large *};
    \node at (1.5,15.2){\large *};
    \node at (1.5,14.2){\large *};
    \node at (1.5,13.2){\large *};
    \node at (1.5,12.2){\large *};
    \node at (2.5,15.2){\large *};
    \node at (2.5,14.2){\large *};
    \node at (2.5,13.2){\large *};
    \node at (2.5,12.2){\large *};
    \node at (3.5,15.2){\large *};
    \node at (3.5,14.2){\large *};
    \node at (3.5,13.2){\large *};
    \node at (3.5,12.2){\large *};
    
    \node at (6.5,9.2){\large *};
    \node at (6.5,8.2){\large *};
    \node at (6.5,7.2){\large *};
    \node at (6.5,6.2){\large *};
    \node at (7.5,9.2){\large *};
    \node at (7.5,8.2){\large *};
    \node at (7.5,7.2){\large *};
    \node at (7.5,6.2){\large *};
    \node at (8.5,9.2){\large *};
    \node at (8.5,8.2){\large *};
    \node at (8.5,7.2){\large *};
    \node at (8.5,6.2){\large *};
    \node at (9.5,9.2){\large *};
    \node at (9.5,8.2){\large *};
    \node at (9.5,7.2){\large *};
    \node at (9.5,6.2){\large *};
    
    \node at (12.5,3.2){\large *};
    \node at (12.5,2.2){\large *};
    \node at (12.5,1.2){\large *};
    \node at (12.5,0.2){\large *};
    \node at (13.5,3.2){\large *};
    \node at (13.5,2.2){\large *};
    \node at (13.5,1.2){\large *};
    \node at (13.5,0.2){\large *};
    \node at (14.5,3.2){\large *};
    \node at (14.5,2.2){\large *};
    \node at (14.5,1.2){\large *};
    \node at (14.5,0.2){\large *};
    \node at (15.5,3.2){\large *};
    \node at (15.5,2.2){\large *};
    \node at (15.5,1.2){\large *};
    \node at (15.5,0.2){\large *};
\end{tikzpicture}$\\
\hline 
$\mathscr{A}_{c_i}$ & $\{\{11,12\}\}$ & $\{\{7,8,9,10\},\;\{13,14,15,16\}\}$ \\
\hline 
$\mathscr{B}_{c_i}$ & $\{\{5,6\}\}$ & $\{\{1,2,3,4\}\}$ \\
\hline 
$\textgoth{P}_{c_i}$ &&\\
&$\begin{tikzpicture}[scale=0.25]
    \def\Node{\node [circle, fill, inner sep=1.5pt]}
    \draw [line width=0.45mm] (0,0)--(0,16)--(16,16)--(16,0)--(0,0);
    \draw (0,16)--(0,6)--(10,6)--(10,4)--(12,4)--(12,0);
    \filldraw[line width=0.25mm, draw=black, fill=gray!10] (4,12)--(12,12)--(12,4)--(10,4)--(10,10)--(4,10)--(4,12);
    \draw [dotted] (0,16)--(16,0);
    
    \draw (10,12)--(12,12)--(12,10)--(10,10)--(10,12);
    
    \node at (10.5,11.2){\large *};
    \node at (10.5,10.2){\large *};
    \node at (11.5,11.2){\large *};
    \node at (11.5,10.2){\large *};
\end{tikzpicture}$ &
$\begin{tikzpicture}[scale=0.25]
    \def\Node{\node [circle, fill, inner sep=1.5pt]}
    \draw [line width=0.45mm] (0,0)--(0,16)--(16,16)--(16,0)--(0,0);
    \draw (0,16)--(0,6)--(10,6)--(10,4)--(12,4)--(12,0);
    \filldraw[line width=0.25mm, draw=black, fill=gray!10] (10,6)--(0,6)--(0,16)--(16,16)--(16,0)--(12,0)--(12,12)--(4,12)--(4,10)--(10,10)--(10,6);
    \draw [dotted] (0,16)--(16,0);
    
    \draw (12,16)--(12,12)--(16,12);
    \draw (0,10)--(4,10)--(4,6)--(0,6);
    \node at (0.5,9.2){\large *};
    \node at (0.5,8.2){\large *};
    \node at (0.5,7.2){\large *};
    \node at (0.5,6.2){\large *};
    \node at (1.5,9.2){\large *};
    \node at (1.5,8.2){\large *};
    \node at (1.5,7.2){\large *};
    \node at (1.5,6.2){\large *};
    \node at (2.5,9.2){\large *};
    \node at (2.5,8.2){\large *};
    \node at (2.5,7.2){\large *};
    \node at (2.5,6.2){\large *};
    \node at (3.5,9.2){\large *};
    \node at (3.5,8.2){\large *};
    \node at (3.5,7.2){\large *};
    \node at (3.5,6.2){\large *};
    
    \node at (12.5,15.2){\large *};
    \node at (12.5,14.2){\large *};
    \node at (12.5,13.2){\large *};
    \node at (12.5,12.2){\large *};
    \node at (13.5,15.2){\large *};
    \node at (13.5,14.2){\large *};
    \node at (13.5,13.2){\large *};
    \node at (13.5,12.2){\large *};
    \node at (14.5,15.2){\large *};
    \node at (14.5,14.2){\large *};
    \node at (14.5,13.2){\large *};
    \node at (14.5,12.2){\large *};
    \node at (15.5,15.2){\large *};
    \node at (15.5,14.2){\large *};
    \node at (15.5,13.2){\large *};
    \node at (15.5,12.2){\large *};
\end{tikzpicture}$\\
\hline\hline 
\end{tabular}
}
\caption{$V_{c_i}$, $\textgoth{C}_{c_i}$, $\mathscr{A}_{c_i}$, $\mathscr{B}_{c_i}$, and $\textgoth{P}_{c_i}$ in $\g$ of type $\frac{10|2|4}{16}$}
\label{comptab}
\end{center}
\end{table}
\end{ex}
\end{subsection}

\section{Regular functionals on $\gl(n)$}
\label{New Functional}

In this section, we construct a regular functional $F$ on a seaweed $\g$ with homotopy type $H(c_1,\cdots,c_h)$.  We do this by developing a broad analytic framework (see Section \ref{Function Theory}) which relies on the choices of regular functionals $F_{c_i}\in\gl(c_i)^*$.   The construction of the functional $F$ involves embedding copies of $F_{c_i}$ into the core of $\mathfrak{g}$ in such a way that the constructed functional $F$ satisfies 

$$\dim\ker(B_F)=\sum_{i=1}^h\dim\ker(B_{F_{c_i}}).$$ 

The regularity of the constructed $F$ is assured by Theorem \ref{Functional Construction}, the proof for which is an induction on the winding-down moves of Lemma \ref{winding down}.  The induction makes heavy use of a {\textit{relations matrix}} (see Section \ref{Function Theory}), a bookkeeping device which encodes, among other things, the degrees of freedom in the system of the equations which define $\ker(B_F)$.

Associated with a relations matrix $B$ is a non-unique minimal set $P$ of matrix positions $(i,j)$.  The entries in the  remaining positions of $B$ are explicitly determined as linear combinations of the entries in the positions in $P$. Each assignment, $\Vec{b}=(b_{i,j} ~ \vert ~ (i,j) \in P)$, of complex numbers to the positions in $P$ yields an element of $\ker (B_F)$, so $\dim\ker(B_F)=\vert P \vert$, and the resulting kernel elements span $\ker (B_F)$.


In Section \ref{A Framework} we develop a framework for the construction of a regular functional on $\mathfrak{gl}(n)$.  The explicit functional is built in Section \ref{An Explicit Regular Functionals on gl(n)}.  We close with Section \ref{Three More Regular Functionals}, where several more explicit regular functionals on $\mathfrak{gl}(n)$ are established.







\begin{subsection}{A relations matrix of a matrix vector space}
\label{Function Theory}

Let $\g$ be a seaweed of type $\frac{a_1|\cdots|a_m}{b_1|\cdots|b_t}$. Every $F\in\g^*$ is defined in terms of the functionals $e_{i,j}^*$. We may therefore write $F$ in the form $F=\sum_{(i,j)\in\mathscr{I}_F}c_{i,j}e_{i,j}^*$, with $c_{i,j}\in\C$ and $\mathscr{I}_F\subseteq 
\{ 1, \ldots, n\} \times \{ 1, \ldots, n\}$ the {\textit{index set}} of $F$. For any matrix $B$, denote by $B^t$ the \textit{{transpose}} of $B$. Similarly, if $F=\sum_{(i,j)\in\mathscr{I}_F}c_{i,j}e_{i,j}^*$, define by $F^t$ the transpose of $F$ (i.e., $F^t=\sum_{(i,j)\in\mathscr{I}_F}c_{i,j}e_{j,i}^*$ and $\mathscr{I}_{F^t}=\{(j,i)\;|\;(i,j)\in\mathscr{I}_F\}$). We call $F$ (and similarly $\mathscr{I}_F$) {\textit{symmetric}} with respect to the main diagonal if $F=F^t$ (i.e., $\mathscr{I}_F=\mathscr{I}_{F^t}$). Using the same terminology, we call $\g$ symmetric if $\g^t:=\{X^t\;|\;X\in\g\}=\g$. This happens if and only if $\overline{a}=\overline{b}$, or equivalently if and only if $\g=\bigoplus_{i=1}^m\gl(a_i)$. Denote by $\mathscr{I}_\g$ the set of all admissible locations in $\g$ (i.e., $\g$ is the linear span of $\{e_{i,j}\;|\;(i,j)\in\mathscr{I}_\g\}$). If $F\in\g^*$ then we assume $\mathscr{I}_F\subseteq\mathscr{I}_\g$. We will use the superscript $\widehat{t}$ to represent transposition across the antidiagonal (i.e., if $F$ is defined on $\gl(n)$, then $F^{\widehat{t}}=\sum_{(i,j)\in\mathscr{I}_F}e^*_{n+1-j,n+1-i}$, etc.), and we have analogous definitions with respect to the antidiagonal. We will also use the superscript $R$ to represent rotation of a matrix twice (i.e., $B^R=(A_n)^{-1}B(A_n)$, where $A_n=\sum_{i=1}^n e_{i,n+1-i}$), and we have all the analogous definitions.

We now introduce a \textit{relations matrix} of a space of matrices, which is formally defined in Definition \ref{relations matrix}. This is an abstract bookkeeping device which encodes the dimension of a space through the degrees of freedom and an explicit basis for the space, demonstrating how entries of a matrix in a given vector space are related to each other. Therefore, a relations matrix is defined (up to a relabeling of the variables $b_i$ encoding the degrees of freedom) by a choice of basis for a space, and so is defined up to conjugation by $G\in GL(n;\C)$. Our purpose in constructing relations matrices for spaces is to infer the number of degrees of freedom of a space from it, so actual form does not matter. See example \ref{Rep Matrix}.
\begin{definition}
\label{relations matrix}
Given a subspace $\mf{q}\subseteq\gl(n)$ with $\dim\mf{q}=m$, fix a basis $\{q_1,\cdots,q_m\}$ of $\mf{q}$. Define a linear transformation $f:\C^m\rightarrow\mf{q}$ by $f(e_i)=q_i$ for each $i$. Given variables $b_1,\cdots,b_m$, the matrix form $$B:=f(b_1,\cdots,b_m)=\sum_{i=1}^mb_iq_i$$ is a {relations matrix} of $\mf{q}$. By substitution of the variables $b_i$ with elements of $\C$, the matrix $B$ satisfies the following statement: $\mf{q}=\{B\;|\;b_i\in\C\}$.
\end{definition}

\begin{ex}
\label{Rep Matrix}
Let $\mathfrak{q}\subseteq\gl(2)$ be the space of matrices $\left(\begin{array}{cc}x_1&x_2\\x_3&x_4\end{array}\right)$ subject to the constraints $x_1=x_2+x_4$ and $x_3=x_2$. A basis for the space is $$\left\{\left(\begin{array}{cc}1&1\\1&0\end{array}\right),\;\left(\begin{array}{cc}1&0\\0&1\end{array}\right)\right\}.$$ The matrix $$B=b_1q_1+b_2q_2=b_1\left(\begin{array}{cc}1&1\\1&0\end{array}\right)+b_2\left(\begin{array}{cc}1&0\\0&1\end{array}\right)=\left(\begin{array}{cc}b_1+b_2&b_1\\b_1&b_2\end{array}\right)$$ is a relations matrix of $\mf{q}$. 
\end{ex}

To facilitate the construction of a relations matrix of $\ker(B_F)$, we make use of the following technical lemmas. The first, Lemma \ref{Symmetry Lemma}, is used to shorten the necessary calculations in constructing a relations matrix by any existent symmetry in the seaweed and functional.


\begin{lemma}
\label{Symmetry Lemma}
Let $F\in\g^*$ for a seaweed $\g$ such that $F$ and $\g$ are symmetric with respect to the main diagonal \textup(or the antidiagonal\textup). Let $B=[b_{i,j}]$ be a relations matrix of $\ker(B_F)$. Let $\mathscr{B}=\{b_{i,j}\}$ be the set of free variables in $B$ -- i.e., if $\mathscr{I}$ is the set of indices in $\mathscr{B}$, then for each $(i,j)\in \mathscr{I}_\g$, there exist complex $c_{\alpha,\beta}$ such that $b_{i,j}=\sum_{(\alpha,\beta)\in\mathscr{I}} c_{\alpha,\beta}b_{\alpha,\beta}$. For all $(i,j)\in\mathscr{I}_\g$, if $b_{i,j}=\sum_{(\alpha,\beta)\in\mathscr{I}} c_{\alpha,\beta}b_{\alpha,\beta}$ with $c_{\alpha,\beta}\in\C$, then $b_{j,i}=\sum_{(\alpha,\beta)\in\mathscr{I}} c_{\alpha,\beta}b_{\beta,\alpha}$ \textup(respectively, $b_{n+1-j,n+1-i}=\sum_{(\alpha,\beta)\in\mathscr{I}}c_{\alpha_\beta}b_{n+1-\beta,n+1-\alpha}$\textup).
\end{lemma}

\begin{proof} 
We establish the theorem assuming symmetry across the main diagonal (the antidiagonal proof is similar.) For each $(i,j)\in\mathscr{I}_\g$, there exist $c_{\alpha,\beta}\in\C$ such that 

$$b_{i,j}=\sum_{(\alpha,\beta)\in\mathscr{I}}c_{\alpha,\beta}b_{\alpha,\beta}$$ 

\noindent
by the definition of $\mathscr{B}$. For each $B\in\ker(B_F)$, consider $B^t=[b_{i,j}']$ and note that $b_{i,j}'=b_{j,i}$. Evidently, $B^t\in\ker(B_{F^t})$, where $F^t$ is defined on $\g^t$. However, $\g^t=\g$ and $F^t=F$ by assumption. Therefore, $B^t\in\ker(B_F)$, and $$b_{j,i}=b_{i,j}'=\sum_{(\alpha,\beta)\in\mathscr{I}}c_{\alpha,\beta}b_{\alpha,\beta}'=\sum_{(\alpha,\beta)\in\mathscr{I}}c_{\alpha,\beta}b_{\beta,\alpha}. $$ 
\end{proof}


We have the following easy corollary to Lemma \ref{Symmetry Lemma}.

\begin{lemma}
\label{zero symmetry}
If $F\in\g^*$, and both $F$ and $\g$ are symmetric with respect to the main diagonal \textup(or the antidiagonal\textup), then for any relations matrix $B$ of $\ker(B_F)$, $b_{i,j}=0$ if and only if $b_{j,i}=0$ \textup(respectively, $b_{n+1-j,n+1-i}=0$\textup).
\end{lemma}

To prove that a matrix $B$ is in $\ker(B_F)$ for $F\in\g^*$ amounts to showing that the entries $b_{i,j}$ in $B$ satisfy a specific system of equations. This system is developed in Lemma \ref{Syst of Eqs}.

\begin{lemma}
\label{Syst of Eqs}
Let $\g$ be a seaweed, and let $F=\sum_{(i,j)\in\mathscr{I}_F}c_{i,j}e_{i,j}^*\in\g^*$ with $c_{i,j}\in\C$. The space $\ker(B_F)$ is spanned by all matrices $B=[b_{i,j}]$ whose entries $b_{i,j}$ form a solution to the two sets of equations:
\begin{enumerate}
    \item $\sum_{(s,j)\in\mathscr{I}_F}c_{s,j}b_{s,i}=\sum_{(i,s)\in\mathscr{I}_F}c_{i,s}b_{j,s},\hspace{3em}\text{ for all }(i,j)\in\mathscr{I}_\g;$ 
    \item $b_{i,j}=0,\hspace{14.75em}\text{ for all }(i,j)\not\in\mathscr{I}_\g.$
\end{enumerate}
\end{lemma}

\begin{proof}
Let $B=[b_{i,j}]\in\ker(B_F)$. The second set of equations follow trivially. To show $B\in\ker(B_F)$, it is necessary and sufficient to require $B_F(B,e_{i,j})=0$, for all $(i,j)\in\mathscr{I}_\g$. Consider the image of $e_{i,j}$ under $B_F(B,\cdot)$. To start, note that
\begin{center}
\scalebox{0.75}{
\begin{tabular}{c}
$[B,e_{i,j}]=\left(\begin{array}{ccccccccc}0&0&0&\cdots&0&b_{1,i}&0&\cdots&0
\\
0&0&0&\cdots&0&b_{2,i}&0&\cdots&0\\
0&0&0&\cdots&0&b_{3,i}&0&\cdots&0\\
&&&&\vdots &&&&\\
0&0&0&\cdots&0&b_{n-1,i}&0&\cdots&0\\
0&0&0&\cdots&0&b_{n,i}&0&\cdots&0\end{array}\right)-\left(\begin{array}{ccccccc}0&0&0&\cdots&0&0&0\\
0&0&0&\cdots&0&0&0\\
&&&\vdots &&&\\
0&0&0&\cdots&0&0&0\\
b_{j,1}&b_{j,2}&b_{j,3}&\cdots &b_{j,n-2}&b_{j,n-1}&b_{j,n}\\
0&0&0&\cdots&0&0&0\\
&&&\vdots &&&\\
0&0&0&\cdots&0&0&0
\end{array}\right),$
\end{tabular}
}
\end{center}
where the above matrices are $n\times n$ matrices with a potentially non-zero column $j$ and non-zero row $i$, respectively.  It follows that
\begin{equation} 
\label{syst1}
F([B,e_{i,j}])=\sum_{(s,j)\in\mathscr{I}_F}c_{s,j}b_{s,i}-\sum_{(i,s)\in\mathscr{I}_F}c_{i,s}b_{j,s}.
\end{equation}
Upon evaluating (\ref{syst1}) at zero, the first set of equations follow. 
\end{proof}
\end{subsection}

\begin{subsection}{A framework for building regular functionals on seaweed algebras}
\label{A Framework}

To describe how we will construct a functional on a seaweed $\g$, first assume that $F_c$ represents a functional (not necessarily regular) on $\gl(c)$ for any $c>0$. The functionals $F_{c_i}$ will be our building blocks for any seaweed of homotopy type $H(c_1,\cdots,c_h)$. Several explicit examples of regular functionals are provided in Sections \ref{An Explicit Regular Functionals on gl(n)} and \ref{Three More Regular Functionals}. The indices in $\mathscr{I}_{F_c}$ will be fixed as a subset of $c\times c$. 

\begin{definition}
Given a seaweed $\g$, a functional $F\in\g^*$, and $a\in\N$, define the {shift of $F$ by $a$} as the new functional 
\begin{equation}
    \label{shift}
    F^a:=\sum_{(i,j)\in\mathscr{I}_F}e_{i+a,j+a}^*.
\end{equation}
Note: The right-hand side of \textup(\ref{shift}\textup) is defined only when the indices are admissible indices in the seaweed.
\end{definition}

\noindent
Given $\g_1\subseteq\gl(n_1)$ and $\g_2\subseteq\gl(n_2)$ with $F_i\in\g_i^*$, we define the functional $F_1\oplus F_2$ in $(\g_1\oplus\g_2)^*$ by $F_1+F_2^{n_1}$.

\begin{lemma}
\label{direct sum}
Let $\g$ be a seaweed and assume that there exist $\g_i\subseteq\gl(n_i)$ such that $\g=\bigoplus_{i=1}^k\g_i$. Let $F_i\in\g_i^*$, for all $i$ and define $F=\bigoplus_{i=1}^k F_i$. A matrix $B$ is a relations matrix of $\ker(B_F)$ if and only if $B=\bigoplus_{i=1}^k B_i,$ where $B_i$ is a relations matrix of $\ker(B_{F_i})$ for each $i$. It follows that $\dim\ker(B_F)=\sum_{i=1}^k\dim\ker(B_{F_i}).$
\end{lemma}

\begin{proof}
By induction, it suffices to prove the claim for $\g=\g_1\oplus\g_2$ with $\g_i\subseteq\gl(n_i)$. Let $F_i\in\g_i^*$ for each $i$, and define $F=F_1\oplus F_2$. By construction, $$\ker(B_F)=\ker(B_{F_1})\oplus\ker(B_{F_2}),$$ and therefore $\dim\ker(B_F)=\dim\ker(B_{F_1})+\dim\ker(B_{F_2})$.

For the reverse direction, assume $B_i$ a relations matrix of $\ker(B_{F_i})$. By definition, there exist linear transformations $f_i:\C^{m_i}\rightarrow\ker(B_{F_i})$, where $m_i$ is the dimension of $\ker(B_{F_i})$, appropriately defined so that $B_i=f(b_1^i,\cdots,b_{m_i}^i)$. Define $f:\C^{m_1+m_2}\rightarrow\ker(B_F)$ by $$f(x_1,\cdots,x_m)=f_1(x_1,\cdots,x_{m_1})\oplus f_2(x_{m_1+1},\cdots,x_m).$$ The matrix $B=f(b_1,\cdots,b_m)=f_1(b_1,\cdots,b_{m_1})\oplus f_2(b_{m_1+1},\cdots,b_m)=B_1\oplus B_2$ is a relations matrix of $\ker(B_F)$.  The dimension result follows.  \end{proof}

\begin{theorem}
\label{ds regulars}
If $\g\subseteq\gl(n)$ is such that $\g=\bigoplus_{i=1}^k\g_i$ for $\g_i\subseteq\gl(n_i)$, and $F\in\g^*$, then $F$ is regular if and only if $F=\bigoplus_{i=1}^k F_i$ with $F_i$ regular on $\g_i$ for each $i$.
\end{theorem} 

\begin{proof}
Fix $F_i\in\g_i^*$ such that $F=\bigoplus_{i=1}^k F_i$. 

Assume that $F$ is regular and, towards a contradiction, that there exists $j$ such that $F_j$ is not regular on $\g_j$. Fix $F_j'$ regular on $\g_j$ and define $F'=\bigoplus_{i=1}^{j-1}F_i\oplus F_j'\bigoplus_{i=j+1}^kF_i.$ By definition, $\dim\ker(B_{F_j'})<\dim\ker(B_{F_j})$, and by Lemma \ref{direct sum}

\begin{align*}
\dim\ker(B_{F'})&=\sum_{\underset{i\neq j}{i=1}}^k\dim\ker(B_{F_i})+\dim\ker(B_{F_j'})\\
&<\sum_{\underset{i\neq j}{i=1}}^k\dim\ker(B_{F_i})+\dim\ker(B_{F_j})\\
&=\dim\ker(B_F).    
\end{align*}

\noindent
This contradicts the regularity of $F$.

Now, assume that $F_i$ is regular for all $i$. Again, if $F$ is not regular fix a regular $F'\in\g^*$. Let $F_i'\in\g_i^*$ be such that $F'=\bigoplus_{i=1}^k F_i'$. We have 

$$\sum_{i=1}^k\dim\ker(B_{F_i})=\dim\ker(B_F)>\dim\ker(B_{F'})=\sum_{i=1}^k\dim\ker(B_{F_i'}).$$

\noindent 
Let $j$ be the first index such that $\dim\ker(B_{F_i})>\dim\ker(B_{F_i'})$. Then $F_i$ is not regular on $\g_i$.
\end{proof}

\begin{lemma}
\label{New functions by Symmetry}
Let $\g$ be a seaweed and let $F\in\g^*$. If $\g=\g^t$, then $$\dim\ker(B_F)=\dim\ker(B_{F^t}).$$ Similarly, if $\g=\g^{\widehat{t}}$, then $$\dim\ker(B_F)=\dim\ker(B_{F^{\widehat{t}}}).$$ It follows that $F$ is regular if and only if $F^t$ \textup(respectively, $F^{\widehat{t}}$ \textup) is regular on $\g$.
\end{lemma}

Now, using the component meander associated with a seaweed $\g$, we describe a method for building a functional $F\in\g^*$ using functionals $F_{c_i}\in\gl(c_i)^*$ over the components $c_i$ of $\g$'s homotopy type $H(c_1,\cdots,c_h)$.

\begin{definition}
\label{Functional Construction def}
Let $\g$ be a seaweed with homotopy type $H(c_1,\cdots,c_h)$. Let $\mathscr{A}_{c_i}$ and $\mathscr{B}_{c_i}$ be defined as in Definition \ref{core}. Let $F_c^R=\sum_{(i,j)\in\mathscr{I}_{F_c}}e_{c+1-i,c+1-j}^*$ (i.e., $F_c^R$ is a rotation of the indices in $F_c$). Define sets 
\begin{center}
\scalebox{0.85}{
\begin{tabular}{c}
$\mathscr{D}^a_{c_i}=\bigcup_{P\in\textgoth{P}}\{(I'-s,J'+s)\;|\;s\in[0,c_i-1],\;\;I'=\max A_I,\;\;J'=\min A_J,\text{ for }P=A_I\times A_J\}$\\
\end{tabular}
}
\end{center}
and 
\begin{center}
\scalebox{0.85}{
\begin{tabular}{c}
$\mathscr{D}_{c_i}=\bigcup_{P\in\textgoth{P}}\{(I'+s,J'+s)\;|\;s\in[0,c_i-1],\;\;I'=\min A_I,\;\;J'=\min A_J,\text{ for }P=A_I\times A_J\}$.\\
\end{tabular}
}
\end{center}

The sets $\mathscr{D}_{c_i}^a$ and $\mathscr{D}_{c_i}$ are the entries on the antidiagonal and main diagonal (respectively) of each $c_i\times c_i$ square $A_I\times A_J$ in $\textgoth{P}_{c_i}$. Given a functional $F_{c_i}\in\gl(c_i)^*$, define functionals $\overline{F}^a_{c_i}$ and $\oF_{c_i}$ in $\g^*$ as follows:
$$\overline{F}^a_{c_i}:=\sum_{A\in\mathscr{A}_{c_i}}(F_{c_i}^R)^{\min(A)-1}+\sum_{A\in\mathscr{B}_{c_i}}(F_{c_i})^{\min(A)-1}+\sum_{(i,j)\in\mathscr{D}^a_{c_i}}e_{i,j}^*,$$ $$\oF_{c_i}:=\sum_{A\in\mathscr{A}_{c_i}\cup\mathscr{B}_{c_i}}(F_{c_i})^{\min(A)-1}+\sum_{(i,j)\in\mathscr{D}_{c_i}}e_{i,j}^*.$$ Define two functionals $\overline{F}^a,\oF\in\g^*$ by $$\overline{F}^a:=\sum_{i=1}^h\overline{F}^a_{c_i},\hspace{2em}\text{ and }\hspace{2em}\oF:=\sum_{i=1}^h\oF_{c_i}.$$
\end{definition}

The proof of Theorem \ref{Functional Construction} deals explicitly with $\oF$ and $\oF^a$ as defined, but no aspect of the proof requires the consistent choice of functionals in the peak blocks (i.e. antidiagonal as in $\oF^a$ or main diagonal as in $\oF$). It follows that these methods may be mixed between and within components of $\g$ (see Example \ref{mixmatch}).

\begin{ex}
\label{example functional}
Let $\g$ be the seaweed of our running Example \ref{compmeander}. Recall that $\g$ has type $\frac{10|4|2}{16}$ and homotopy type $H(2,4)$. Let $\overline{F}^a$ and $\oF$ be constructed using $F_2\in\gl(2)^*$ and $F_4\in\gl(4)^*$ of Theorem \ref{My Theorem 1}.

We illustrate the sets $\mathscr{I}_{\overline{F}^a}$ and $\mathscr{I}_{\oF}$ by placing a black dot in each entry $(i,j)\in\mathscr{I}_{\overline{F}}$ and $(i,j)\in\mathscr{I}_{\overline{F}^a}$ in the matrix form of $\g$ in Figure \ref{solution} \textup(left and right, respectively\textup). The configuration of positions $\g|_4$ is left shaded in grey to emphasize the embedding of the functionals $F_2$ and $F_4$ into the core and how the peak dots affect this choice. The functionals $\overline{F}_4$ and $\oF^a_4$ are the sum of $e_{i,j}^*$ where $(i,j)$ is in the shaded region, while  the functionals $\overline{F}_2$ and $\oF^a_2$ are the sum of $e_{i,j}^*$ over the indices $(i,j)$ outside the shaded region.
\begin{figure}[H]
$$\begin{tikzpicture}[scale=0.2]
    \def\Node{\node [circle, fill, inner sep=1.3pt]}
    \draw [line width=0.45mm] (0,0)--(0,16)--(16,16)--(16,0)--(0,0);
    \draw (0,16)--(0,6)--(10,6)--(10,4)--(12,4)--(12,0);
    \filldraw[line width=0.25mm, draw=black, fill=gray!10] (10,6)--(0,6)--(0,16)--(16,16)--(16,0)--(12,0)--(12,12)--(4,12)--(4,10)--(10,10)--(10,6);
    \draw [dotted] (0,16)--(16,0);
    
    \draw(6,12)--(6,10);
    \draw (10,6)--(12,6);
    \draw (0,16)--(4,16)--(4,12)--(0,12);
    \draw (6,10)--(10,10)--(10,6)--(6,6)--(6,10);
    \draw (12,4)--(16,4);
    \Node at (4.5,11.5){};
    \Node at (4.5,10.5){};
    \Node at (5.5,11.5){};
    
    \Node at (10.5,4.5){};
    \Node at (11.5,5.5){};
    \Node at (11.5,4.5){};
    
    \Node at (10.5,10.5){};
    \Node at (11.5,11.5){};
    
    \Node at (0.5,15.5){};
    \Node at (0.5,14.5){};
    \Node at (0.5,13.5){};
    \Node at (0.5,12.5){};
    \Node at (1.5,15.5){};
    \Node at (1.5,14.5){};
    \Node at (1.5,13.5){};
    \Node at (2.5,15.5){};
    \Node at (2.5,14.5){};
    \Node at (3.5,15.5){};
    
    \Node at (6.5,6.5){};
    \Node at (7.5,7.5){};
    \Node at (7.5,6.5){};
    \Node at (8.5,8.5){};
    \Node at (8.5,7.5){};
    \Node at (8.5,6.5){};
    \Node at (9.5,9.5){};
    \Node at (9.5,8.5){};
    \Node at (9.5,7.5){};
    \Node at (9.5,6.5){};
    
    \Node at (12.5,0.5){};
    \Node at (13.5,1.5){};
    \Node at (13.5,0.5){};
    \Node at (14.5,2.5){};
    \Node at (14.5,1.5){};
    \Node at (14.5,0.5){};
    \Node at (15.5,3.5){};
    \Node at (15.5,2.5){};
    \Node at (15.5,1.5){};
    \Node at (15.5,0.5){};
    
    \Node at (15.5,15.5){};
    \Node at (14.5,14.5){};
    \Node at (13.5,13.5){};
    \Node at (12.5,12.5){};
    
    \Node at (0.5,6.5){};
    \Node at (1.5,7.5){};
    \Node at (2.5,8.5){};
    \Node at (3.5,9.5){};
\end{tikzpicture}
\hspace{4em}
\begin{tikzpicture}[scale=0.2]
    \def\Node{\node [circle, fill, inner sep=1.3pt]}
    \draw [line width=0.45mm] (0,0)--(0,16)--(16,16)--(16,0)--(0,0);
    \draw (0,16)--(0,6)--(10,6)--(10,4)--(12,4)--(12,0);
    \filldraw[line width=0.25mm, draw=black, fill=gray!10] (10,6)--(0,6)--(0,16)--(16,16)--(16,0)--(12,0)--(12,12)--(4,12)--(4,10)--(10,10)--(10,6);
    \draw [dotted] (0,16)--(16,0);
    
    \draw(6,12)--(6,10);
    \draw (10,6)--(12,6);
    \draw (0,16)--(4,16)--(4,12)--(0,12);
    \draw (6,10)--(10,10)--(10,6)--(6,6)--(6,10);
    \draw (12,4)--(16,4);
    \Node at (4.5,11.5){};
    \Node at (4.5,10.5){};
    \Node at (5.5,11.5){};
    
    \Node at (10.5,4.5){};
    \Node at (10.5,5.5){};
    \Node at (11.5,5.5){};
    
    \Node at (10.5,11.5){};
    \Node at (11.5,10.5){};
    
    \Node at (0.5,15.5){};
    \Node at (0.5,14.5){};
    \Node at (0.5,13.5){};
    \Node at (0.5,12.5){};
    \Node at (1.5,15.5){};
    \Node at (1.5,14.5){};
    \Node at (1.5,13.5){};
    \Node at (2.5,15.5){};
    \Node at (2.5,14.5){};
    \Node at (3.5,15.5){};
    
    \Node at (6.5,6.5){};
    \Node at (6.5,7.5){};
    \Node at (6.5,8.5){};
    \Node at (6.5,9.5){};
    \Node at (7.5,7.5){};
    \Node at (7.5,8.5){};
    \Node at (7.5,9.5){};
    \Node at (8.5,8.5){};
    \Node at (8.5,9.5){};
    \Node at (9.5,9.5){};
    
    \Node at (12.5,0.5){};
    \Node at (12.5,1.5){};
    \Node at (12.5,2.5){};
    \Node at (12.5,3.5){};
    \Node at (13.5,1.5){};
    \Node at (13.5,2.5){};
    \Node at (13.5,3.5){};
    \Node at (14.5,2.5){};
    \Node at (14.5,3.5){};
    \Node at (15.5,3.5){};
    
    \Node at (15.5,12.5){};
    \Node at (14.5,13.5){};
    \Node at (13.5,14.5){};
    \Node at (12.5,15.5){};
    
    \Node at (0.5,9.5){};
    \Node at (1.5,8.5){};
    \Node at (2.5,7.5){};
    \Node at (3.5,6.5){};
\end{tikzpicture}$$
\caption{Constructed functionals $\oF^a$ and $\oF$ on $\g$ of type $\frac{10|4|2}{16}$}
\label{solution}
\end{figure}
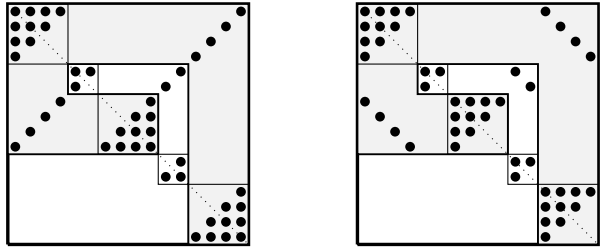
\end{ex}

\begin{theorem}
\label{Functional Construction}
Let $\g$ be a seaweed with homotopy type $H(c_1,\cdots,c_h)$, and let $F_{c_i}\in\gl(c_i)^*$ for each $i$. The functionals $\overline{F},\oF^a\in\g^*$ of Definition \ref{Functional Construction def} are such that 
\begin{equation} 
\label{dimensions}
\dim\ker(B_{\oF^a})=\dim\ker(B_{\overline{F}})=\sum_{i=1}^t\dim\ker(B_{F_{c_i}}).
\end{equation}
\end{theorem}

Assuming Theorem \ref{Functional Construction} for the moment, we have the following immediate Corollary.

\begin{theorem}
The functionals $\oF$ and $\oF^a$ in Definition \ref{Functional Construction def} are regular if and only if $F_{c_i}$ is regular on $\gl(c_i)$, for each $i$.
\end{theorem}

\begin{proof}[Proof of Theorem \ref{Functional Construction}]
The proof is an induction on the winding-down moves of Lemma \ref{winding down}.  For the entirety of this proof, let $\g$ be a seaweed with signature $S$ and $F\in\g^*$ be constructed according to Definition \ref{Functional Construction def}. Let $\g'$ be the seaweed with signature $C(c)S$. The functional constructed by Definition \ref{Functional Construction def} on $\g'$ is $F_c\oplus F$, and the dimension result of equation (\ref{dimensions}) follows trivially. Further, if $B$ is a relations matrix of $\ker(B_F)$ and $B_c$ is a relations matrix of $\ker(B_{F_c})$, then $B_c\oplus B$ is a relations matrix of $\ker(B_{F_c\oplus F})$. 

Theorem \ref{Functional Construction} follows trivially if $\g'$ is of signature $FS$. The functional constructed on $\g'$ by Definition \ref{Functional Construction def} is $F^t$, and $B^t$ is a relations matrix of $\ker(B_{F^t})$. 

Now, assume that $\g'$ has signature $BlS$ and that $\g$ is of type $\frac{a_1|\cdots|a_m}{b_1|\cdots|b_t}$ (the Rotation Contraction move and the Pure Contraction move only require an appropriate relabeling of indices). Without loss of generality, assume that the block $a_1$ in the meander $M(\g)$ associated with $\g$ is part of a single component of size $a_1$ -- the argument for multiple components is a finite number of arguments identical to the following argument. Given $F'$ constructed by Definition \ref{Functional Construction def}, the functional $F$ must be $\underset{i,j>a_1}{\sum_{(i,j)\in\mathscr{I}_{F'}}}e_{i,j}^*$. Let $B'$ be a relations matrix of $F'$ and consider the following division of $B'$ into four quadrants, whose indices $(i,j)$ are relabeled as indicated.

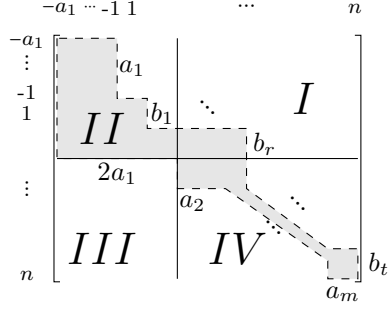
\begin{figure}[H]
$$\begin{tikzpicture}[scale=0.8]
\draw (-1.95,2.05)--(-2.05,2.05)--(-2.05,-2.05)--(-1.95,-2.05);
\draw (2.95,2.05)--(3.05,2.05)--(3.05,-2.05)--(2.95,-2.05);
\filldraw[draw=black, fill=gray!20, dashed] (-2,2)--(-2,0)--(0,0)--(0,-0.5)--(0.8,-0.5)--(2.5,-1.7)--(2.5,-2)--(3,-2)--(3,-1.5)--(2.5,-1.5)--(1.15,-0.5)--(1.15,0.5)--(-0.5,0.5)--(-0.5,1)--(-1,1)--(-1,2)--(-2,2);
\node at (-1,-0.25){$2a_1$};
\node at (0.25,-0.75){$a_2$};
\node at (1.6,-1.15){$\ddots$};
\node at (2.75,-2.25){$a_m$};
\node at (-0.75,1.5){$a_1$};
\node at (-0.25,0.75){$b_1$};
\node at (0.5,0.85){$\ddots$};
\node at (1.45,0.25){$b_r$};
\node at (2,-0.75){$\ddots$};
\node at (3.35,-1.75){$b_t$};
\node at (2.1,1){\LARGE $I$};
\node at (-1.25,0.5){\LARGE $II$};
\node at (-1.25,-1.5){\LARGE $III$};
\node at (1,-1.5){\LARGE $IV$};

\draw (-2,0)--(3,0);
\draw (0,-2)--(0,2);

\node at (-2.5,1.95){\footnotesize $-a_1$};
\node at (-2.5,1.6){\footnotesize $\vdots$};
\node at (-2.5,1.1){\footnotesize -1};
\node at (-2.5,0.75){\footnotesize 1};
\node at (-2.5,-0.5){\footnotesize $\vdots$};
\node at (-2.5,-1.95){\footnotesize $n$};

\node at (-1.95,2.5){\footnotesize{$-a_1$}};
\node at (-1.45,2.5){\tiny{$\cdots$}};
\node at (-1.1,2.5){\footnotesize{-1}};
\node at (-0.75,2.5){\footnotesize{1}};
\node at (1.15,2.5){\footnotesize{$\cdots$}};
\node at (2.95,2.5){\footnotesize{$n$}};
\end{tikzpicture}$$
\caption{The four quadrants of $B'$}
\label{quadrants}
\end{figure}

Assume $F'=\oF$. Now, fix $i,j\in[1,n]$ and consider the images of the basis elements $e_{-i,j}$, $e_{i,j}$, and $e_{i,-j}$ (basis elements in quadrant $II$). We get the following three expressions under the map $B_{F'}([B,\cdot])$.
\begin{equation}
    \label{--}
    e_{-i,-j}\mapsto\left(\sum_{(s,-j)\in\mathscr{I}_{F'}}b_{s,-i}-\sum_{(-i,s)\in\mathscr{I}_{F'}}b_{-j,s}\right)
\end{equation}
\begin{equation}
    \label{+-}
    e_{i,-j}\mapsto\left(\sum_{(s,-j)\in\mathscr{I}_{F'}}b_{s,i}-\sum_{(i,s)\in\mathscr{I}_{F'}}b_{-j,s}\right)
\end{equation}
\begin{equation} 
    \label{++}
    e_{a_1+1-i,a_1+1-j}\mapsto\left(\sum_{(s,a_1+1-j)\in\mathscr{I}_{F'}}b_{s,a_1+1-i}-\sum_{(a_1+1-i,s)\in\mathscr{I}_{F'}}b_{a_1+1-j,s}\right)
\end{equation}
Consider the equations provided by setting the right hand side of (\ref{+-}) equal to zero. Note that $(s,-j)\in\mathscr{I}_{F'}$ if and only if $s=a_1+1-j$ or $s<0$. If $s<0$, then $(s,i)\not\in\mathscr{I}_{\g'}$, so $b_{s,i}=0$. Similarly, $(i,s)\in\mathscr{I}_{F'}$ if and only if $s=i-a_1-1$ or $s>0$. If $s>0$, then $b_{-j,s}=0$ as $(-j,s)\not\in\mathscr{I}_{\g'}$. Therefore, the system of equations resulting from (\ref{+-}) reduces to 
\begin{equation}
    \label{maps}
    b_{a_1+1-j,i}=b_{-j,i-a_1-1}.
\end{equation} 
That is, the top $a_1\times a_1$ block of $B'$ is equal to the second $a_1\times a_1$ block of $B'$. An identical argument on the basis elements $e_{i,-j}$ mapped under $B_{\oF^a}([B,\cdot])$ shows that $b_{-i,-j}=b_{i,j}$, for all $i,j\leq n$, meaning the top $a_1\times a_1$ block of $B'$ is the rotation of the second $a_1\times a_1$ block of $B'$. To show that the application of a Block Elimination move to the functional does not change the dimension of the kernel, it suffices to show that elements in the peak block $[1,a_1]\times[-1,-a_1]$ indicated in Figure \ref{quadrants} are zero. We will show that the elements in the peaks must be defined in terms of previous peak blocks (if any). Therefore, by recursion, it will suffice to consider a seaweed of the form $\frac{2n+m}{n|a_1|\cdots|a_k|n}$, with $\sum a_i=m$ (i.e., the outer most peak block created in a component of size $n$ in $\g$). The recursion on the peak blocks is justified by evaluating the right hand side of (\ref{--}) and (\ref{++}) at zero and summing. Notice that $(s,-j)\in\mathscr{I}_{F'}$ if and only if $s=a_1+1-j$ or $s<0$ (i.e. $(s,-j)$ is one of the copied indices), and $(-i,s)\in\mathscr{I}_{F'}$ implies $s<0$. Therefore, combining (\ref{maps}) with the equation given by evaluating (\ref{--}) at zero is equivalent to the system

$$\underset{s<0}{\sum_{(s,-j)\in\mathscr{I}_{F'}}}b_{s,-i}+b_{a_1+1-j,-i}=\underset{s<0}{\sum_{(-i,s)\in\mathscr{I}_{F'}}}b_{-j,s}$$

\begin{equation}
\label{1}
    \Leftrightarrow \underset{s>0}{\sum_{(s,a_1+1-j)\in\mathscr{I}_{F'}}}b_{s,a_1+1-i}+b_{a_1+1-j,-i}=\underset{s>0}{\sum_{(a_1+1-i,s)\in\mathscr{I}_{F'}}}b_{a_1+1-j,s}.
\end{equation}

\noindent
When evaluated at zero, the right hand side of (\ref{++}) is equivalent to the system:
\begin{equation} 
\label{2}
\underset{s>0}{\sum_{(s,a_1+1-j)\in\mathscr{I}_{F'}}}b_{s,a_1+1-i}=\underset{0<s\leq a_1}{\sum_{(a_1+1-i,s)\in\mathscr{I}_{F'}}}b_{a_1+1-j,s}+\underset{a_1<s}{\sum_{(a_1+1-i,s)\in\mathscr{I}_{F'}}}b_{a_1+1-j,s}+b_{a_1+1-j,-i}.
\end{equation}

\noindent
Combining equations (\ref{1}) and (\ref{2}), so that the appropriate summations cancel, we have the following equation:
\begin{equation}
\label{peaks}
    -2b_{a_1+1-j,-i}=\underset{a_1<s}{\sum_{(a_1+1-i,s)\in\mathscr{I}_{F'}}}b_{a_1+1-j,s}.
\end{equation}

The same argument for $\oF^a$ yields an equation similar to equation (\ref{peaks}). Without loss of generality, consider the seaweed $\g$ of type $\frac{2n+m}{n|m|n}$. Let $G_m\in\gl(m)^*$. Define $\oF=(F_n\oplus G_m\oplus F_n)+\sum_{i=1}^ne_{n+m+i,i}^*$ and $\oF^a=(F_c\oplus G_m\oplus F_c^R)+\sum_{i=1}^ne_{2n+m+1-i,i}^*.$ The indices in these functionals are pictured in Figure \ref{basis step}.

\begin{figure}[H]
$$\begin{tikzpicture}[scale=0.4]
    \def\Node{\node [circle, fill, inner sep=1.5pt]}
    \draw (0,0)--(10,0)--(10,10)--(0,10)--(0,0);
    \draw [line width=0.45mm, fill=gray!20] (0,0)--(10,0)--(10,4)--(6,4)--(6,6)--(4,6)--(4,10)--(0,10)--(0,0);
    \draw (0,6)--(4,6)--(4,0);
    \draw (0,4)--(6,4)--(6,0);
    \node [label=left:{$n$}] at (0,8){};
    \node [label=left:{$n$}] at (0,2){};
    \node [label=below:{$n$}] at (2,0){};
    \node [label=below:{$n$}] at (8,0){};
    \node [label=left:{$m$}] at (0,5){};
    \node [label=below:{$m$}] at (5,0){};
    \node at (2,8){$\mathscr{I}_{F_n}$};
    \node at (8,2){$\mathscr{I}_{F_n}$};
    \node at (5,5){$\mathscr{I}_{G_m}$};
    \Node at (0.5,3.5){};
    \Node at (1.5,2.5){};
    \node at (2.5,1.75){\rotatebox{-10}{$\ddots$}};
    \Node at (3.5,0.5){};
\end{tikzpicture}
\hspace{2em}
\begin{tikzpicture}[scale=0.4]
    \def\Node{\node [circle, fill, inner sep=1.5pt]}
    \draw (0,0)--(10,0)--(10,10)--(0,10)--(0,0);
    \draw [line width=0.45mm, fill=gray!20] (0,0)--(10,0)--(10,4)--(6,4)--(6,6)--(4,6)--(4,10)--(0,10)--(0,0);
    \draw (0,6)--(4,6)--(4,0);
    \draw (0,4)--(6,4)--(6,0);
    \node [label=left:{$n$}] at (0,8){};
    \node [label=left:{$n$}] at (0,2){};
    \node [label=below:{$n$}] at (2,0){};
    \node [label=below:{$n$}] at (8,0){};
    \node [label=left:{$m$}] at (0,5){};
    \node [label=below:{$m$}] at (5,0){};
    \node at (2,8){$\mathscr{I}_{F_n}$};
    \node at (8,2){$\mathscr{I}_{F_n^R}$};
    \node at (5,5){$\mathscr{I}_{G_m}$};
    \Node at (0.5,0.5){};
    \Node at (1.5,1.5){};
    \node at (2.35,2.5){\rotatebox{80}{$\ddots$}};
    \Node at (3.5,3.5){};
\end{tikzpicture}$$
\caption{Indices in $\mathscr{I}_{F'}$ \textup(left\textup) and $\mathscr{I}_{\oF^a}$ \textup(right\textup)}
\label{basis step}
\end{figure}
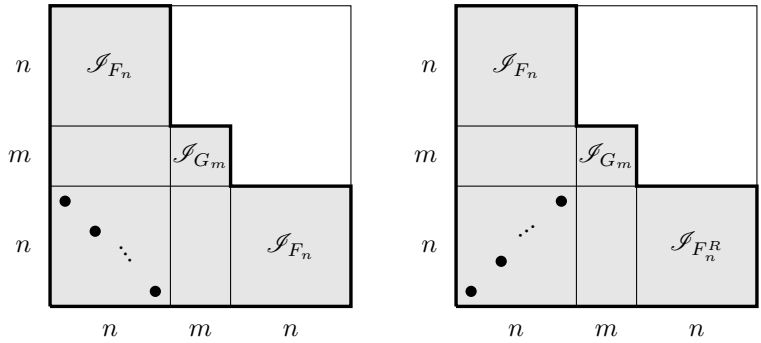

\noindent 
For $\oF$, consider the images of the basis elements under $B_{F'}(B,\cdot)$. We have the following cases:
\begin{enumerate}
    \item $e_{i,j}\mapsto \sum_{(s,j)\in\mathscr{I}_{F_n}}b_{s,i}+b_{n+m+j,i}-\sum_{(i,s)\in\mathscr{I}_{F_n}}b_{j,s}$, \\
    for $i,j\in[1,n]$,
    \item $e_{i,j}\mapsto\sum_{(s,j)\in\mathscr{I}_{F_n}}b_{s,i}+b_{n+m+j,i}-\sum_{(i-n,s-n)\in\mathscr{I}_{G_m}}b_{j,s}$, \\
    for $i\in[n+1,m]$, $j\in[1,n]$,
    \item $e_{i,j}\mapsto \sum_{(s-n,j-n)\in\mathscr{I}_{G_m}}b_{s,i}-\sum_{(i-n,s-n)\in\mathscr{I}_{G_m}}b_{j,s}$, \\
    for $i,j\in [n+1,m]$,
    \item $e_{i,j}\mapsto\sum_{(s,j)\in\mathscr{I}_{F_n}}b_{s,i}+b_{n+m+j,i}-\sum_{(i-n,m,s-n-m)\in\mathscr{I}_{F_n}}b_{j,s}-b_{j,i-n-m}$, \\
    for $i\in[n+m+1,2n+m]$, $j\in [1,n]$,
    \item $e_{i,j}\mapsto\sum_{(s-n,j-n)\in\mathscr{I}_{G_m}}b_{s,i}-\sum_{(i-n,m,s-n-m)\in\mathscr{I}_{F_n}}b_{j,s}-b_{j,i-n-m}$, \\
    for $i\in[n+m+1,2n+m]$, $j\in [n+1,m]$,
    \item $e_{i,j}\mapsto\sum_{(s-n-m,j-n-m)\in\mathscr{I}_{F_n}}b_{s,i}-\sum_{(i-n,m,s-n-m)\in\mathscr{I}_{F_n}}b_{j,s}-b_{j,i-n-m}$, \\
    for $i,j\in[n+m+1,2n+m]$.
\end{enumerate}
The expressions in Case 1, when evaluated at zero, combine with the expressions in Case 6 evaluated at zero to show that $b_{j,i}=0$, for all $j\in[n+m+1,2n+m]$, $i\in[1,n]$. The equations generated by evaluating the Case 4 expressions at zero yield that $b_{i,j}=b_{i+n+m,j+n+m}$, for all $i,j\in[1,n]$. The expressions in Case 1, knowing that $b_{n+m+j,i}=0$, solve to a relations matrix $B_1$ of $\ker(B_{F_n})$. 

In Case 2, both summations are zero as the indices are not in $\mathscr{I}_{\g}$, and so the equation which results from evaluating the right hand side at zero simplifies to $b_{n+m+j,i}=0$, for all $j\in[1,n]$ and $i\in[n+1,m]$. The same argument on the expressions in Case 5 generates the $m\times n$ rectangle in the first column of Figure \ref{basis step} \textup(left\textup) must also be a zero matrix. Finally, upon evaluating the expressions in Case 3 at zero, we obtain a relations matrix $B_2$ of $\ker(B_{G_m})$. 

The final evaluation of the system of equations given by mapping the basis elements to zero under $B_{\oF}([B,\cdot])$ will be that a relations matrix $B$ of $\ker(B_{\oF})$ is 

$$B=B_1\oplus B_2\oplus B_1.$$ 

\noindent 
Hence, $\dim\ker(B_{F'})=\dim\ker(B_{F_n})+\dim\ker(B_{G_m})$. A similar argument on $\oF^a$ shows that a relations matrix $B'$ for $\ker(B_{\oF^a})$ is of the form

$$B'=B_1\oplus B_2\oplus B_1^R,$$ 

\noindent 
and the dimension claim holds. 
\end{proof}

We have the following immediate Corollary from the proof of Theorem \ref{Functional Construction}.

\begin{theorem}
\label{kernel}
Let $\g$ be a seaweed with homotopy type $H(c_1,\cdots,c_h)$, and let $F_{c_i}\in\gl(c_i)^*$ for each $i$. Any functional $F$ constructed via the methods of Definition \ref{Functional Construction def} is such that, if $C_{c_{i_1}}\oplus C_{c_{i_2}}\oplus\cdots\oplus C_{c_{i_k}}$ is the core of $\g$ \textup(necessarily, $i_j\in\{1,\cdots,h\}$, for all $j$\textup), then a relations matrix of the space $\ker(B_F)$ is 

$$B_{i_1}^r\oplus\cdots\oplus B_{i_k}^r,$$

\noindent 
where the superscript, $r$, represents the appropriate rotation of the matrix where necessary.
\end{theorem}

\begin{ex}
\label{mixmatch}
By direct computation \textup(as demonstrated in the proof of Theorem \ref{My Theorem 1}\textup), $F_2$ and $F_4$ are regular on $\gl(2)$ and $\gl(4)$ respectively, and they have the respective relations matrices:

$$B_4=\left(\begin{array}{cccc}
b_1+b_2+b_3+b_4&b_1+b_2+b_3&b_1+b_2&b_1\\
b_1+b_2+b_3&b_1+b_2+b_4&b_1+b_3&b_2\\
b_1+b_2&b_1+b_3&b_2+b_4&b_3\\
b_1&b_2&b_3&b_4
\end{array}\right),\hspace{1em}
B_2=\left(\begin{array}{cc}b_5+b_6&b_5\\b_5&b_6\end{array}\right).$$ 

\noindent 
Let $F$ be the functional whose indices are illustrated in Figure \ref{mixmatch fig}.
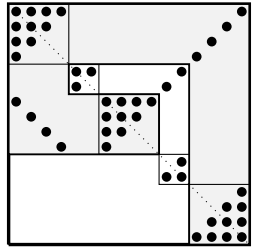
\begin{figure}[H]
$$\begin{tikzpicture}[scale=0.2]
    \def\Node{\node [circle, fill, inner sep=1.3pt]}
    \draw [line width=0.45mm] (0,0)--(0,16)--(16,16)--(16,0)--(0,0);
    \draw (0,16)--(0,6)--(10,6)--(10,4)--(12,4)--(12,0);
    \filldraw[line width=0.25mm, draw=black, fill=gray!10] (10,6)--(0,6)--(0,16)--(16,16)--(16,0)--(12,0)--(12,12)--(4,12)--(4,10)--(10,10)--(10,6);
    \draw [dotted] (0,16)--(16,0);
    
    \draw(6,12)--(6,10);
    \draw (10,6)--(12,6);
    \draw (0,16)--(4,16)--(4,12)--(0,12);
    \draw (6,10)--(10,10)--(10,6)--(6,6)--(6,10);
    \draw (12,4)--(16,4);
    \Node at (4.5,11.5){};
    \Node at (4.5,10.5){};
    \Node at (5.5,11.5){};
    
    \Node at (10.5,4.5){};
    \Node at (11.5,5.5){};
    \Node at (11.5,4.5){};
    
    \Node at (10.5,10.5){};
    \Node at (11.5,11.5){};
    
    \Node at (0.5,15.5){};
    \Node at (0.5,14.5){};
    \Node at (0.5,13.5){};
    \Node at (0.5,12.5){};
    \Node at (1.5,15.5){};
    \Node at (1.5,14.5){};
    \Node at (1.5,13.5){};
    \Node at (2.5,15.5){};
    \Node at (2.5,14.5){};
    \Node at (3.5,15.5){};
    
    \Node at (6.5,6.5){};
    \Node at (6.5,7.5){};
    \Node at (6.5,8.5){};
    \Node at (6.5,9.5){};
    \Node at (7.5,7.5){};
    \Node at (7.5,8.5){};
    \Node at (7.5,9.5){};
    \Node at (8.5,8.5){};
    \Node at (8.5,9.5){};
    \Node at (9.5,9.5){};
    
    \Node at (12.5,0.5){};
    \Node at (13.5,1.5){};
    \Node at (13.5,0.5){};
    \Node at (14.5,2.5){};
    \Node at (14.5,1.5){};
    \Node at (14.5,0.5){};
    \Node at (15.5,3.5){};
    \Node at (15.5,2.5){};
    \Node at (15.5,1.5){};
    \Node at (15.5,0.5){};
    
    \Node at (15.5,15.5){};
    \Node at (14.5,14.5){};
    \Node at (13.5,13.5){};
    \Node at (12.5,12.5){};
    
    \Node at (0.5,9.5){};
    \Node at (1.5,8.5){};
    \Node at (2.5,7.5){};
    \Node at (3.5,6.5){};
\end{tikzpicture}$$
\caption{Indices in $\mathscr{I}_F$}
\label{mixmatch fig}
\end{figure}
A relations matrix of $\ker(B_F)$ is $B_4\oplus B_2\oplus B_4\oplus B_2^R\oplus B_4^R$.
\end{ex}
\end{subsection}

\begin{subsection}{An explicit regular functional on $\gl(n)$}
\label{An Explicit Regular Functionals on gl(n)}

The purpose of this section is to provide an explicit regular functional $F_n$ on $\gl(n)$. We will leverage this construction together with Definition \ref{Functional Construction def} to construct regular fuctionals on any seaweed subalgebra $\g\subseteq\gl(n)$ with homotopy type $H(c_1,\cdots,c_h)$ by embedding the functionals $F_{c_i}$ appropriately. 


\begin{theorem}
\label{My Theorem 1}
The functional $F_n=\sum_{i=1}^{n}\sum_{j=1}^{n+1-i}e_{i,j}^*$ is regular on $\gl(n)$.
\end{theorem}

The indices in $\mathscr{I}_{F_n}$ are illustrated in Figure \ref{F_n inds} as the grey region and solid lines.
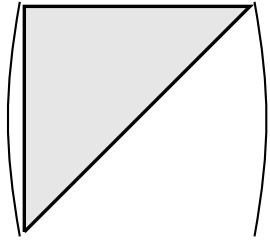
\begin{figure}[H]
$$\begin{tikzpicture}[scale=0.6]
    \filldraw[draw=black, fill=gray!20, line width=0.45mm] (-0.5,-0.5)--(4.5,4.5)--(-0.5,4.5)--(-0.5,-0.5);
    
    \draw [thick] (-0.6,-0.6) to[bend left=10] (-0.6,4.6);
    \draw [thick] (4.6,-0.6) to[bend right=10] (4.6,4.6);    
\end{tikzpicture}$$
\caption{Indices in $\mathscr{I}_{F_n}$}
\label{F_n inds}
\end{figure}

\begin{proof}  See Appendix A.  
\end{proof}
\end{subsection}

\begin{subsection}{Additional regular functionals on $\gl(n)$}
\label{Three More Regular Functionals}

We define the {\textit{size}} of a functional $F$ to be equal to $|\mathscr{I}_F|$. Computationally, smaller is better. We provide seven additional regular functionals 
$H_n$, $H_n'$, $K_n$, $K_n'$, $G_n$, $G_n'$, and $F_n'$ all of which are based on $F_n$. Their relative sizes are:

$$
|\mathscr{I}_{F_n'}|\leq |\mathscr{I}_{G_n}|=|\mathscr{I}_{G_n'}| \leq |\mathscr{I}_{H_n}| = |\mathscr{I}_{H_n'}| \leq |\mathscr{I}_{K_n}|=|\mathscr{I}_{K_n'}|\leq |\mathscr{I}_{F_n}|.
$$

The smallest of these functionals, $F_n'$, is smaller than $F_n$ by $2n-1$ terms (i.e., $|\mathscr{I}_{F_n}|=|\mathscr{I}_{F_n'}|+2n-1$). The proofs of these seven functionals are closely related and all require the regularity of $F_n$ as an inductive hypothesis. These proofs are omitted from this paper, but detailed in \textbf{[10]}.

\begin{theorem}
\label{My Theorem 2}
The functional $G_n=e_{1,1}^*+\sum_{i=2}^{n-1}\sum_{j=2}^{n+1-i}e_{i,j}^*=e_{1,1}^*\oplus F_{n-2}$ is regular on $\gl(n)$ for $n\geq 4$.
\end{theorem}

As a visual aid, the indices of the functional of $G_n$ of Theorem \ref{My Theorem 2} is the functional such that $\mathscr{I}_{G_n}$ is the set of indices illustrated by the grey region in Figure \ref{Kn} \textup(left\textup). As immediate corollaries to the proof of Theorem \ref{My Theorem 2}, we have two more regular functionals on $\gl(n)$.

\begin{theorem}
\label{Functional 3}
The functional $H_n=\sum_{i=1}^{n-1}\sum_{j=1}^{n-i}e_{i,j}^*=F_{n-1}$ is regular on $\gl(n)$.
\end{theorem}

As a visual aid, the indices in $\mathscr{I}_{H_n}$ are illustrated in Figure \ref{Kn} (center).

\begin{theorem}
\label{Functional 4}
The functional $K_n=e_{1,1}^*+\sum_{i=2}^n\sum_{j=2}^{n+2-i}e_{i,j}=G_{n+1}$ is regular on $\gl(n)$.
\end{theorem}

As a visual aid, the indices in $\mathscr{I}_{K_n}$ are illustrated in Figure \ref{Kn} \textup(right\textup).
\begin{figure}[H]
$$\begin{tikzpicture}[scale=0.54]
    \filldraw[draw=black, fill=gray!20, line width=0.45mm] (0,0)--(4,4)--(0,4)--(0,0);
    \filldraw[draw=black, fill=gray!20, line width=0.45mm] (-0.1,4.1)--(-0.1,4.5)--(-0.5,4.5)--(-0.5,4.1)--(-0.1,4.1);
    
    \draw [thick] (-0.6,-0.6) to[bend left=10] (-0.6,4.6);
    \draw [thick] (4.6,-0.6) to[bend right=10] (4.6,4.6);
    
    \node at (-0.3,3.7){0};
    \node at (-0.3,3){0};
    \node at (-0.3,2.3){0};
    \node at (-0.3,1.65){$\vdots$};
    \node at (-0.3,1.1){0};
    \node at (-0.3,0.4){0};
    \node at (-0.3,-0.3){0};
    
    \node at (4.3,3.7){0};
    \node at (4.3,3){0};
    \node at (4.3,2.3){0};
    \node at (4.3,1.65){$\vdots$};
    \node at (4.3,1.1){0};
    \node at (4.3,0.4){0};
    \node at (4.3,-0.3){0};
    
    \node at (0.3,4.3){0};
    \node at (1,4.3){0};
    \node at (1.7,4.3){0};
    \node at (2.35,4.3){$\cdots$};
    \node at (2.9,4.3){0};
    \node at (3.6,4.3){0};
    \node at (4.3,4.3){0};
    
    \node at (0.3,-0.3){0};
    \node at (1,-0.3){0};
    \node at (1.7,-0.3){0};
    \node at (2.35,-0.3){$\cdots$};
    \node at (2.9,-0.3){0};
    \node at (3.6,-0.3){0};
    \node at (4.3,-0.3){0};
    
    \node at (3,1){\Huge{0}};
\end{tikzpicture}
\hspace{4em}
\begin{tikzpicture}[scale=0.6]
    \filldraw[draw=black, fill=gray!20, line width=0.45mm] (0,0)--(4,4)--(0,4)--(0,0);
    
    \draw [thick] (-0.1,-0.6) to[bend left=10] (-0.1,4.1);
    \draw [thick] (4.6,-0.6) to[bend right=10] (4.6,4.1);
    
    \node at (4.3,3.7){0};
    \node at (4.3,3){0};
    \node at (4.3,2.3){0};
    \node at (4.3,1.65){$\vdots$};
    \node at (4.3,1.1){0};
    \node at (4.3,0.4){0};
    \node at (4.3,-0.3){0};
    
    \node at (0.3,-0.3){0};
    \node at (1,-0.3){0};
    \node at (1.7,-0.3){0};
    \node at (2.35,-0.3){$\cdots$};
    \node at (2.9,-0.3){0};
    \node at (3.6,-0.3){0};
    \node at (4.3,-0.3){0};
    
    \node at (3,1){\Huge{0}};
\end{tikzpicture}
\hspace{4em}
\begin{tikzpicture}[scale=0.6]
    \filldraw[draw=black, fill=gray!20, line width=0.45mm] (0,0)--(4,4)--(0,4)--(0,0);
    \filldraw[draw=black, fill=gray!20, line width=0.45mm] (-0.1,4.1)--(-0.1,4.5)--(-0.5,4.5)--(-0.5,4.1)--(-0.1,4.1);
    
    \draw [thick] (-0.6,-0.1) to[bend left=10] (-0.6,4.6);
    \draw [thick] (4.1,-0.1) to[bend right=10] (4.1,4.6);
    
    \node at (-0.3,3.7){0};
    \node at (-0.3,3){0};
    \node at (-0.3,2.3){0};
    \node at (-0.3,1.65){$\vdots$};
    \node at (-0.3,1.1){0};
    \node at (-0.3,0.4){0};
    
    \node at (0.3,4.3){0};
    \node at (1,4.3){0};
    \node at (1.7,4.3){0};
    \node at (2.35,4.3){$\cdots$};
    \node at (2.9,4.3){0};
    \node at (3.6,4.3){0};
    
    \node at (3,1){\Huge{0}};
\end{tikzpicture}$$
\caption{Indices in $\mathscr{I}_{G_n}$ \textup(left\textup), $\mathscr{I}_{H_n}$ (center), and $\mathscr{I}_{K_n}$ \textup(right\textup)}
\label{Kn}
\end{figure}
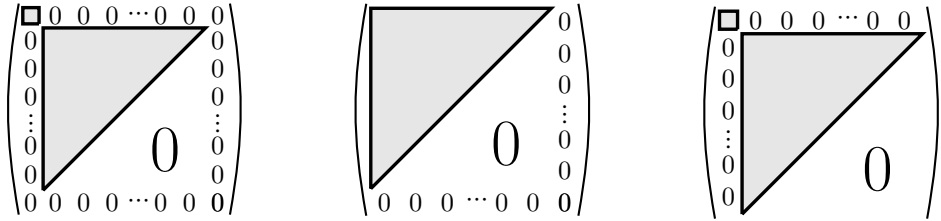

Through an identical linear algebra argument to the one constructed in the proof of Theorem \ref{My Theorem 2}, we get four more functionals.
\begin{theorem}
\label{final group of regulars I swear}
The functionals $$G_n'=0\oplus F_{n-2}\oplus e_{1,1}^*,\hspace{4em}K_n'=F_{n-1}\oplus e_{1,1}^*,$$ $$H_n'=0\oplus F_{n-1},\hspace{2em}\text{ and }\hspace{2em}F_n'=0\oplus F_{n-2}\oplus 0$$ are regular on $\gl(n)$.
\end{theorem}

Let $f$ be a functional of Theorem \ref{final group of regulars I swear}. A relations matrices for $\ker(B_f)$ is an appropriate direct sum of relations matrices of $\ker(B_{F_{n-1}})$ and $\ker(B_{F_{n-2}})$, and the matrix $(b_i)$.

The transposition across the antidiagonal of any functional defined in this section and in section \ref{An Explicit Regular Functionals on gl(n)} is also regular on $\gl(n)$. Note that $F_n'$ is the smallest of the eight regular functionals thus far constructed.
\end{subsection}

\section{Regular functionals on simple Lie algebras}
\label{Semisimple Cases}

In this section, we transition from building regular functionals on seaweed subalgebras of $\gl(n)$ to building regular functionals on seaweed subalgebras of the classical Lie algebras $A_n=\mf{sl}(n+1)$ and $C_n=\mf{sp}(2n)$.




\begin{subsection}{Type-$A$ seaweeds}
\label{type A}



Seaweed subalgebras of $A_n$ are constructed in the same way as seaweeds in $\gl(n+1)$, but satisfy an additional algebraic constraint -- they have trace zero. 

\begin{definition}
Let $n$ be an integer and let $(a_1,\cdots,a_m)$ and $(b_1,\cdots,b_t)$ be two compositions of $n+1$. If the seaweed of type $\frac{a_1|\cdots|a_m}{b_1|\cdots|b_t}$ is further required to have trace zero, then this seaweed is said to be of Type-$A$ and is denoted $\p_n^A\;\frac{a_1|\cdots|a_m}{b_1|\cdots|b_t}$.
\end{definition}



\end{subsection}

To begin, we must know how restricting to algebras of trace zero affects the index of a seaweed.

\begin{theorem}[Dergachev and A. Kirillov, \textbf{[8]}]
\label{index A}
If $\g$ is a seaweed of type $\p_n^A\frac{a_1|\cdots|a_m}{b_1|\cdots|b_t}$, then 

$$\ind\g=2C+P-1,$$

\noindent 
 where $C$ is the number of cycles and $P$ is the number of paths and isolated points in the meander associated with $\g$.
\end{theorem}

We have the following immediate Corollary.

\begin{theorem}
\label{index of A_n}
The Lie algebra $A_n$ has index $n$.
\end{theorem}

\begin{theorem}
\label{functional on An}
The functional $F_n=\sum_{i=1}^n\sum_{j=1}^{n+1-i}e_{i,j}^*$ of Theorem \ref{My Theorem 1} is regular on $A_n$.
\end{theorem}

\begin{proof}
The proof will consist of showing the equivalence of the systems of equations for $B_{F_n}(B,b)$ and $B_{F_n}(B,b')$ (for the sets of basis elements $b$ of $\gl(n)$ and basis elements $b'$ of $\sl(n+1)$).

For all $i\neq j$ with $i,j\leq n$, we have 
\begin{equation}
\label{eqa}
    e_{i,j}\mapsto\left(\sum_{s=1}^{n+1-j}b_{s,i}-\sum_{s=1}^{n+1-i}b_{j,s}\right).
\end{equation} 
The system of equations which results when the expressions in (\ref{eqa}) are evaluated at zero is identical to the system of equations for the respective images for these basis elements in $\gl(n)$. Now, consider the basis elements $e_{i,i}-e_{i+1,i+1}$ with $i\leq n$. By requiring $B_{F_n}(B,e_{i,i}-e_{i+1,i+1})=0$, we get the weaker condition that $B_{F_n}(B,e_{i,i})=B_{F_n}(B,e_{i+1,i+1})$, for all $i$. However, we have

$$B_{F_n}(B,e_{n+1,n+1})=0.$$

\noindent 
Therefore, the system of equations for the $n^2$ basis elements of $A_n$ noted in Equation (\ref{eqa}) is isomorphic to the system of equations on $\gl(n)$. It suffices to address the last $2n$ basis elements $e_{i,n+1}$ and $e_{n+1,i}$ for $i\neq n+1$. Consider the image of the first $n$ basis elements under $B_{F_n}([B,\cdot])$: 
\begin{equation}
\label{AAA}
    e_{i,n+1}\mapsto-\sum_{s=1}^{n+1-i}b_{n+1,s}.
\end{equation} 
By induction, we get $b_{n+1,i}=0$, for all $i\in[1,n]$. The argument is similar for $b_{i,n+1}=0$, for all $i\in[1,n]$, or Lemma \ref{zero symmetry} may be applied. The resulting relations matrix of $\ker(B_{F_n})$ on $A_n$ is $B\oplus (a)$, where $B$ is a relations matrix of $\ker(B_{F_n})$ defined on $\gl(n)$ and $a=-\sum_{i=1}^n b_{i,i}$.
\end{proof}

\noindent 
Note that $F_n$ defined on $A_n$ is the sum of functionals $e_{i,j}^*$ {strictly above} the antidiagonal.

\begin{theorem}
\label{reg A func}
If $\p_n^A \frac{a_1|\cdots|a_m}{b_1|\cdots|b_t}$ is a seaweed with homotopy type $H(c_1,\cdots,c_h)$, any funcitonal $\oF$ built using Definition \ref{Functional Construction def} with functionals $f_{c_i-1}$ embedded into the components of size $c_i$ is such that 

$$\dim\ker(B_{\oF})=-1+\sum_{i=1}^h\dim\ker(B_{f_{c_i-1}}),$$

\noindent 
where $\dim\ker(B_{\oF})$ is over $\mf{sl}(n+1)$, but $\dim\ker(B_{f_{c_i}})$ is the dimension of the kernel in $\gl(c_i)$.
\end{theorem}

\begin{proof}
The proof is similar to the proof of Theorem \ref{Functional Construction}. We apply induction on the winding-down moves of Lemma \ref{winding down} to place appropriately adjusted copies of $\ker(B_{f_{c_i}})$ into the core of $\g$. The difference in Type-$A$ is that in the first Component Creation move in the winding-up of the meander associated with $\g$, there is an index $t$ such that $b_{t,t}$ is the negative sum of the diagonal to ensure the vanishing trace condition of $A_n$ (this was $b_{n+1,n+1}$ for $B$ a relations matrix of $F_n$ in Theorem \ref{functional on An}). Under the winding-up moves, we map the functional as we did in Definition \ref{Functional Construction def}, but the kernel adjustments must be such that the sum of all instances of $b_{t,t}$ on the diagonal maintains the vanishing trace condition.
\end{proof}

This completely resolves the problem of naming regular functionals for seaweed subalgebras of Type-$A$. See Example \ref{finalex AA}.

\begin{ex}
\label{finalex AA}
 Consider $\g=\p_7^A\frac{4|4}{8}$. According to Theorem \ref{reg A func} \textup(and embedding $F_{3}$ of Theorem \ref{functional on An}\textup) yields functionals
 \vspace{-1em}
 \begin{center}
    \scalebox{0.8}{
     \begin{tabular}{c}
     $\oF^a=e_{1,1}^*+e_{1,2}^*+e_{1,3}^*+e_{1,8}^*+e_{2,1}^*+e_{2,2}^*+e_{2,7}^*+e_{3,1}^*+e_{3,6}^*+e_{4,5}^*+e_{6,8}^*+e_{7,7}^*+e_{7,8}^*+e_{8,6}^*+e_{8,7}^*+e_{8,8}^*,$ \\
     $\oF=e_{1,1}^*+e_{1,2}^*+e_{1,3}^*+e_{1,5}^*+e_{2,1}^*+e_{2,2}^*+e_{2,6}^*+e_{3,1}^*+e_{3,7}^*+e_{4,8}^*+e_{6,8}^*+e_{7,7}^*+e_{7,8}^*+e_{8,6}^*+e_{8,7}^*+e_{8,8}^*,$
     \end{tabular}
     }
 \end{center} 
 The indices for $\oF^a$ and $\oF$ are shown in Figure \ref{finalex A} \textup(left and right, respectively\textup).
 
 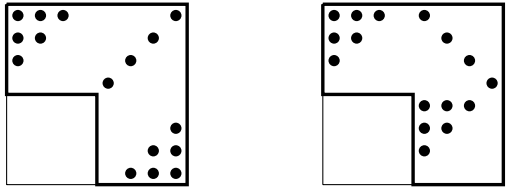
\begin{figure}[H]
$$\begin{tikzpicture}[scale=0.3]
    \def\Node{\node [circle, fill, inner sep=1.5pt]}
    \draw (0,0)--(0,8)--(8,8)--(8,0)--(0,0);
    \draw [line width=0.45mm] (0,8)--(0,4)--(4,4)--(4,0)--(8,0)--(8,8)--(0,8);
    \Node at (0.5,7.5){};
    \Node at (1.5,7.5){};
    \Node at (2.5,7.5){};
    \Node at (0.5,6.5){};
    \Node at (1.5,6.5){};
    \Node at (0.5,5.5){};
    \Node at (4.5,4.5){};
    \Node at (5.5,5.5){};
    \Node at (6.5,6.5){};
    \Node at (7.5,7.5){};
    \Node at (7.5,2.5){};
    \Node at (6.5,1.5){};
    \Node at (7.5,1.5){};
    \Node at (5.5,0.5){};
    \Node at (6.5,0.5){};
    \Node at (7.5,0.5){};
\end{tikzpicture}
\hspace{5em}
\begin{tikzpicture}[scale=0.3]
    \def\Node{\node [circle, fill, inner sep=1.5pt]}
    \draw (0,0)--(0,8)--(8,8)--(8,0)--(0,0);
    \draw [line width=0.45mm] (0,8)--(0,4)--(4,4)--(4,0)--(8,0)--(8,8)--(0,8);
    \Node at (0.5,7.5){};
    \Node at (1.5,7.5){};
    \Node at (2.5,7.5){};
    \Node at (0.5,6.5){};
    \Node at (1.5,6.5){};
    \Node at (0.5,5.5){};
    \Node at (4.5,7.5){};
    \Node at (5.5,6.5){};
    \Node at (6.5,5.5){};
    \Node at (7.5,4.5){};
    \Node at (4.5,3.5){};
    \Node at (5.5,3.5){};
    \Node at (6.5,3.5){};
    \Node at (4.5,2.5){};
    \Node at (5.5,2.5){};
    \Node at (4.5,1.5){};
\end{tikzpicture}$$
\caption{Indices $\mathscr{I}_{\oF^a}$ \textup(left\textup) and $\mathscr{I}_{\oF}$ \textup(right\textup) on $\p_7^A\frac{4|4}{8}$}
\label{finalex A}
\end{figure}
A messy calculation yields relations matrices $B$ and $B'$ of $\ker(B_{\oF^a})$ $\ker(B_{\oF})$, respectively.

$$
B=\left(\begin{array}{ccc}
b_1+b_2+b_3&b_1+b_2&b_1\\
b_1+b_2&b_1+b_3&b_2\\
b_1&b_2&b_3
\end{array}\right)
\bigoplus 
(-2b_1-b_2-3b_3)
\bigoplus 
(-2b_1-b_2-3b_3)
\bigoplus 
\left(\begin{array}{ccc}
b_3&b_2&b_1\\
b_2&b_1+b_3&b_1+b_2\\
b_1&b_1+b_2&b_1+b_2+b_3
\end{array}\right)
$$

$$
B'=\left(\begin{array}{ccc}
b_1+b_2+b_3&b_1+b_2&b_1\\
b_1+b_2&b_1+b_3&b_2\\
b_1&b_2&b_3
\end{array}\right)
\bigoplus 
(-2b_1-b_2-3b_3)
\bigoplus 
\left(\begin{array}{ccc}
b_1+b_2+b_3&b_1+b_2&b_1\\
b_1+b_2&b_1+b_3&b_2\\
b_1&b_2&b_3
\end{array}\right)
\bigoplus 
(-2b_1-b_2-3b_3)
$$

Evidently $\ind\g=3$,  it follows that $\oF^a$ and $\oF$ are regular.
\end{ex}

\begin{subsection}{Type-$C$ seaweeds}
\label{type C}

Seaweed subalgebras of $C_n$ are constructed as in $\gl(2n)$, but they must also be subalgebras of $\mf{sp}(2n)$. Because of the symmetry across the antidiagonal of $C_n$, we have a simplified notation for seaweeds of type-$C$.

\begin{definition}
Given two partial compositions $(a_1,\cdots,a_m)$ and $(b_1,\cdots,b_t)$ of $n$ \textup(i.e., $\sum_{i=1}^ma_i,\sum_{i=1}^tb_i\leq n$\textup), let $\g$ be the seaweed of type $\frac{a_1|\cdots|a_m|2(n-\sum_{i=1}^ma_i)|a_m|\cdots|a_1}{b_1|\cdots|b_t|2(n-\sum_{i=1}^tb_i)|b_t|\cdots|b_1}$ which is a subalgebra of $C_n$. This is the standard seaweed of type $\p_n^C\frac{a_1|\cdots|a_m}{b_1|\cdots|b_t}$.
\end{definition}

\begin{ex}
Consider the seaweed $\g=\p_3^C\frac{3}{2}$. This is the set of all matrices in $\mf{sp}(6)$ whose nonzero entries occur in the marked entries of Figure \ref{exC1}.
\begin{figure}[H]
$$\begin{tikzpicture}[scale=0.5]
\draw (0,0)--(0,6)--(6,6)--(6,0)--(0,0);
\draw [dotted](0,0)--(6,6);
\draw [dotted] (0,6)--(6,0);
\draw [dotted] (3,0)--(3,6);
\draw [dotted] (0,3)--(6,3);
\draw [line width=0.45mm](0,6)--(0,3)--(3,3)--(3,0)--(6,0)--(6,2)--(4,2)--(4,4)--(2,4)--(2,6)--(0,6);
\node at (0.5,5.3){\LARGE *};
\node at (0.5,4.3){\LARGE *};
\node at (0.5,3.3){\LARGE *};
\node at (1.5,5.3){\LARGE *};
\node at (1.5,4.3){\LARGE *};
\node at (1.5,3.3){\LARGE *};
\node at (2.5,3.3){\LARGE *};
\node at (3.5,3.3){\LARGE *};
\node at (3.5,2.3){\LARGE *};
\node at (3.5,1.3){\LARGE *};
\node at (3.5,0.3){\LARGE *};
\node at (4.5,1.3){\LARGE *};
\node at (4.5,0.3){\LARGE *};
\node at (5.5,1.3){\LARGE *};
\node at (5.5,0.3){\LARGE *};
\node[label=left:{3}] at (0,4.5){};
\node[label=left:{3}] at (3,1.5){};
\node[label=above:{2}] at (1,6){};
\node[label=above:{2}] at (3.5,4){};
\node[label=above:{2}] at (5,1.8){};
 \end{tikzpicture}$$
\caption{Seaweed of type $\p_3^C\frac{3}{2}$}
    \label{exC1}
\end{figure}
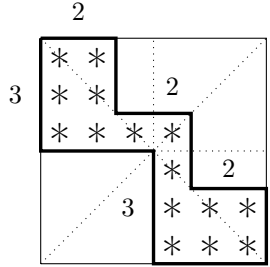
\end{ex}


To ease computations, we can leverage the symmetry across the antidiagonal of a seaweed subalgebra of $C_n$, and make use of a meander on $n$ vertices instead of the full $2n$ vertices that a seaweed subalgebra of $\gl(2n)$ would normally require. 


\begin{definition}
\label{smaller meander}
Let $\g=\p_n^C\frac{a_1|\cdots|a_m}{b_1|\cdots|b_t}$. The shortened meander associated with $\g$ \textup(denoted $M_n^C$ to differentiate it from the meander on $2n$ vertices\textup) is constructed as follows. Place $n$ vertices $v_1$ through $v_n$ in a line. Create two partitions \textup(top and bottom\textup) of the vertices based on the given partial compositions of $n$. Draw arcs in the first $M(\g)$ top blocks  and the first $t$ bottom blocks as you would a meander in $\gl(n)$. There may be vertices left over. We define the following sets: $T_a=\{v_i\;|\;i>\sum_{i=1}^ma_i\}$ and $T_b=\{v_i\;|\;i>\sum_{i=1}^tb_i\}$. The set $T_\g=(T_a\cup T_b)\backslash(T_a\cap T_b)$ is the {tail} of $M_n^C$. The {aftertail} $T_\g^a$ of the meander is $T_a\cap T_b$. See example \ref{exC}.
\end{definition}

\begin{ex}
\label{exC}
 Consider the seaweed $\g=\p_7^C\frac{1|1|3}{3|3}$. The meander $M_7^C$ associated with $\g$ is illustrated in Figure \ref{meander C}.
\begin{figure}[H]
$$\begin{tikzpicture}[scale=0.7]
    \def\Node{\node [circle, fill, inner sep=2pt]}
    \Node (A) at (10,0){};
    \Node (B) at (11,0){};
    \Node (C) at (12,0){};
    \Node (D) at (13,0){};
    \Node (E) at (14,0){};
    \Node [blue] (F) at (15,0){};
    \Node [red] (G) at (16,0){};
    
    \draw (C) to[bend left=60](E);
    \draw (A) to[bend right=60](C);
    \draw (D) to[bend right=60](F);
\end{tikzpicture}$$
\caption{Meander $M_7^C$ associated with $\p_7^C(\{\alpha_1,\alpha_2,\alpha_5\}\;|\;\{\alpha_3,\alpha_6\})$}
\label{meander C}
\end{figure}
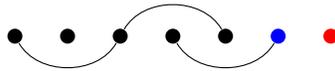
By definition, we have $T_a=\{v_6,v_7\},$ $T_b=\{v_7\},$ $T_\g=\{v_6\},$ and $T_\g^a=\{v_7\}.$ The vertices in $T_\g$ and $T_\g^a$ are colored blue and red, respectively, in Figure \ref{meander C}.
\end{ex}

When describing how to construct a regular functional, we first address any component of the meander $M_n^C$ which interacts with the tail, followed by the aftertail. From there, the embedding of functionals $F_n$ will be as it is in $\gl(n)$ for any part of the meander which remains unaddressed.

To begin, we must know how restricting to algebras in $\mf{sp}(2n)$ affects the index of a seaweed.

\begin{theorem}[Coll, Hyatt, and Magnant \textbf{[5]}]
\label{comp index C}
If $\g$ is a seaweed of type $\p_n^C\frac{a_1|\cdots|a_m}{b_1|\cdots|b_t}$, then 

$$\ind\g=2C+\Tilde{P},$$

\noindent 
where $C$ is the number of cycles and $\Tilde{P}$ is the number of paths with zero or two endpoints in the tail of the meander $M_n^C$ associated with $\g$.
\end{theorem}

\begin{remark}
An isolated point contained in the tail is considered to have one endpoint in the tail of a meander, as it only has one endpoint to begin with.
\end{remark}

We have the following immediate Corollary.

\begin{theorem}
\label{C cor}
The Lie algebra $C_n$ has index $n$.
\end{theorem}

Just as we used a meander half the size of the full meander in $\gl(2n)$, when there is symmetry across the antidiagonal, it suffices to consider a functional with indices on or above the antidiagonal only. The functional $F'=F+\sum_{(i,j)\in\mathscr{I}_F}c_{i,j}e_{2n+1-j,2n+1-i}^*$, where $c_{i,j}$ is the negative coefficient of $e_{i,j}^*$ in $F$ for $i,j\in[1,n]$, and equal to the coefficient of $e_{i,j}^*$ in $F$ otherwise, has the same kernel of the Kirillov form.

\begin{theorem}
The Functional $F_n=\sum_{i=1}^n\sum_{j=1}^{n+1-i}e_{i,j}^*$ of Theorem \ref{My Theorem 1} is regular on $C_n$.
\end{theorem}

\begin{proof}
Consider the standard basis for $C_n$. For $i,j\leq n$, we have 
\begin{equation} 
\label{P}
e_{i,j}-e_{2n+1-j,2n+1-i}\mapsto\left(\sum_{s=1}^{n+1-j}b_{s,i}-\sum_{s=1}^{n+1-i}b_{j,s}\right).
\end{equation} 
This system of equations which results from evaluating the $n^2$ expressions on the right hand side of (\ref{P}) at zero is equivalent to the system of equations for the image of the basis elements $e_{i,j}$ for $F_n$ defined on $\gl(n)$. For $(i,j)$ with $i+j\leq 2n+1$, $i\leq n$, and $j>n$ we have 
\begin{equation} 
\label{PP}
e_{i,j}+e_{2n+1-j,2n+1-i}\mapsto\left(-\sum_{s=1}^{n+1-i}b_{j,s}-\sum_{s=1}^{j-n}b_{2n+1-i,s}\right).
\end{equation} Through a linear algebra argument similar to those in section \ref{New Functional}, the solution to the system of equations which results from evaluating the right hand side of (\ref{PP}) at zero is $b_{j,i}=0$, for all $i,j$ defined. By Lemma \ref{zero symmetry}, this implies $b_{i,j}=0$. A relations matrix of $\ker(B_{F_n})$ is $B\oplus (-B^{\widehat{t}})$, where $B$ is a relations matrix of $\ker(B_{F_n})$ on $\gl(n)$.
\end{proof}

\begin{ex}
Consider the Lie algebra $C_4$. A relations matrix $B$ of $\ker(B_{F_4})$ is
\begin{center}
\scalebox{0.75}{
\begin{tabular}{c}
$B=\left(\begin{array}{cccc}
b_1+b_2+b_3+b_4&b_1+b_2+b_3&b_1+b_2&b_1\\
b_1+b_2+b_3&b_1+b_2+b_4 & b_1+b_3 & b_2\\
b_1+b_2&b_1+b_3&b_2+b_4&b_3\\
b_1&b_2&b_3&b_4
\end{array}\right)\bigoplus 
\left(\begin{array}{cccc}
-b_4&-b_3&-b_2&-b_1\\
-b_3&-b_2-b_4&-b_1-b_3&-b_1-b_2\\
-b_2&-b_1-b_3&-b_1-b_2-b_4&-b_1-b_2-b_3\\
-b_1&-b_1-b_2&-b_1-b_2-b_3&-b_1-b_2-b_3-b_4
\end{array}\right)$
\end{tabular}
}
\end{center}
\end{ex}

Now, we address proper seaweed subalgebras of $C_n$. We describe the adjustments needed from Definition \ref{Functional Construction def} to account for the aftertail and tail in Theorem \ref{func on C}.

\begin{theorem}
\label{func on C}
Let $\g$ be a seaweed of type-$C$ with associated meander $M_n^C$ and full meander $M(\g)$ defined on $2n$ vertices whose homotopy type is $H(c_1,\cdots,c_h)$. Let $f_c$ represent a functional on $\gl(c)$, for all $c$. Let $A=\mathscr{I}_{T_\g}\times\mathscr{I}_{T_\g}$ be the indices in the square block on the diagonal of $\g$ which contains the tail and $B=\mathscr{I}_{T_\g^a}\times\mathscr{I}_{T_\g^a}$ be the indices in the square block on the diagonal of $\g$ which contains the aftertail. For each $c_i$ such that $\textgoth{C}_{c_i}\cap A\neq\emptyset$ (i.e., each component whose core interacts with the tail of the meander $M_n^C$), define $\oF_{c_i}$ as in Definition \ref{Functional Construction def} by embedding a functional $f_{\lfloor c_i/2\rfloor}$, except only sum over $e_{i,j}^*$ with $i+j\leq2n+1$ (i.e., on or above the antidiagonal). For each $c_i$ such that $\textgoth{C}_{c_i}\cap A=\textgoth{C}_{c_i}\cap B=\emptyset$, define $\oF_{c_i}$ as in Definition \ref{Functional Construction def} except only sum over $e_{i,j}^*$ with $i+j\leq2n+1$. As in Definition \ref{Functional Construction def}, in both these embeddings we allow for the choice to rotate the indices or not by adding the appropriate functionals in the peak blocks for any peak block which occurs strictly above the antidiagonal of $\g$. The only difference is that, when crossing the antidiagonal, we require the choice of functionals over the main diagonal of the peak block \textup(which occur on or above the antidiagonal\textup). Finally, if $t=|T_\g^a|$, then the final functional 

$$F=\sum\oF_{c_i}+f_t^{n-t}$$

\noindent 
is such that

$$\dim\ker(B_F)=\sum\dim\ker(B_{f_{c_i}})+\sum\dim\ker(B_{f_{\lfloor c_i/2\rfloor}})+\dim\ker(B_{f_t}),$$

\noindent 
where $\dim\ker(B_F)$ is over $\g$, for each $i$ $\dim\ker(B_{f_{c_i}})$ is over $\gl(c_i)$ and $\dim\ker(B_{f_{\lfloor c_i/2\rfloor}})$ is over $\p_{c_i}^C(\{c_i\}\;|\;\emptyset)$, and $\dim\ker(B_{f_t})$ is over $C_t$.

As before, the constructed functional is regular if and only if we embed regular functionals in each component.
\end{theorem}

We first introduce the following nontrivial example which demonstrates Theorem \ref{func on C} and highlights the differences between the tail and aftertail.

\begin{ex}
\label{finalex C}
Consider $\g$ of type $\p_{18}^C\frac{5|10|6|10|5}{2|4|3|1|1|14|1|1|3|4|2}$. The meanders $M_{18}^C$ and $M(\g)$ are shown in Figure\ref{crazymeander2} \textup(left and right, respectively\textup), with the tail vertices and components colored blue and the aftertail vertices and component colored red. It follows from Theorem \ref{comp index C} that $\ind\g=7$.

\begin{figure}[H]
$$\begin{tikzpicture}[scale=0.28]
    \def\Node{\node [circle, fill, inner sep=1.5pt]}
    \Node (1) at (1,0){};
    \Node (2) at (2,0){};
    \Node (3) at (3,0){};
    \Node (4) at (4,0){};
    \Node (5) at (5,0){};
    \Node (6) at (6,0){};
    \Node (7) at (7,0){};
    \Node (8) at (8,0){};
    \Node (9) at (9,0){};
    \Node (10) at (10,0){};
    \Node (11) at (11,0){};
    \Node [blue] (12) at (12,0){};
    \Node [blue] (13) at (13,0){};
    \Node [blue] (14) at (14,0){};
    \Node [blue] (15) at (15,0){};
    \Node [red] (16) at (16,0){};
    \Node [red] (17) at (17,0){};
    \Node [red] (18) at (18,0){};
    \draw (1) to[bend left=60] (5);
    \draw (2) to[bend left=60] (4);
    \draw [line width=0.45mm, blue] (6) to[bend left=60] (15);
    \draw [line width=0.45mm, blue] (7) to[bend left=60] (14);
    \draw [line width=0.45mm, blue] (8) to[bend left=60] (13);
    \draw [line width=0.45mm, blue] (9) to[bend left=60] (12);
    \draw (10) to[bend left=60] (11);
    
    \draw (1) to[bend right=60] (2);
    \draw [line width=0.45mm, blue] (3) to[bend right=60] (6);
    \draw (4) to[bend right=60] (5);
    \draw [line width=0.45mm, blue] (7) to[bend right=60] (9);
    \node at (1,-4.1){};
\end{tikzpicture}
\hspace{2em}
\begin{tikzpicture}[scale=0.28]
    \def\Node{\node [circle, fill, inner sep=1.5pt]}
    \Node (1) at (1,0){};
    \Node (2) at (2,0){};
    \Node (3) at (3,0){};
    \Node (4) at (4,0){};
    \Node (5) at (5,0){};
    \Node (6) at (6,0){};
    \Node (7) at (7,0){};
    \Node (8) at (8,0){};
    \Node (9) at (9,0){};
    \Node (10) at (10,0){};
    \Node (11) at (11,0){};
    \Node [blue] (12) at (12,0){};
    \Node [blue] (13) at (13,0){};
    \Node [blue] (14) at (14,0){};
    \Node [blue] (15) at (15,0){};
    \Node [red] (16) at (16,0){};
    \Node [red] (17) at (17,0){};
    \Node [red] (18) at (18,0){};
    \Node [red] (1800) at (19,0){};
    \Node [red] (1700) at (20,0){};
    \Node [red] (1600) at (21,0){};
    \Node [blue] (1500) at (22,0){};
    \Node [blue] (1400) at (23,0){};
    \Node [blue] (1300) at (24,0){};
    \Node [blue] (1200) at (25,0){};
    \Node (1100) at (26,0){};
    \Node (1000) at (27,0){};
    \Node (900) at (28,0){};
    \Node (800) at (29,0){};
    \Node (700) at (30,0){};
    \Node (600) at (31,0){};
    \Node (500) at (32,0){};
    \Node (400) at (33,0){};
    \Node (300) at (34,0){};
    \Node (200) at (35,0){};
    \Node (100) at (36,0){};
    \draw (1) to[bend left=60] (5);
    \draw (2) to[bend left=60] (4);
    \draw [line width=0.45mm, blue] (6) to[bend left=60] (15);
    \draw [line width=0.45mm, blue] (7) to[bend left=60] (14);
    \draw [line width=0.45mm, blue] (8) to[bend left=60] (13);
    \draw [line width=0.45mm, blue] (9) to[bend left=60] (12);
    \draw (10) to[bend left=60] (11);
    
    \draw (1) to[bend right=60] (2);
    \draw [line width=0.45mm, blue] (3) to[bend right=60] (6);
    \draw (4) to[bend right=60] (5);
    \draw [line width=0.45mm, blue] (7) to[bend right=60] (9);
    
    \draw (100) to[bend right=60] (500);
    \draw (200) to[bend right=60] (400);
    
    \draw [line width=0.45mm, blue] (600) to[bend right=60] (1500);
    \draw [line width=0.45mm, blue] (700) to[bend right=60] (1400);
    \draw [line width=0.45mm, blue] (800) to[bend right=60] (1300);
    \draw [line width=0.45mm, blue] (900) to[bend right=60] (1200);
    \draw (1000) to[bend right=60] (1100);
    
    \draw (100) to[bend left=60] (200);
    \draw [line width=0.45mm, blue] (300) to[bend left=60] (600);
    \draw (400) to[bend left=60] (500);
    \draw [line width=0.45mm, blue] (700) to[bend left=60] (900);
    
    \draw [line width=0.45mm, blue] (12) to[bend right=60] (1200);
    \draw [line width=0.45mm, blue] (13) to[bend right=60] (1300);
    \draw [line width=0.45mm, blue] (14) to[bend right=60] (1400);
    \draw [line width=0.45mm, blue] (15) to[bend right=60] (1500);
    \draw [line width=0.45mm, red] (16) to[bend right=60] (1600) to[bend right=60] (16);
    \draw [line width=0.45mm, red] (17) to[bend right=60] (1700) to[bend right=60] (17);
    \draw [line width=0.45mm, red] (18) to[bend right=60] (1800) to[bend right=60] (18);
\end{tikzpicture}$$
\vspace{-2em}
\caption{Meanders $M_{18}^C$ and $M(\g)$ associated with $\g$ of type $\p_{18}^C\frac{5|10|6|10|5}{2|4|3|1|1|14|1|1|3|4|2}$}
\label{crazymeander2}
\end{figure}
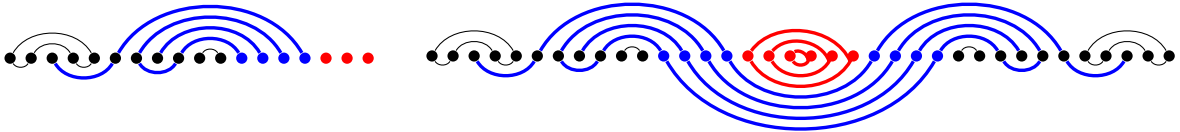 

The functional $\oF^a$ described by Theorem \ref{func on C} by embedding functionals $F_c$ from Theorem \ref{My Theorem 1} is illustrated in Figure \ref{bigex}, where the indices of $\g$ in the aftertail component are colored red, the indices in the tail components are colored blue, and a black dot is placed in each index of $\mathscr{I}_{\oF^a}$. We have added lines to emphasize the core and components of $\g$.

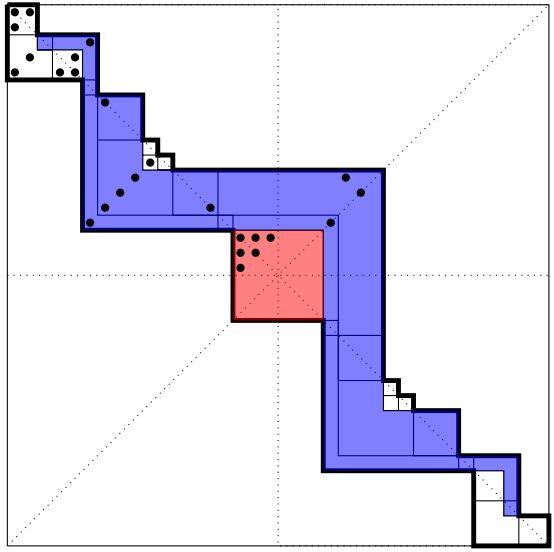
\begin{figure}[H]
$$\begin{tikzpicture}[scale=0.2]
    \def\Node{\node [circle, fill, inner sep=1.1pt]}
    \draw (0,0)--(36,0)--(36,36)--(0,36)--(0,0);
    \draw [line width=.65mm] (0,36)--(0,31)--(5,31)--(5,21)--(15,21)--(15,15)--(21,15)--(21,5)--(31,5)--(31,0)--(36,0)--(36,2)--(34,2)--(34,6)--(30,6)--(30,9)--(27,9)--(27,10)--(26,10)--(26,11)--(25,11)--(25,25)--(11,25)--(11,26)--(10,26)--(10,27)--(9,27)--(9,30)--(6,30)--(6,34)--(2,34)--(2,36)--(0,36);
    \draw [dotted] (0,0)--(36,36)--(0,36)--(36,0)--(18,0)--(18,36);
    \draw [dotted] (0,18)--(36,18);
    \draw (0,34)--(2,34)--(2,33)--(3,33)--(3,31)--(5,31)--(5,33)--(2,33)--(2,34)--(2,35);
    \draw (3,34)--(3,33);
    \draw (5,31)--(6,31)--(6,30)--(5,30);
    \draw (6,30)--(6,22)--(15,22)--(15,21)--(14,21)--(14,22);
    \draw (6,27)--(9,27)--(9,25)--(11,25)--(10,25)--(10,26)--(9,26);
    \draw (11,25)--(11,22)--(14,22)--(14,25);
    \draw (14,21)--(21,21)--(21,14)--(25,14)--(25,11)--(26,11)--(26,10);
    \draw (21,14)--(22,14)--(22,11)--(25,11)--(25,9)--(27,9)--(27,6)--(30,6);
    \draw (21,15)--(22,15)--(22,14)--(25,14);
    \draw (14,22)--(22,22)--(22,14);
    \draw (22,14)--(22,6)--(30,6)--(30,5);
    \draw (25,10)--(26,10)--(26,9);
    \draw (31,6)--(31,5)--(33,5)--(33,2)--(34,2)--(34,3)--(33,3);
    \draw (31,3)--(33,3);
    \draw (34,0)--(34,2);
    
    \draw [fill=blue, opacity=0.5] (2,34)--(6,34)--(6,30)--(9,30)--(9,25)--(25,25)--(25,9)--(30,9)--(30,6)--(34,6)--(34,2)--(33,2)--(33,5)--(21,5)--(21,21)--(5,21)--(5,33)--(2,33)--(2,34);
    \draw [fill=red, opacity=0.5] (15,21)--(21,21)--(21,15)--(15,15)--(15,21);
    
    \Node at (0.5,35.5){};
    \Node at (0.5,34.5){};
    \Node at (0.5,31.5){};
    \Node at (1.5,35.5){};
    \Node at (1.5,32.5){};
    \Node at (3.5,31.5){};
    \Node at (4.5,32.5){};
    \Node at (4.5,31.5){};
    \Node at (5.5,33.5){};
    \Node at (5.5,21.5){};
    \Node at (6.5,29.5){};
    \Node at (6.5,22.5){};
    \Node at (7.5,23.5){};
    \Node at (8.5,24.5){};
    \Node at (9.5,25.5){};
    \Node at (13.5,22.5){};
    \Node at (15.5,20.5){};
    \Node at (15.5,19.5){};
    \Node at (15.5,18.5){};
    \Node at (16.5,20.5){};
    \Node at (16.5,19.5){};
    \Node at (17.5,20.5){};
    \Node at (21.5,21.5){};
    \Node at (22.5,24.5){};
    \Node at (23.5,23.5){};
\end{tikzpicture}$$
\caption{Indices in $\mathscr{I}_{\oF}$ on $\p_{18}^C\frac{5|10|6|10|5}{2|4|3|1|1|14|1|1|3|4|2}$}
\label{bigex}
\end{figure}



Direct computation yields that $B\oplus(-B^{\widehat{t}})$ is a relations matrix of $\ker(B_{\oF})$, with 
\begin{center}
\scalebox{0.75}{
\begin{tabular}{c}
$B=
\left(\begin{array}{cc}
b_1+b_2&b_1\\
b_1&b_2
\end{array}\right)
\bigoplus (0)\bigoplus 
\left(\begin{array}{cc}
b_2&b_1\\
b_1&b_1+b_2
\end{array}\right)
\bigoplus (0)
\bigoplus 
\left(\begin{array}{ccc}
b_3&0&0\\
0&0&0\\
0&0&-b_3
\end{array}\right)
\bigoplus \left(\begin{array}{cc}
b_4&0\\
0&b_4
\end{array}\right)
\bigoplus
\left(\begin{array}{ccc}
-b_3&0&0\\
0&0&0\\
0&0&b_3
\end{array}\right)
\bigoplus (0)\bigoplus 
\left(\begin{array}{ccc}
b_5+b_6+b_7&b_5+b_6&b_5\\
b_5+b_6&b_5+b_7&b_6\\
b_5&b_6&b_7
\end{array}\right).$
\end{tabular}
}
\end{center}

\end{ex}

\begin{proof}[Proof of Theorem \ref{func on C}] The proof is similar to the proof of Theorem \ref{Functional Construction}. Unlike in Type-$A$, the Component Creation move in Type-$C$ yields a direct sum, so the result for this move follows the proof of Theorem \ref{Functional Construction}. However, some care is needed to address the tail and aftertail. By definition, the aftertail is a self-contained component (a set of nested cycles which is not wound-up) and, therefore, the proof follows from the proof of Theorem \ref{Functional Construction}. For the components which have cores that intersect the tail nontrivially, the induction is the same as in the proof of Theorem \ref{Functional Construction}, except that a separate base case is needed. 

Note that a component of size $c_i$ no longer contributes $c_i$ to the index of $\g$, but rather $\lfloor\frac{c_i}{2}\rfloor$. For the base case on tail components, consider the seaweed $\p_{c_i}^C(\{c_i\}\;|\;\emptyset)$ and the functional 

$$F=F_{\lfloor c_i/2\rfloor}+\sum_{i=1}^{\lceil c_i/2\rceil}e_{i,c_i+i}^*.$$

\noindent 
Let $B$ be a $\lfloor\frac{c_i}{2}\rfloor\times\lfloor\frac{c_i}{2}\rfloor$ relations matrix of $\ker(B_{F_{\lfloor c_i/2\rfloor}})$ on $\gl(\lfloor\frac{c_i}{2}\rfloor)$. By direct computation, if $c_i$ is even then $\ker(B_{F})$ has a relations matrix 

$$B\oplus(-B^{\widehat{t}})\oplus B\oplus(-B^{\widehat{t}}).$$ 

\noindent 
If $n$ is odd, then a relations matrix of $\ker(B_F)$ is 

$$B\oplus(0)\oplus(-B^{\widehat{t}})\oplus B\oplus(0)\oplus (-B^{\widehat{t}}).$$
\end{proof}

We introduce the following \textit{reduced homotopy type} for seaweeds of Type-$C$ to construct the analogue of Theorem \ref{Index and Homotopy}. We use the word ``reduced" as some of the $c_i$'s are omitted from the full homotopy type $H(c_1,\cdots,c_h)$.

\begin{definition}
\label{reducedmeander}
Let $\g$ be a seaweed subalgebra of $C_n$ with reduced meander $M_n^C$ and full meander $M(\g)$. Color the aftertail component \textup(if any\textup) of $M(\g)$ red and the tail components \textup(if any\textup) of $M(\g)$ blue. Eliminate any arcs and vertices to the right of $v_n$ in $M(\g)$ which are not colored red or blue. This produces a meander $M'$ on $I$ vertices with $I\in[n,2n]$. Apply Lemma \ref{winding down} to $M'$ to unwind it, and in each Component Elimination move $C(c)$, color $c$ the color of the component removed. Then $H_C(c_1,\cdots,c_h)$ is the {reduced homotopy type} of a Type-$C$ seaweed.
\end{definition}

\begin{ex}
Consider $\g$ of Example \ref{finalex C}. The meander $M(\g)$ is in Figure \ref{crazymeander2} \textup(right\textup). The meander $M'$ of Definition \ref{reducedmeander} is shown in Figure \ref{smaller}.

\begin{figure}[H]
$$\begin{tikzpicture}[scale=0.28]
    \def\Node{\node [circle, fill, inner sep=1.3pt]}
    \Node (1) at (1,0){};
    \Node (2) at (2,0){};
    \Node (3) at (3,0){};
    \Node (4) at (4,0){};
    \Node (5) at (5,0){};
    \Node (6) at (6,0){};
    \Node (7) at (7,0){};
    \Node (8) at (8,0){};
    \Node (9) at (9,0){};
    \Node (10) at (10,0){};
    \Node (11) at (11,0){};
    \Node [blue] (12) at (12,0){};
    \Node [blue] (13) at (13,0){};
    \Node [blue] (14) at (14,0){};
    \Node [blue] (15) at (15,0){};
    \Node [red] (16) at (16,0){};
    \Node [red] (17) at (17,0){};
    \Node [red] (18) at (18,0){};
    \Node [red] (1800) at (19,0){};
    \Node [red] (1700) at (20,0){};
    \Node [red] (1600) at (21,0){};
    \Node [blue] (1500) at (22,0){};
    \Node [blue] (1400) at (23,0){};
    \Node [blue] (1300) at (24,0){};
    \Node [blue] (1200) at (25,0){};
    \Node [blue] (900) at (26,0){};
    \Node [blue] (800) at (27,0){};
    \Node [blue] (700) at (28,0){};
    \Node [blue] (600) at (29,0){};
    \Node [blue] (300) at (30,0){};
    \draw (1) to[bend left=60] (5);
    \draw (2) to[bend left=60] (4);
    \draw [line width=0.45mm, blue] (6) to[bend left=60] (15);
    \draw [line width=0.45mm, blue] (7) to[bend left=60] (14);
    \draw [line width=0.45mm, blue] (8) to[bend left=60] (13);
    \draw [line width=0.45mm, blue] (9) to[bend left=60] (12);
    \draw (10) to[bend left=60] (11);
    
    \draw (1) to[bend right=60] (2);
    \draw [line width=0.45mm, blue] (3) to[bend right=60] (6);
    \draw (4) to[bend right=60] (5);
    \draw [line width=0.45mm, blue] (7) to[bend right=60] (9);
    
    \draw [line width=0.45mm, blue] (600) to[bend right=60] (1500);
    \draw [line width=0.45mm, blue] (700) to[bend right=60] (1400);
    \draw [line width=0.45mm, blue] (800) to[bend right=60] (1300);
    \draw [line width=0.45mm, blue] (900) to[bend right=60] (1200);
    
    \draw [line width=0.45mm, blue] (300) to[bend left=60] (600);
    \draw [line width=0.45mm, blue] (700) to[bend left=60] (900);
    
    \draw [line width=0.45mm, blue] (12) to[bend right=60] (1200);
    \draw [line width=0.45mm, blue] (13) to[bend right=60] (1300);
    \draw [line width=0.45mm, blue] (14) to[bend right=60] (1400);
    \draw [line width=0.45mm, blue] (15) to[bend right=60] (1500);
    \draw [line width=0.45mm, red] (16) to[bend right=60] (1600) to[bend right=60] (16);
    \draw [line width=0.45mm, red] (17) to[bend right=60] (1700) to[bend right=60] (17);
    \draw [line width=0.45mm, red] (18) to[bend right=60] (1800) to[bend right=60] (18);
\end{tikzpicture}$$
\caption{Reduced meander $M'$ of $\p_{18}^C\frac{5|10|6|10|5}{2|4|3|1|1|14|1|1|3|4|2}$}
\label{smaller}
\end{figure}
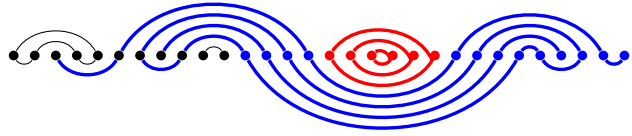

The reduced homotopy type of $\g$ is $H_C(2,1,{\textcolor{blue}{1}},{\textcolor{blue}{3}},{\textcolor{red}{6}})$, shown in Figure \ref{htpyC}.

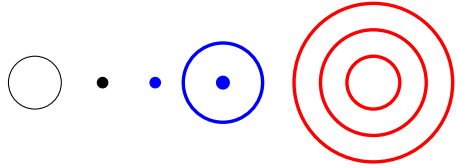
\begin{figure}[H]
$$\begin{tikzpicture}[scale=2]
\def\Node{\node [circle, fill, inner sep=1.5pt]}
	\draw (0.25,0) node[draw,circle,fill=white,minimum size=20pt,inner sep=0pt] (2+) {};
    
	\draw (0.7,0) node[draw,circle,fill=black,minimum size=4pt,inner sep=0pt] (2+) {};
	
	\draw (1.05,0) node[blue,draw,circle,fill=blue,minimum size=4pt,inner sep=0pt] (2+) {};
    
	\draw (1.5,0) node[blue,draw,line width=0.45mm,,circle,fill=white,minimum size=30,inner sep=0pt] (2+) {};
	\draw (1.5,0) node[blue,draw,line width=0.45mm,,circle,fill=blue,minimum size=4pt,inner sep=0pt] (2+) {};
    
	\draw (2.5,0) node[red,draw,line width=0.45mm,circle,fill=white,minimum size=60,inner sep=0pt] (2+) {};
	\draw (2.5,0) node[red,draw,line width=0.45mm,,circle,fill=white,minimum size=40,inner sep=0pt] (2+) {};
	\draw (2.5,0) node[red,draw,line width=0.45mm,,circle,fill=white,minimum size=20pt,inner sep=0pt] (2+) {};
\end{tikzpicture}$$
\caption{Reduced homotopy type $H_C(2,1,{\textcolor{blue}{1}},{\textcolor{blue}{3}},{\textcolor{red}{6}})$}
\label{htpyC}
\end{figure}
\end{ex}

The following theorem is the Type-$C$ analogue of the theorem in Type-$A$ and $\gl(n)$ (cf. Theorem \ref{Index and Homotopy}). Note that in Type-$A$, there is no tail or aftertail.

\begin{theorem}\label{reduced homotopy type}
If $\g$ is a seaweed of type-$C$ with reduced homotopy type $H_C(c_1,\cdots,c_{h_1},\bm{\textcolor{blue}{c_{h_1+1}}},\cdots,\bm{\textcolor{blue}{c_{h_2}}},\bm{\textcolor{red}{c_{h_2+1}}}),$ then 

$$\ind\g=\sum_{i=1}^{h_1}c_i+\sum_{i=h_1+1}^{h_2}\left\lfloor\frac{\bm{\textcolor{blue}{c_i}}}{2}\right\rfloor+\frac{\bm{\textcolor{red}{c_{h_2+1}}}}{2}.$$
\end{theorem}
\end{subsection}

\begin{subsection}{Type-$B$ seaweeds}
\label{Type B}
Naming explicit regular functionals on seaweed subalgebras of the special orthogonal matrix algebras requires an adapted framework from that established in previous sections of this paper.  The necessary modification is to adjust the embedding of a functional in any tail component of odd size, as the previously established framework in section \ref{type C} would require the use of a functional $e_{i,n+1-i}^*$, which is precluded by the forced zeroes on the antidiagonal of $\mf{so}(n)$. 


Seaweed subalgebras of $B_n$ are constructed as in $\gl(2n+1)$, but they must also be subalgebras of $\mathfrak{so}(2n+1)$.  Because of the symmetry across the antidiagonal of $B_n$, we have a simplified notation for seaweeds of Type-$B$.

\begin{definition}
Given two partial compositions $(a_1,\cdots,a_m)$ and $(b_1,\cdots,b_t)$ of $n$ \textup(i.e., $\sum_{i=1}^ma_i,\sum_{i=1}^tb_i\leq n$\textup), let $\g$ be the seaweed of type $\frac{a_1|\cdots|a_m|2(n-\sum_{i=1}^ma_i)+1|a_m|\cdots|a_1}{b_1|\cdots|b_t|2(n-\sum_{i=1}^tb_i)+1|b_t|\cdots|b_1}$ which is a subalgebra of $B_n$.  This is the standard seaweed of type $\p_n^B\frac{a_1|\cdots|a_m}{b_1|\cdots|b_t}$.
\end{definition}

\begin{ex}
Consider the seaweed $\g=\p_3^B\frac{3}{2}$.  This is the set of all matrices in $\mathfrak{so}(7)$ whose possible nonzero entries occur in the marked entries of Figure \ref{type B construction}.
\begin{figure}[H]
 $$
 \begin{tikzpicture}[scale=0.5]
 \draw (0,0)--(0,7)--(7,7)--(7,0)--(0,0);
 \draw [line width=0.45mm](0,7)--(0,4)--(3,4)--(3,3)--(4,3)--(4,0)--(7,0)--(7,2)--(5,2)--(5,5)--(2,5)--(2,7)--(0,7);
 \node at (0.5,0.5){0};
 \node at (1.5,1.5){0};
 \node at (2.5,2.5){0};
 \node at (3.5,3.5){0};
 \node at (4.5,4.5){0};
 \node at (5.5,5.5){0};
 \node at (6.5,6.5){0};
 \node at (0.5,6.3){\LARGE *};
 \node at (0.5,5.3){\LARGE *};
 \node at (0.5,4.3){\LARGE *};
 \node at (1.5,6.3){\LARGE *};
 \node at (1.5,5.3){\LARGE *};
 \node at (1.5,4.3){\LARGE *};
 \node at (2.5,4.3){\LARGE *};
 \node at (3.5,4.3){\LARGE *};
 \node at (4.5,3.3){\LARGE *};
 \node at (4.5,2.3){\LARGE *};
 \node at (4.5,1.3){\LARGE *};
 \node at (4.5,0.3){\LARGE *};
 \node at (5.5,1.3){\LARGE *};
 \node at (5.5,0.3){\LARGE *};
 \node at (6.5,1.3){\LARGE *};
 \node at (6.5,0.3){\LARGE *};
 \node [label=left:{3}] at (0,5.5){};
 \node [label=left:{1}] at (3,3.5){};
 \node [label=left:{3}] at (4,1.5){};
 \node [label=above:{2}] at (1,7){};
 \node [label=above:{3}] at (3.5,5){};
 \node [label=above:{2}] at (6,2){};
 \end{tikzpicture}$$
 \caption{Seaweed of type $\p_3^B\frac{3}{2}$}
 \label{type B construction}
 \end{figure}
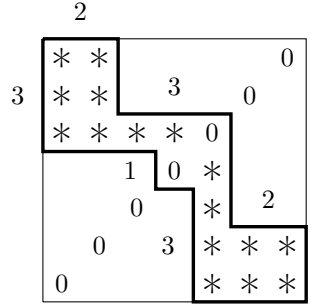
\end{ex}

To ease computations, we can leverage the symmetry across the antidiagonal of a seaweed subalgebra of $B_n$, and make use of a meander on $n$ vertices instead of the full $2n+1$ vertices that a seaweed subalgebra of $\gl(2n+1)$ would normally require.

\begin{definition}
Let $\g=\p_n^B\frac{a_1|\cdots|a_m}{b_1|\cdots|b_t}$.  The shortened meander associated with $\g$ \textup(denoted $M_n^B$ to differentiate it from the meander on $2n+1$ vertices and from a shortened meander of type-$C$\textup) is constructed exactly as in type-$C$.  We define the tail and aftertail analogously.  See example \ref{exC}.
\end{definition}

When describing how to construct a regular functional on $\p_n^B\frac{a_1|\cdots|a_m}{b_1|\cdots|b_t}$, we first embed functionals for each component of the homotopy type of $M_n^B$ which intersects the tail.  We then embed a functional to the aftertail component, if it is present. From there, the embedding of functionals $F_n$ will be as it is in $\gl(n)$ for any part of the meander which remains unaddressed.  What makes type-$B$ different from type-$C$ lies entirely in tail components of the meander which are of odd size. In type-$C$, when there is a tail component of odd size in the meander, this leads to the presence of an odd peak block which lies on the antidiagonal of the parent $2n\times2n$ matrix.  The regular functional embedding described for a type-$C$ meander requires the use of functionals $e_{i,j}^*$ which are on the main diagonal of the peak block.  Therefore, due to the odd size of the defined tail component, we know that exactly one of these functionals will have a coordinate location which is on the antidiagonal of the larger matrix, a position which is a forced zero in type-$B$.  We must address how to correct for this lost information when we describe the new embedding necessary to build a regular functional on a type-$B$ seaweed. First, we include a couple parallel definitions from section \ref{type C}.


To begin, we must address how restricting to algebras in $\mf{so}(2n+1)$ affects the index of a seaweed.

\begin{theorem}[Coll, Hyatt, and Magnant \textbf{[6]}; Panyushev and Yakimova \textbf{[17]}]
\label{B ind}
If $\g=\p_n^B\frac{a_1|\cdots|a_m}{b_1|\cdots|b_t}$ is a seaweed of Type-$B$, then $$\ind\g=2C+\tilde{P},$$ where $C$ is the number of cycles and $\tilde{P}$ is the number of paths and isolated points with zero or two endpoints in the tail of the meander $M_n^B$ associated with $\g$.
\end{theorem}

\begin{remark}
As in type-$C$, an isolated point contained in the tail is considered to have one endpoint in the tail of a meander, as it only has one endpoint to begin with.
\end{remark}

We have the following immediate Corollary.

\begin{theorem}
\label{B cor}
The Lie algebra $B_n$ has index $n$.
\end{theorem}

Just as we used a meander half the size of the full meander in $\gl(2n+1)$, when there is symmetry across the antidiagonal, it suffices to consider a functional with indices above the antidiagonal only. The functional $F'=F+\sum_{(i,j)\in\mathscr{I}_F}c_{i,j}e_{2n+1-j,2n+1-i}^*$, where $c_{i,j}$ is the appropriate coefficient of $e_{i,j}^*$ in $F$ to model the restrictions on $\mathfrak{so}(2n+1)$ has the same kernel of the Kirillov form.

\begin{theorem}
The Functional $F_n=\sum_{i=1}^n\sum_{j=1}^{n+1-i}e_{i,j}^*$ of Theorem \ref{My Theorem 1} is regular on $B_n$.
\end{theorem}

\begin{proof}
The systems of equations generated by requiring $B_{F_n}(B,e_{i,j}-e_{2n+2-j,2n+2-i})=0$ on $\mf{so}(2n+1)$ and $B_{F_n}(B,e_{i,j})=0$ on $\gl(n)$ for $i,j\leq n$ are equivalent. To prove $F_n$ is regular on $B_n$, it suffices to show that $b_{i,j}=0$ for all $i\in[1,n]$, $j>n$. Note that
\begin{equation} 
\label{midrow}
e_{i,n+1}-e_{n+1,2n+1-i}\mapsto\sum_{s=1}^{n+1-i}b_{n+1,s}.
\end{equation}
By setting the expressions on the right hand side of (\ref{midrow}) equal to zero, we get a system of equations $n$ whose solution is $b_{n+1,i}=0$ for $s\in[1,n]$ (this is seen by induction, the base case is $i=n$ and the induction goes down to $i=1$). We get $b_{i,j}=0$ for all $(i,j)\in[1,n]\times[n+2,2n+1]$ through a linear algebra argument similar to that in the proof of Theorem \ref{My Theorem 2} on the set of equations
$$B_{F_n}(B,e_{j,i}-e_{2n+2-i,2n+2-j})=\sum_{s=1}^{n+1-i}b_{s,j}+\sum_{s=1}^{j-(n+1)}b_{s,2n+2-i}=0.$$
In conclusion, a relations matrix of $\ker(B_{F_n})$ on $B_n$ will be $$B\oplus(0)\oplus \left(-B^{\widehat{t}}\right),$$ where $B$ is an $n\times n$ relations matrix of $\ker(B_{F_n})$ on $\gl(n)$.
\end{proof}

\begin{ex}
Consider the Lie algebra $B_4$. A relations matrix $B$ of $\ker(B_{F_4})$ is
\begin{center}
\scalebox{0.75}{
\begin{tabular}{c}
$B=\left(\begin{array}{cccc}
b_1+b_2+b_3+b_4&b_1+b_2+b_3&b_1+b_2&b_1\\
b_1+b_2+b_3&b_1+b_2+b_4 & b_1+b_3 & b_2\\
b_1+b_2&b_1+b_3&b_2+b_4&b_3\\
b_1&b_2&b_3&b_4
\end{array}\right)
\bigoplus \text{\Large{\textup(\textup0\textup)}} \bigoplus
\left(\begin{array}{cccc}
-b_4&-b_3&-b_2&-b_1\\
-b_3&-b_2-b_4&-b_1-b_3&-b_1-b_2\\
-b_2&-b_1-b_3&-b_1-b_2-b_4&-b_1-b_2-b_3\\
-b_1&-b_1-b_2&-b_1-b_2-b_3&-b_1-b_2-b_3-b_4
\end{array}\right)$
\end{tabular}
}
\end{center}
\end{ex}

The reduced homotopy type $H_B(c_1,\cdots,c_{h_1},\bm{\textcolor{blue}{c_{h_1+1}}},\cdots,\bm{\textcolor{blue}{c_{h_2}}},\bm{\textcolor{red}{c_{h_2+1}}})$ on $\g$ is defined the same as in type-$C$. We have the immediate analogue of Theorem \ref{reduced homotopy type}.

\begin{theorem}
\label{index and homotopy C}
If $\g$ is a seaweed of type-$B$ with reduced homotopy type $$H_B(c_1,\cdots,c_{h_1},\bm{\textcolor{blue}{c_{h_1+1}}},\cdots,\bm{\textcolor{blue}{c_{h_2}}},\bm{\textcolor{red}{c_{h_2+1}}}),$$ then $$\ind\g=\sum_{i=1}^{h_1}c_i+\sum_{i=1}^{h_2}\left\lfloor\frac{c_i}{2}\right\rfloor+\frac{c_{h_2+1}-1}{2}.$$
\end{theorem}

\begin{theorem}
\label{typebfunctional}
Let $\g$ be a type-$B$ seaweed with reduced homotopy type $H_B(c_1,\cdots,c_{h_1},\bm{\textcolor{blue}{c_{h_1+1}}},\cdots,\bm{\textcolor{blue}{c_{h_2}}},\bm{\textcolor{red}{c_{h_2+1}}}),$ and let $F$ be the functional constructed as in Theorem \ref{func on C} with the following modifications. Without loss of generality, since all the peak blocks which occurs on the antidiagonal of $\g$ must occur on the same side of the main diagonal (i.e., either they are all above the main diagonal or they are all below the main diagonal), let us assume that all peak blocks for the tail occur above the main diagonal of $\g$ (if they all occur below the main diagonal, for every coordinate $(i,j)$ throughout this definition use the appropriate transposition $(j,i)$ instead).  Let $A:=\{i_1,\cdots,i_k\}\subset\{h_1+1,\cdots,h_2\}$ such that $A$ is the maximal set satisfying that $\bm{\textcolor{blue}{c_{i}}}$ is odd for all $i\in A$ and $i_1<\cdots<i_k$. Define a set $B=\{j_1,\cdots,j_k\}$ such that $i_s+j_s=2n+2$ for all $s\in[1,k]$ (i.e., $(i_s,j_s)$ is on the antidiagonal of $\g$ for all $s$).  For all $s\in[1,k]$, define a functional $f_s=e_{i_s,n+1}^*+\sum_{t=s+1}^ke_{i_s,j_s}^*$.  In $F$, for replace every functional $e_{i_s,j_s}^*$ with the functional $f_s$.  Then the index calculated using the constructed functional $F$ is equal to the appropriate summation of the dimensions of the smaller kernel spaces corresponding to the chosen embedded functionals in the core of $\g$, as in Theorem \ref{func on C}.
\end{theorem}

The proof requires a modification to the base case on the tail components for the odd component along with an appropriate induction along the finite set of odd-sized tail components, starting from the closest to the aftertail and working in decreasing index number from there.

\begin{ex}
Consider $\g=\p_{12}^B\frac{1|5|3|1}{4}$. This is a subalgebra of the seaweed of type  $\frac{1|5|3|1|5|1|3|5|1}{4|17|4}$. The meanders $M_{12}^B$ and $M$ are shown in Figures \ref{crazyBex1} and \ref{crazyBex2}, respectively, with the tail vertices and components colored blue. It follows from Theorem \ref{B ind} that $\ind\g=5$.
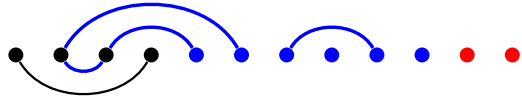
\begin{figure}[H]
$$\begin{tikzpicture}[scale=0.6]
	\def\Node{\node [circle, fill, inner sep=2pt]}
    \Node (1) at (1,0){};
    \Node (2) at (2,0){};
    \Node (3) at (3,0){};
    \Node (4) at (4,0){};
    \Node [blue] (5) at (5,0){};
    \Node [blue] (6) at (6,0){};
    \Node [blue] (7) at (7,0){};
    \Node [blue] (8) at (8,0){};
    \Node [blue] (9) at (9,0){};
    \Node [blue] (10) at (10,0){};
    \Node [red] (11) at (11,0){};
    \Node [red] (12) at (12,0){};
    
    \draw [line width=0.45mm, blue] (2) to[bend left=60](6);
    \draw [line width=0.45mm, blue] (3) to[bend left=60](5);
    \draw [line width=0.45mm, blue] (7) to[bend left=60](9);
    
    \draw [line width=0.3mm] (1) to[bend right=60](4);
    \draw [line width=0.45mm, blue] (2) to[bend right=60](3);
\end{tikzpicture}$$
\caption{Meander $M_{12}^B$ associated with $\p_{12}^B\frac{1|5|3|1}{4}$}
\label{crazyBex1}
\end{figure}

\begin{figure}[H]
$$\begin{tikzpicture}[scale=0.6]
	\def\Node{\node [circle, fill, inner sep=2pt]}
    \Node (1) at (1,0){};
    \Node (2) at (2,0){};
    \Node (3) at (3,0){};
    \Node (4) at (4,0){};
    \Node [blue] (5) at (5,0){};
    \Node [blue] (6) at (6,0){};
    \Node [blue] (7) at (7,0){};
    \Node [blue] (8) at (8,0){};
    \Node [blue] (9) at (9,0){};
    \Node [blue] (10) at (10,0){};
    \Node [red] (11) at (11,0){};
    \Node [red] (12) at (12,0){};
    
    \draw [line width=0.45mm, blue] (2) to[bend left=60](6);
    \draw [line width=0.45mm, blue] (3) to[bend left=60](5);
    \draw [line width=0.45mm, blue] (7) to[bend left=60](9);
    
    \draw [line width=0.3mm] (1) to[bend right=60](4);
    \draw [line width=0.45mm, blue] (2) to[bend right=60](3);
    
    \Node [red] at (13,0){};  
    
    \Node (100) at (25,0){};
    \Node (200) at (24,0){};
    \Node (300) at (23,0){};
    \Node (400) at (22,0){};
    \Node [blue] (500) at (21,0){};
    \Node [blue] (600) at (20,0){};
    \Node [blue] (700) at (19,0){};
    \Node [blue] (800) at (18,0){};
    \Node [blue] (900) at (17,0){};
    \Node [blue] (1000) at (16,0){};
    \Node [red] (1100) at (15,0){};
    \Node [red] (1200) at (14,0){};
    
    \draw [line width=0.45mm, red] (11) to[bend left=60] (1100);
    \draw [line width=0.45mm, red] (1100) to[bend left=60] (11);
    \draw [line width=0.45mm, red] (12) to[bend left=60] (1200);
    \draw [line width=0.45mm, red] (1200) to[bend left=60] (12);
        
    \draw [line width=0.45mm, blue] (200) to[bend right=60](600);
    \draw [line width=0.45mm, blue] (300) to[bend right=60](500);
    \draw [line width=0.45mm, blue] (700) to[bend right=60](900);
    
    \draw [line width=0.3mm] (100) to[bend left=60](400);
    \draw [line width=0.45mm, blue] (200) to[bend left=60](300);
\end{tikzpicture}$$
\caption{Meander $M$ associated with $\p_{12}^B\frac{1|5|3|1}{4}$}
\label{crazyBex2}
\end{figure}
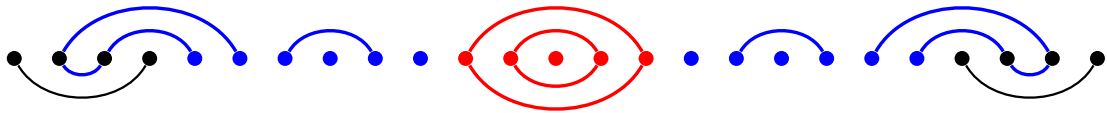
A functional described by Theorem \ref{typebfunctional} is 
\begin{align*}
    F=&e_{1,4}^*+e_{2,3}^*+e_{6,2}^*+e_{5,3}^*+e_{7,7}^*+e_{5,20}^*+e_{7,17}^*+e_{8,13}^*+e_{10,13}^*+e_{11,11}^*+e_{12,11}^*+e_{11,12}^*+e_{8,16}^*
\end{align*}
The seaweed $\g$ is illustrated in Figure \ref{crazyBex}, where the indices in the tail and aftertail components are colored blue and red, respectively, and a black dot is placed over each index in $\mathscr{I}_{F}$. Lines have been added to emphasize the core components of $\g$.

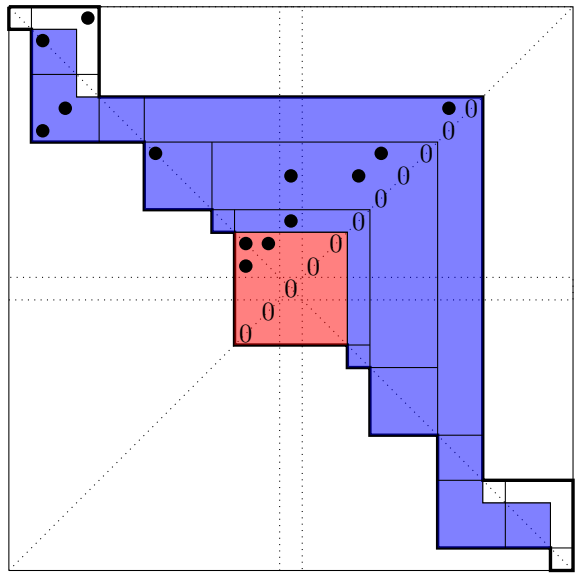
\begin{figure}[H]
$$\begin{tikzpicture}[scale=0.3]
	\def\Node{\node [circle, fill, inner sep=1.8pt]}
    \draw (0,0)--(0,25)--(25,25)--(25,0)--(0,0);
    \draw [dotted] (0,12)--(25,12)--(25,13)--(0,13);
    \draw [dotted] (12,0)--(12,25)--(13,25)--(13,0);
    \draw [dotted] (0,0)--(25,25)--(0,25)--(25,0);
    \draw [line width=0.45mm] (0,25)--(4,25)--(4,21)--(21,21)--(21,4)--(25,4)--(25,0)--(24,0)--(24,1)--(19,1)--(19,6)--(16,6)--(16,9)--(15,9)--(15,10)--(10,10)--(10,15)--(9,15)--(9,16)--(6,16)--(6,19)--(1,19)--(1,24)--(0,24)--(0,25);
    \draw [fill=blue, opacity=0.5] (1,24)--(3,24)--(3,21)--(21,21)--(21,3)--(24,3)--(24,1)--(19,1)--(19,6)--(16,6)--(16,9)--(15,9)--(15,15)--(9,15)--(9,16)--(6,16)--(6,19)--(1,19)--(1,24);
    \draw [fill=red, opacity=0.5](10,10)--(15,10)--(15,15)--(10,15)--(10,10);
    \node at (10.5,10.5){0};
    \node at (11.5,11.5){0};
    \node at (12.5,12.5){0};
    \node at (13.5,13.5){0};
    \node at (14.5,14.5){0};
    \node at (15.5,15.5){0};
    \node at (16.5,16.5){0};
    \node at (17.5,17.5){0};
    \node at (18.5,18.5){0};
    \node at (19.5,19.5){0};
    \node at (20.5,20.5){0};
    \draw (1,25)--(1,24)--(3,24)--(3,22)--(1,22);
    \draw (4,22)--(3,22)--(3,21)--(4,21)--(4,19)--(6,19)--(6,21);
    \draw (6,19)--(19,19)--(19,6);
    \draw (25,1)--(24,1)--(24,3)--(22,3)--(22,1);
    \draw (22,4)--(22,3)--(21,3)--(21,4)--(19,4)--(19,6)--(21,6);
    \draw (9,19)--(9,16)--(16,16)--(16,9)--(19,9);
    \draw (10,16)--(10,15)--(15,15)--(15,10)--(16,10);
    \Node at (3.5,24.5){};
    \Node at (1.5,23.5){};
    \Node at (1.5,19.5){};
    \Node at (2.5,20.5){};
    \Node at (6.5,18.5){};
    \Node at (19.5,20.5){};
    \Node at (16.5,18.5){};
    \Node at (12.5,17.5){};
    \Node at (12.5,15.5){};
    \Node at (10.5,14.5){};
    \Node at (10.5,13.5){};
    \Node at (11.5,14.5){};
    \Node at (15.5,17.5){};
\end{tikzpicture}$$
\caption{Indices in $\mathscr{I}_{F}$ on $\p_{12}^B\frac{1|5|3|1}{4}$}
\label{crazyBex}
\end{figure}

Direct computation yields that $B\oplus (0) \oplus \left(-B^{\widehat{t}}\right)$ is relations matrix for $\ker(B_{F})$, with
$$B=(b_1)\oplus\left(\begin{array}{cc}b_2&0\\0&-b_2\end{array}\right)\oplus(b_1)\oplus\left(\begin{array}{cc}-b_2&0\\0&b_2\end{array}\right)\oplus\left(\begin{array}{ccc}b_3&0&0\\0&0&0\\0&0&-b_3\end{array}\right)\oplus(0)\oplus\left(\begin{array}{cc}b_4+b_5&b_4\\b_4&b_5\end{array}\right).$$
\end{ex}
\end{subsection}

\begin{section}{Remark}
Naming explicit regular functionals on seaweed subalgebras of $D_n$ requires a different framework from that established in section \ref{Type B}.  It is possible that this may result in $\ker(B_F)$ having nonzero entries outside the core of the algebra. The difficulty arises from the forced zeroes on the antidiagonal of the elements of these matrix algebras. In general, these antidiagional 0's preclude the use of the functionals $e_{i,2n+1-i}^*$ which the framework in section \ref{type C} requires. In type-$B$ seaweeds, it was a natural first guess to simply ``move" the excluded functional $e_{i,j}^*$ to the center column (i.e., use instead the functional $e_{i,n+1}^*$).  This was not the full extent of the solution, as discussed in section \ref{Type B}, as a more recursive approach was necessary, but it was the base step for the inductive proof.  Unfortunately, there is no such ``natural guess'' in type-$D$.  Further, if one uses the more general basis-free definition of a biparabolic algebra, the situation for type-$D$ seaweeds is further complicated by their occasional lack of seaweed ``shape", when represented in matrix form (see \textbf{[2]}).

In a forthcoming articles, a frameworks for the development of regular functionals on seaweeds of type-$D$ will be presented. 
\end{section}

\vspace{5em}

\noindent
\Large{\textbf{References}}
\normalsize{}
\begin{itemize}
    \item [[1\hspace{-0.5em}]]  I. D. Ado.  The representation of Lie algebras by matrices. \textit{Uspekhi Matem. Nawk, (N.S.),} 2:169-173, 1947.
    \item [[2\hspace{-0.5em}]] A. Cameron.  Combinatorial index formulas for Lie algebras of seaweed type, Ph.D thesis, Lehigh University. \textit{Manuscript}, 2019.
    \item [[3\hspace{-0.5em}]] V. Coll and A. Dougherty.  On Kostant's cascade and regular functionals on seaweed Lie algebras. \textit{Manuscript}, 2019.
    \item [[4\hspace{-0.5em}]] V. Coll, A. Dougherty, M. Hyatt, and N. Mayers.  Meander graphs and Frobenius seaweed Lie algebras III. \textit{Journal of Generalized Lie Theory and Applications}, 11(2), 2017.
    \item [[5\hspace{-0.5em}]] V. Coll, M. Hyatt, and C. Magnant.  Symplectic meanders.  \textit{Communications in Algebra}, 45(11):4717-4729, 2017.
    \item [[6\hspace{-0.5em}]] V. Coll, M. Hyatt, C. Magnant, and H. Want.  Meander graphs and Frobenius seaweed Lie algebras II.  \textit{Journal of Generalized Lie Theory and Applications}, 9(1), 2015.
    \item [[7\hspace{-0.5em}]] V. Coll, C. Magnant, and H. Want.  The signature of a meander.  \textit{arXiv:1206.2705}, 2012.
    \item [[8\hspace{-0.5em}]] V. Dergachev and A. Kirillov.  Index of Lie algebras of seaweed type.  \textit{Journal of Lie Theory}, 10(2):331-343, 2000.
    \item [[9\hspace{-0.5em}]] J. Dixmier.  \textit{Algebres Enveloppantes}.  Gauthier-Villars, 1974.
    \item [[10\hspace{-0.5em}]] A. Dougherty.  Regular functionals on seaweed Lie algebras, Ph.D thesis, Lehigh university. \textit{Preprint}, 2019.
    \item [[11\hspace{-0.5em}]] D. I. Panyushev.  An extension of Rais' theorem and seaweed subalgebras of simple Lie algebras.  \textit{Annales de l'institut Fourier}, 55(3):693-715, 2005.
    \item [[12\hspace{-0.5em}]] A. Joseph.  The minimal orbit in a simple Lie algebra and its associated maximal ideal. \textit{Annales scientifiques de l'Ecole Normale Superieure}, Ser. 4, 9(1):1-29, 1976.
    \item [[13\hspace{-0.5em}]] A. Joseph.  The integrality of an adapted pair.  \textit{Transformation Groups}, 20(3):771-816, 2015.
    \item [[14\hspace{-0.5em}]] B. Kostant.  The cascade of orthogonal roots and the coadjoint structure of the nilradical of a borel subgroup of a semisimple Lie group.  \textit{Moscow Mathematical Journal}, 12(3):605-620, 2012.
    \item [[15\hspace{-0.5em}]] A. I. Ooms.  On Lie algebras having a primitive universal enveloping algebra.  \textit{J. Algebra}, 32(3):488-500, 1974.
    \item [[16\hspace{-0.5em}]] D. Panyushev.  Inductive formulas for the index of seaweed Lie algebras.  \textit{Moscow Mathematical Journal}, 1(2):221-241, 2001.
    \item [[17\hspace{-0.5em}]] D. Panyushev and O. Yakimova.  On seaweed subalgebras and meander graphs in type C. \textit{Pacific Journal of Mathematics}, 285(2):485-499, 2016.
    \item [[18\hspace{-0.5em}]] P. Tauvel and R. W. Yu. Indice et formes lineaires stables dans les algebres de Lie.  \textit{Journal of Algebra}, 273:507-516, 2004.
    \item [[19\hspace{-0.5em}]] J. Tits.  Une remarque sur la structure des algebres de Lie semi-simple complexes.  \textit{Proc. Konninkl. ederl. Akad. Wetenschappen, Series A}, 63:48-53, 1960.
\end{itemize}

\section{Appendix A}
Recall from Theorem \ref{My Theorem 1} that the indices in $\mathscr{I}_{F_n}$ are illustrated in Figure \ref{F_n inds A} as the grey region and solid lines.
\begin{figure}[H]
$$\begin{tikzpicture}[scale=0.6]
    \filldraw[draw=black, fill=gray!20, line width=0.45mm] (-0.5,-0.5)--(4.5,4.5)--(-0.5,4.5)--(-0.5,-0.5);
    
    \draw [thick] (-0.6,-0.6) to[bend left=10] (-0.6,4.6);
    \draw [thick] (4.6,-0.6) to[bend right=10] (4.6,4.6);    
\end{tikzpicture}$$
\caption{Indices in $\mathscr{I}_{F_n}$}
\label{F_n inds A}
\end{figure}
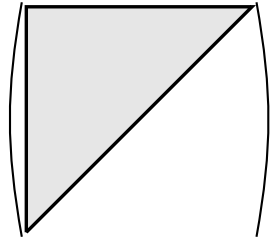

\begin{proof}[Proof of Theorem \ref{My Theorem 1}]
Let $B=[b_{i,j}]$ be a relations matrix of $\ker(B_{F_n})$. It follows from Theorem \ref{Index Cor} that the minimum dimension of $\ker(B_F)$ over all $F\in\g^*$ is $n$. The Theorem follows from the verification of the following two claims:

\bigskip 

\textbf{Claim 1:} For each $(i,j)\in (n-1)\times(n-1)$, $b_{i,j}=\sum_{s=1}^n c_sb_{s,n}+\sum_{s=1}^{n-1}c_s'b_{n,s}$ for suitable coefficients $c_s\in\C$,

\bigskip

\noindent
and 

\bigskip

\textbf{Claim 2:} $b_{n,s}=b_{s,n}$, for all $s\in[1,n]$.

\bigskip

To understand why these claims are sufficient, we argue as follows: Claim 1 will define the top $(n-1)\times(n-1)$ matrix in terms of the elements in the last row/column of $B$, determining that there are at most $2n-1$ degrees of freedom in $B$. Claim 2 will then establish that there are exactly (as it cannot possibly be smaller due to minimality of the index) $n$ degrees of freedom in these $2n-1$ positions. 

\bigskip 

\textbf{\underline{Proof of Claim 1:}}
\bigskip

The proof is by induction. We proceed according to the following steps listed and verified below.
\begin{enumerate}
    \item An application of Lemma \ref{Symmetry Lemma} (``symmetry lemma") halves the work by allowing us to only consider indices $(i,j)$ illustrated in Figure \ref{Claim1Proof}.
    \item The system of equations $F([B,e_{i,j}])=0$ is developed explicitly, along with two formulas which will be needed in the inductive step.
    \item For the base case, we show that the first and last row are explicit sums of elements $b_{n,s}$. Proceeding by induction on pairs of rows (first and last) moving towards the center of $B$ in the halved domain, we show that for $i\in\left[1,\lceil\frac{n}{2}\rceil\right]$, any elements $b_{i,j}$ and $b_{n+1-i,j}$ can be defined in terms of the previous row defined. More specifically, $b_{i,j}=b_{i-1,j-1}+b_{n,s}$, for some $s\in[1,n]$ and $b_{n+1-i,j}=b_{n+2-i,j-1}+b_{n,r}$, for some $r\in[1,n]$. 
\end{enumerate}

\bigskip
\textbf{Step 1:} For ease of notation, let $b_s=b_{n,s}$ and $b_s'=b_{s,n}$, for all $s\in[1,n]$ -- note that $b_n=b_n'$. By invoking Lemma \ref{Symmetry Lemma} (and making use of our convenient choice for $b_s$ and $b_s'$ being symmetric across the diagonal), it suffices to show the claim for all elements $b_{i,j}$ such that 

\begin{align*}
    (i,j)\in\mathscr{I}&=\left\{(i,j)\;\Bigg|\; i\in \left[1,\left\lceil\frac{n}{2}\right\rceil\right],\;j\in[i,n+1-i]\right\}\;\cup\;\left\{(i,j)\;\Bigg|\;i\in\left(\left\lceil\frac{n}{2}\right\rceil,n\right],\;j\in(n+1-i,i]\right\}.
\end{align*}

\noindent
The indices in $\mathscr{I}$ are illustrated in Figure \ref{Claim1Proof} as the grey regions and solid lines.
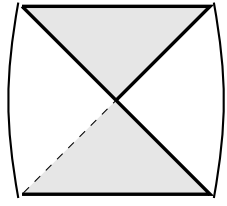
\begin{figure}[H]
$$\begin{tikzpicture}[scale=0.5]
    \draw [thick] (-0.6,-0.6) to[bend left=10] (-0.6,4.6);
    \draw [thick] (4.6,-0.6) to[bend right=10] (4.6,4.6);
    
    \filldraw[line width=0.45mm, draw=black, fill=gray!20] (-0.5,4.5)--(4.5,4.5)--(2,2)--(-0.5,4.5);
    \filldraw[dashed, draw=black,fill=gray!20] (-0.5,-0.5)--(4.5,-0.5)--(2,2)--(-0.5,-0.5);
    \draw [line width=0.45mm] (-0.5,-0.5)--(4.5,-0.5)--(2,2);
\end{tikzpicture}$$
\caption{Indices in $\mathscr{I}$}
\label{Claim1Proof}
\end{figure}
We will define $b_{i,j}$ in terms of elements $b_s$ over all $(i,j)\in\mathscr{I}$, and it will follow that every $(i,j)\not\in\mathscr{I}$ (and all $(i,i)$ on the diagonal) are defined in terms of elements $b_s'$.

\bigskip 
\textbf{Step 2:} To begin, observe that

$$\mathscr{I}_{F_n}=\{(i,j)\;|\;1\leq i\leq n,\;\;1\leq j\leq n+1-i\}=\{(i,j)\;|\;1\leq j\leq n,\;\;1\leq i\leq n+1-j\}$$

\noindent 
and refer to Lemma \ref{Syst of Eqs} to see that $B$ must satisfy $n^2$ conditions of the form
\begin{equation} 
\label{eqs}
\sum_{s=1}^{n+1-j}b_{s,i}=\sum_{s=1}^{n+1-i}b_{j,s}
\end{equation}
over $(i,j)\in \mathscr{I}_{\gl(n)}$. There are no additional requirements on $B$ as there are no forced zeroes in $\gl(n)$. For all $i\in[2,n-1]$ consider  applying equation (\ref{eqs}) as follows:

$$b_{i,j}=\sum_{s=1}^ib_{s,j}-\sum_{s=1}^{i-1}b_{s,j}=\sum_{s=1}^{n+1-(n+1-i)}b_{s,j}-\sum_{s=1}^{n+1-(n+1-(i-1))}b_{s,j}=\sum_{s=1}^{n+1-j}b_{n+1-i,s}-\sum_{s=1}^{n+1-j}b_{n+1-(i-1),s}.$$

\noindent 
This yields the following formula:
\vspace{-2em}
\begin{center}
\begin{equation}
\label{eq1}
b_{n+2-i,n+1-j}=b_{i,j+1}+b_{n+1-i,n+1-j}-b_{i,j}.
\end{equation}
\end{center}
By now expressing $b_{i,j}$ as $\sum_{s=1}^jb_{i,s}-\sum_{s=1}^{j-1}b_{i,s}$ and applying equation (\ref{eqs}), we get a second formula:
\vspace{-2em}
\begin{center}
\begin{equation}
\label{eq2}
b_{n+1-i,n+1-j}=b_{i,j}+b_{n+2-i,n+1-j}-b_{i,j+1}.
\end{equation}
\end{center}

\bigskip 
\textbf{Step 3:} We proceed by induction. The base of the induction will be defining the first and last rows of $B$ in terms of elements $b_s$. From there, assuming we have defined $b_{i,j}$ appropriately for all $(i,j)\in\mathscr{I}$ with $i\in[1,I]\cup[n+1-I,n]$ (i.e., the first and last $I$ rows of $B$), we will define $b_{I+1,j}$ and $b_{n+1-(I+1),j}$ for $(I+1,j),(n+1-(I+1),j)\in\mathscr{I}$ in terms of elements $b_s$. The last row is already filled by $b_{n,i}=b_i$ for $i\in[1,n]$. The first row comes from equation (\ref{eqs}) evaluated for $j=n$: \begin{equation} 
\label{first row}
b_{1,i}=\sum_{s=1}^{n+1-n}b_{s,i}=\sum_{s=1}^{n+1-i}b_{n,s}=\sum_{s=1}^{n+1-i}b_s.
\end{equation}
This establishes the base. 

Now, for the induction hypothesis, assume that for some $I\in[1,\lfloor\frac{n}{2}\rfloor)$, $b_{i,j}$ and $b_{n+1-i,j}$ are defined in terms of elements $b_s$, for all indices $(i,j),(n+1-i,j)\in\mathscr{I}$ with $i\leq I$ (some care is needed if $I=\lfloor\frac{n}{2}\rfloor$ -- we handle this separately for Equations (\ref{eq3}) and (\ref{eq4}), depending on whether $n$ is even or odd). We assert the following about the indices $(i,j)\in\mathscr{I}$:
\vspace{-2em}
\begin{center}
\begin{equation}
\label{eq3}
b_{i,j}=b_{i-1,j-1}-b_{n+3-i-j},
\end{equation}
\end{center}
for $(i,j)\in\mathscr{I}$ such that $i\in\left[1,\left\lfloor\frac{n}{2}\right\rfloor\right]$, if $n$ is odd, and $i\in\left[1,\frac{n}{2}+1\right]$, if $n$ is even,
\vspace{-2em}
\begin{center}
\begin{equation}
\label{eq4}
b_{n+1-i,j}=b_{n+2-i,j+1}+b_{j-i+1},
\end{equation}
\end{center}
for $(i,j)\in\mathscr{I}$ such that $i\in\left[1,\left\lfloor\frac{n}{2}\right\rfloor+1\right]$, if $n$ is odd, and $i\in\left[1,\frac{n}{2}\right]$, if $n$ is even. The need for the domain restrictions is due to the necessary conditions that $n+3-i-j>0$ and $j-i+1>0$. Equations (\ref{eq3}) and (\ref{eq4}) recursively define the entries of $B$ for indices in $\mathscr{I}$ in terms of elements $b_s$.

Now, for the induction, assume that Equations (\ref{eq3}) and (\ref{eq4}) are true for all $b_{i,j}$ and $b_{n+1-i,j}$ with $i\in[2,I]$ and $(i,j),(n+1-i,j)\in\mathscr{I}$. Consider $b_{I+1,j}$ and $b_{n+1-(I+1),j}$. Let $i=n+1-I$ and $s=n+1-j$. We invoke equation (\ref{eq1}) twice to yield:

\begin{align*}
b_{I+1,j}&=b_{n+2-i,n+1-s}=b_{i,s+1}+b_{n+1-i,n+1-s}-b_{i,s}=b_{n+1-I,s+1}+b_{I,n+1-s}-b_{n+1-I,s},
\end{align*}

\vspace{-1em}
\noindent 
and 

\begin{align*}
b_{I,j-1}&=b_{n+2-(i+1),n+1-(s+1)}=b_{i+1,s+2}+b_{n+1-(i+1),n+1-(s+1)}-b_{i+1,s+1}=b_{n+2-I,s+2}+b_{I-1,n-s}-b_{n+2-I,s+1}.
\end{align*}

\noindent 
Therefore, by the induction hypotheses on $(I,s+1)$, $(I,n+1-s)$, and $(I,s)$, we get

\begin{align*}
b_{I,j-1}-b_{I+1,j}&=(b_{n+2-I,s+2}-b_{n+1-I,s+1})+(b_{I-1,n-s}-b_{I,n+1-s})-(b_{n+2-I,s+1}-b_{n+1-I,s})\\
&=-b_{s+2-I}+b_{n+3-I-(n+1-s)}+b_{s+1-I}\\
&=-b_{s+2-I}+b_{s+2-I}+b_{s+1-I}\\
&=b_{s+1-I}\\
&=b_{n+3-j-(I+1)}.
\end{align*}

\noindent 
The foregoing equation establishes (\ref{eq3}) on $(I+1,j)$. In a similar fashion, we invoke (\ref{eq2}) twice to yield:

\begin{align*}
b_{n+1-(I+1),j}&=b_{n+1-(I+1),n+1-s}=b_{I+1,s}+b_{n+2-(I+1),n+1-s}-b_{I+1,s+1}=b_{I+1,s}+b_{n+1-I,n+1-s}-b_{I+1,s+1},
\end{align*}

\noindent 
and

\begin{align*}
b_{n+2-(I+1),j+1}&=b_{n+1-I,n+1-(s-1)}=b_{I,s-1}+b_{n+2-I,n+1-(s-1)}-b_{I,s}=b_{I,s-1}+b_{n+2-I,n+2-s}-b_{I,s}.
\end{align*}

\noindent 
By the inductive hypotheses and equation (\ref{eq3}) on $(I+1,s)$, $(I,n+1-s)$, and $(I+1,s+1)$, we have
\vspace{-1em}

\begin{align*}
b_{n+1-(I+1),j}-b_{n+2-(I+1),j+1}&=(b_{I+1,s}-b_{I,s-1})+(b_{n+1-I,n+1-s}-b_{n+2-I,n+2-s})-(b_{I+1,s+1}-b_{I,s})\\
&=-b_{n+3-(I+1)-s}+b_{n+1-s-I+1}+b_{n+3-(I+1)-(s+1)}\\
&=-b_{n+2-I-s}+b_{n+2-I-s}+b_{n+1-I-s}\\
&=b_{n+1-I-s}\\
&=b_{n+1-I-(n+1-j)}\\
&=b_{j-(I+1)+1}.
\end{align*}

\noindent 
The foregoing equation establishes (\ref{eq4}) on $(I+1,j)$. Therefore, to complete Claim 1, it suffices show that equations (\ref{eq3}) and (\ref{eq4}) hold as a relation between the indices of $\mathscr{I}$ for $I=2$.

Recall from equation (\ref{first row}) that $b_{1,j}=\sum_{s=1}^{n+1-j}b_s$ holds for all $j\in[1,n]$. By (\ref{eq1}), we get 

\begin{align*} 
b_{2,j}&=b_{n,n+2-j}+b_{1,j}-b_{n,n+1-j}\\
&=\left(\sum_{s=1}^{n+2-j}b_s\right)-b_{n+1-j}\\
&=\left(\sum_{s=1}^{n+2-j}b_s\right)-b_{n+2-j}+b_{n+2-j}-b_{n+1-j}\\
&=b_{1,j-1}+b_{n+1-j}.
\end{align*}

\noindent 
The foregoing equation verifies (\ref{eq3}) for $I=2$. Now, by (\ref{eq2}) and the equation justified immediately above,

\begin{align*}
b_{n-1,j}&=b_{2,n+1-j}+b_{n,j}-b_{2,n+2-j}\\
&=(b_{1,n-j}-b_j)+b_j-(b_{1,n+1-j}-b_{j-1})\\
&=\left(\sum_{s=1}^{n+1-(n-j)}b_s\right)-\left(\sum_{s=1}^{n+1-(n+1-j)}b_s\right)+b_{j-1}\\
&=\left(\sum_{s=1}^{j+1}b_s\right)-\left(\sum_{s=1}^jb_s\right)+b_{j-1}\\
&=b_{j+1}+b_{j-1}\\
&=b_{n,j+1}+b_{(j-2)+1}.
\end{align*}

\noindent 
This establishes Claim 1.

\bigskip 

\textbf{\underline{Proof of Claim 2:}} 
\bigskip

The proof is by two separate inductions along the main diagonal of $B$. By Claim 1, every element $b_{i,i}$ is defined in terms of elements $b_{n,s}$. By Lemma \ref{Symmetry Lemma}, we know that every element $b_{i,i}$ is also defined in terms of elements $b_{s,n}$. The proofs proceed by equating these two formulas for $b_{i,i}$. The first induction will move from the lower right entry of $B$ up to the center of the matrix. At that time, the second induction picks up with an appropriate handoff of the indices depending on whether $n$ was even or odd to move up the rest of the diagonal. Assume that $n$ is even -- the argument for $n$ odd simply requires an appropriate adjustment of the even case. The proof of Claim 2 rests on the establishment of the following two subclaims, both of which are established by (yet another) induction.

\bigskip

\textbf{Claim 2.1:} For $k\in\left[0,\frac{n}{2}-1\right]$,  $b_{n-k,n-k}=\sum_{s=0}^kb_{n-2s}\text{ and }b_{n-2k}=b_{n-2k}',$

\bigskip 

\textbf{Claim 2.2:} For $k\in\left[0,\frac{n}{2}-1\right]$,  $b_{\frac{n}{2}-k,\frac{n}{2}-k}=b_{\frac{n}{2}+1,\frac{n}{2}+1}+\sum_{s=0}^k b_{1+2s}\text{ and }b_{1+2k}=b_{1+2k}'$.

\bigskip 
\noindent
\textbf{Proof of Claim 2.1:} 
The proof is by induction. Trivially, $b_{n,n}=b_n=b_n'$. Further, by Formula (\ref{eq4}),

$$b_{n-1,n-1}=b_{n,n}+b_{n-2}.$$

\noindent 
Now, by Lemma \ref{Symmetry Lemma} we have

$$b_n'+b_{n-2}'=b_n+b_{n-2}\hspace{2em}\Rightarrow\hspace{2em}b_{n-2}=b_{n-2}'.$$

\noindent 
 Now, fix $k<\frac{n}{2}-1$ and assume for all $K\leq k$ Claim 2.1 holds. By Formula (\ref{eq4}), we have
 
 $$b_{n-(k+1),n-(k+1)}=b_{n-k,n-k}+b_{n-2k-2}=\sum_{s=0}^kb_{n-2s}+b_{n-2(k+1)}$$

\noindent 
 (this is seen using $i=k+2$, $j=n-(k+1)$ in Formula (\ref{eq4})). By Lemma \ref{Symmetry Lemma}, we have
 
 $$\sum_{s=0}^kb_{n-2s}+b_{n-2(k+1)}=\sum_{s=0}^kb_{n-2s}'+b_{n-2(k+1)}'.$$

\noindent 
 By induction, $b_{n-2(k+1)}=b_{n-2(k+1)}'$ and $b_{n-(k+1),n-(k+1)}=\sum_{s=0}^{k+1}b_{n-2s}$. This proves Claim 2.1. \hfill $\blacksquare$

\bigskip 
\noindent
\textbf{Proof of Claim 2.2:} Claim 2.2 will yield that for all odd indices $s$, $b_s=b_s'$ which is the second half of Claim 2 provided $n$ is even. The proof is by induction. 

Since $n$ is even, by Formula (\ref{eq3}) and the fact that $n-(\frac{n}{2}-1)=\frac{n}{2}+1$, we can see that 

$$b_{\frac{n}{2},\frac{n}{2}}=b_{\frac{n}{2}+1,\frac{n}{2}+1}+b_{n+3-(\frac{n}{2}+1)-(\frac{n}{2}+1)}=b_{\frac{n}{2}+1,\frac{n}{2}+1}+b_1.$$ 

\noindent 
By Lemma \ref{Symmetry Lemma} and Claim 2.1, $b_1=b_1'$. Now, assume for some $k<\frac{n}{2}-1$ we have that Claim 2.2  holds for all $K\leq k$. Then by Formula (\ref{eq3}), we have 

$$b_{\frac{n}{2}-(k+1),\frac{n}{2}-(k+1)}=b_{\frac{n}{2}-k,\frac{n}{2}-k}+b_{3+2k}=\left(b_{\frac{n}{2}+1,\frac{n}{2}+1}+\sum_{s=0}^kb_{1+2s}\right)+b_{1+2(k+1)}.$$

\noindent 
By Lemma \ref{Symmetry Lemma}, $b_{1+2(k+1)}=b_{1+2(k+1)}'$ and $b_{\frac{n}{2}-(k+1),\frac{n}{2}-(k+1)}=b_{\frac{n}{2}+1,\frac{n}{2}+1}+\sum_{s=0}^{k+1}b_{1+2s}$, as desired. This proves Claim 2.2. \hfill $\blacksquare$

This completes the proof of Claim 2, and thus the proof of Theorem \ref{My Theorem 1}.
\end{proof}

\end{document}